\documentclass[reqno]{amsart}
\usepackage{amsthm, amscd, amsfonts,graphicx}
\usepackage{color}
\usepackage{enumerate}
\usepackage{verbatim}
\numberwithin{equation}{section}
\theoremstyle{plain}
 \newtheorem{theorem}{Theorem}[section]
 \newtheorem{lemma}[theorem]{Lemma}

 \newtheorem{corollary}[theorem]{Corollary}
\theoremstyle{definition}
 \newtheorem{definition}[theorem]{Definition}
 \newtheorem{remark}[theorem]{Remark}
 \newtheorem{prooftechnique}[theorem]{Proof Technique}
 \newtheorem{problem}[theorem]{Problem}
\theoremstyle{remark}
 
 \newtheorem{example}[theorem]{Example}
 \newtheorem{terminology}[theorem]{Terminology}
 \newtheorem*{preparation}{Lattice theoretical preparatory part}
 \newtheorem*{computation}{Computational part}
\newenvironment{enumeratei}{\begin{enumerate}[\quad\upshape (i)]} {\end{enumerate}}
\newcommand \alg [1] {\mathcal #1}
\newcommand \dom [1] {\textup{Dom}(#1)}
\newcommand \many{\boldsymbol\sigma}
\newcommand \smany{$\many$-many}
\newcommand \krlist {\mathcal L_{\textup{KR}}}
\newcommand \dual [1] {#1^{\boldsymbol\delta}}
\newcommand \azeset[1] {(C#1)}
\newcommand \eset[1] {\par \azeset{#1}}
\newcommand \aceset[1] {#1}
\newcommand \ceset[1] {\par \aceset{#1}}
\newcommand \vonal {\noalign{\hrule}}
\newcommand \fem {F_1^-}

\DeclareMathOperator \Sub {Sub}

\newcommand\set [1]{\{#1\}}

\renewcommand \phi {\varphi}
\newcommand \url [1] {{\texttt{#1}}}

\newcommand \red [1] {{\color{red}#1\color{black}}}

\newcommand \nothing [1] {}

\newcommand \csakau [1] {}
\begin{document}
\title[Eighty-three sublattices]
{Eighty-three sublattices and planarity}

\author[G.\ Cz\'edli]{G\'abor Cz\'edli}
\address{University of Szeged, Bolyai Institute, Szeged,
Aradi v\'ertan\'uk tere 1, Hungary 6720}
\email{czedli@math.u-szeged.hu}
\urladdr{http://www.math.u-szeged.hu/~czedli/}

\thanks{This research of was supported by the Hungarian Research, Development and Innovation Office under grant number KH 126581.}



\dedicatory{Dedicated to professor George A.\ Gr\"atzer on his eighty-third birthday}

\begin{abstract} Let $L$ be a finite $n$-element lattice. We prove that if $L$ has at least $83\cdot 2^{n-8}$ sublattices, then $L$ is planar. For $n>8$, this result is sharp since  there is a non-planar lattice with exactly $83\cdot 2^{n-8}-1$ sublattices.
\end{abstract}

\date{\red{\hfill  July 1, 2019}}

\maketitle

\section{Our result  and introduction} 
A finite lattice is said to be \emph{planar} if it has a Hasse diagram that is also a planar representation of a graph. Our goal is to prove that finite lattices with many sublattices are planar. Namely, we are going to prove the following theorem.

\begin{theorem}\label{thmmain} Let $L$ be a finite lattice, and let  $n:=|L|$ denote the number of its elements. If $L$ has at least $83\cdot 2^{n-8}$ sublattices, then it is a planar lattice.  
\end{theorem}

Another variant of this result together with a comment on its sharpness will be stated in Theorem~\ref{thmreform}.

\subsection*{Notes on the dedication}
As a coincidence, the number \emph{eighty-three} plays a key role in Theorem~\ref{thmmain}, and I found this theorem recently, in the same year when professor George Gr\"atzer, the founder of Algebra Universalis, celebrates his \emph{eighty-third} birthday. For more about him, the reader is referred to my biographic paper \cite{czgonGG} and the interview \cite{czggginterview} with him.  Furthermore, the topic of the present paper is close to his current research interest on planar lattices; this interest has been witnessed, say, by 
Cz\'edli and Gr\"atzer~ \cite{czgggchapter} and \cite{czgggresection},
Cz\'edli, Gr\"atzer, and Lakser~\cite{czggglakser},
Gr\"atzer \cite{ggnotes10}, \cite{ggsection14}, \cite{ggONczg}, \cite{ggswing15}, \cite{ggforkcon16}, and \cite{ggtrajcon18}, 
Gr\"atzer and Knapp~\cite{ggknapp1}, \cite{ggknapp2}, \cite{ggknappAU}, \cite{ggknapp3}, and \cite{ggknapp4}, Gr\"atzer and Lakser~\cite{gglakser92},  Gr\"atzer, Lakser, and Schmidt~\cite{gglakserscht}, Gr\"atzer and Quackenbush~\cite{ggqbush}, 
Gr\"atzer and Schmidt~\cite{ggscht14}, and  Gr\"atzer and Wares~\cite{ggwares}.
These facts motivate the dedication.

\begin{remark}\label{remarkmantissa}
Although $41.5\cdot 2^{n-7}$, $20.75\cdot 2^{n-6}$, $10.375\cdot 2^{n-5}$, \dots{} and  $166\cdot 2^{n-9}$, $332\cdot 2^{n-10}$, $664\cdot 2^{n-11}$, \dots{} are all equal to $83\cdot 2^{n-8}$, we want to avoid fractions as well as large coefficients of powers of 2. This explains the formulation of Theorem~\ref{thmmain}.
\end{remark}

\begin{remark}\label{remarksharp}
 For $n\geq9$, Theorem~\ref{thmmain} is sharp, since we will present an $n$-element non-planar lattice with exactly $83\cdot 2^{n-8}-1$ sublattices. For $n<9$, Theorem~\ref{thmmain} can easily be made sharp as follows. Whenever $n\leq 7$, every $n$-element lattice is planar, regardless the number of its sublattices. While the eight-element boolean lattice has exactly 73 sublattices, every eight-element lattice with at least $74=74\cdot 2^{8-8}$ sublattices is planar.  
\end{remark}

Next, we mention some earlier results that motivate the present paper. As a counterpart of Theorem~\ref{thmmain}, finite lattices with many \emph{congruences} are also planar; see Cz\'edli \cite{czglatmancplanar} for details. Finite lattices with ``very many" congruences or sublattices have been described by Ahmed and Horv\'ath~\cite{delbrineszter}, Cz\'edli~\cite{czgnotelatmanyC}, Cz\'edli and Horv\'ath~\cite{czgkhe}, and Mure\c san and Kulin~\cite{kulinmuresan}.  

\subsection*{Outline} The rest of the paper is devoted to the proof of Theorem~\ref{thmmain}. 
In Section~\ref{sectiontools}, we 
recall the main result of Kelly and Rival~\cite{kellyrival}; this deep result will be the main tool used in the paper. Some easy lemmas and  the proof of Remark~\ref{remarksharp}  are also presented, and we introduce a terminology that allows us to formulate Theorem~\ref{thmmain} in an equivalent and more convenient form; see Theorem~\ref{thmreform}.  
Also, this section points out some difficulties explaining why we do not see a computer-free way to prove  Theorem~\ref{thmreform} (equivalently, Theorem~\ref{thmmain}) and why a lot of human effort is needed in addition to the brutal force of computers. 
Section~\ref{sectionrestlemmas}
gives some more details of this computer-assisted effort but
the proofs of some lemmas stated there are available only from separate files or  from the appendices of the \csakau{extended version of the}%
paper. Also, Section~\ref{sectionrestlemmas} combines many of our lemmas and corollaries to complete the proof of  Theorem~\ref{thmreform} and, thus,  Theorem~\ref{thmmain}.

\section{Tools and difficulties}\label{sectiontools}
\subsection{Relative number of subuniverses}
Let $F$ be a set of binary operation symbols.
By a \emph{binary partial algebra} $\alg A$ of type $F$ we mean a structure 
$\alg A=(A; F_A)$ such that $A$ is a nonempty set, $F_A=\set{f_A: f\in F}$, and for each $f\in F$,  $f_A$ is a map from a subset $\dom {f_A}$ of $A^2$ to $A$. That is, $f_A$ is a binary \emph{partial operation} on $A$.
If $\dom {f_A}=A^2$ for all $f\in F$, then $\alg A$ is a \emph{binary algebra} (without the adjective ``partial'').  In particular, every lattice is a binary algebra; note that we write $\vee$ and $\wedge$ instead of $\vee_A$ and $\wedge_A$ when the meaning is clear from the context. 
A \emph{subuniverse} of  $\alg A$
is a subset $X$ of $A$ such that $X$ is closed with respect to all partial operations, that is, whenever $x,y\in X$, $f\in F$ and $(x,y)\in\dom {f_A}$, then $f_A(x,y)\in X$. The set of subuniverses of $\alg A$ will be denoted by $\Sub(\alg A)$. 
For a lattice $\alg L=(L;\set{\vee,\wedge})$, we will write $L$ rather than $\alg L$. Note that the number of sublattices of $L$ is $|\Sub(L)|-1$, since the set of sublattices of $L$ is $\Sub(L)\setminus\set{\emptyset}$. If $\alg B=(B,F_B)$ with $B\subseteq A$ is another binary partial algebra  of type $F$ such that $\dom{F_B}\subseteq B^2\cap \dom{F_A}$ for every $f\in F$  
and $f_B(x,y)=f_A(x,y)$ for all $(x,y)\in \dom{F_B}$, then $\alg B$ is said to be a \emph{weak partial subalgebra} of $\alg A$.

It is straightforward to drop the adjective ``binary'' from the concepts defined above. Even if this adjective is dropped in Lemma~\ref{lemmapartext}, to be stated soon, we will use this lemma only for the binary case. 
All lattices, posets, and partial algebras in this paper are automatically assumed to be \emph{finite} even if this is not repeated all the time. 

This paper is about lattices with \emph{many} sublattices. 
Large lattices have a lot of subuniverses and sublattices since every singleton subset of a lattice is a sublattice. So it is reasonable to define the meaning of ``many'' with the help of the following notation.

\begin{definition}
The \emph{relative number of subuniverses} of an $n$-element finite binary partial algebra $\alg A=(A,F_A)$ is defined to be and denoted by
\[
\many(\alg A):=|\Sub(\alg A)|\cdot 2^{8-n}.
\] 
Furthermore, we say that a finite lattice $L$ has \emph{\smany{} sublattices} or, in other words, it has \emph{\smany{}  subuniverses} if $\many(L)>83$.
\end{definition}

This concept and notation will play a crucial role in the rest of the paper. Since $|\Sub(L)|$ is larger than the number of sublattices by 1, we can reformulate Theorem~\ref{thmmain} and a part of Remark~\ref{remarksharp} as follows.

\begin{theorem}\label{thmreform}
If $L$ is a finite lattice such that $\many(L)>83$, then $L$ is planar. In other words, finite lattices with \smany{} sublattices are planar. Furthermore, for every natural number $n\geq 9$, there exists an $n$-element lattice $L$ such that $\many(L)=83$ and $L$ is not planar.
\end{theorem}

The importance of the concepts introduced in this section so far is well explained by the following easy lemma.

\begin{lemma}\label{lemmapartext}
If $\alg B=(B,F_B)$ is a weak partial subalgebra of a finite partial 
algebra  $\alg A=(A,F_A)$, then $\many(\alg A)\leq \many(\alg B)$.  
\end{lemma}

\begin{proof} Let $m:=|B|$ and $n:=|A|$. Then  $k:=n-m=|A\setminus B|\geq 0$. Define an equivalence relation $\sim$ on $\Sub(\alg A)$ by letting $X\sim Y$ mean that $X\cap B=Y\cap B$.  Since $X\cap B\in\Sub(\alg B)$ for every $X\in\Sub(\alg A)$, this equivalence has at most $|\Sub(\alg B)|$ blocks. Every block of $\sim$ is a subset of
$\set{U\cup X: X\subseteq  A\setminus B}$ for some $U\in\Sub(\alg B)$. Since $A\setminus B$ has $2^k$ subsets, 
every block of $\sim$ consists of at most $2^k$ elements of $\Sub(\alg A)$. Therefore, $ |\Sub(\alg A)| \leq |\Sub(\alg B)|\cdot 2^k$. Dividing this inequality by $ 2^{n-8} = 2^{m-8}\cdot 2^k$, we obtain the validity of the lemma.
\end{proof}

\begin{figure}[htb] 
\centerline
{\includegraphics[scale=1.0]{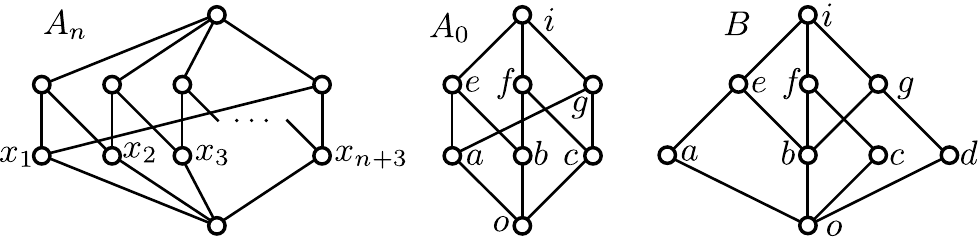}}
\caption{$A_n$, the boolean lattice $A_0$, and $B$
\label{figAB}}
\end{figure}%

\subsection{The Kelly--Rival list}
For a poset $P$, its dual will be denoted by $\dual P$.
With reference to Kelly and Rival~\cite{kellyrival} or, equivalently, to Figures~\ref{figAB}--\ref{figGH}, the \emph{Kelly--Rival list} of lattices is defined as follows.
\[
\krlist:=\set{A_n, E_n, \dual{E_n}, F_n, G_n, H_n: n\geq 0} \cup \set{B,\dual B, C, \dual C, D,\dual D}.
\]
Note that $A_n$, $F_n$, $G_n$, and $H_n$ are selfdual lattices. The key tool we need is the following deep result.

\begin{theorem}[{Kelly and Rival~\cite{kellyrival}}]\label{thmKellyRival}
A finite lattice is planar if and only if it does not contain any lattice in $\krlist$ as a subposet.
\end{theorem} 

\begin{figure}[htb] 
\centerline
{\includegraphics[scale=1.0]{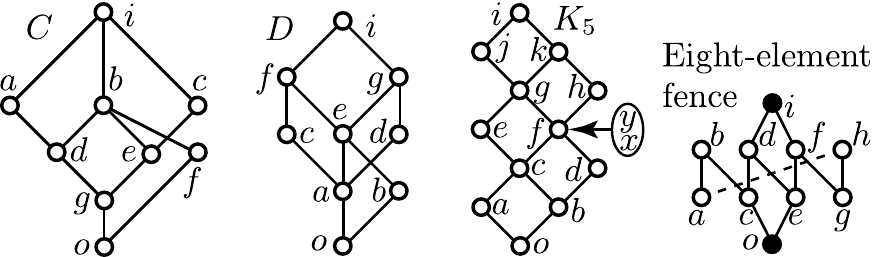}}
\caption{Lattices $C$, $D$, $K_5$, and the eight-element fence; disregard $x$ and $y$  in the oval, the black-filled elements, and the dashed line
\label{figCDK}}
\end{figure}%

Lemma~\ref{lemmapartext} and Theorem~\ref{thmKellyRival} raise the following problem; $B,C,\dots,H_0$ still denote lattices in $\krlist$. Note that being a subposet is a weaker assumption than being a sublattice.

\begin{problem}\label{problemBG}
Let $X$ and $L$ be finite lattices such that $X$ is a \emph{subposet} of $L$.
\begin{enumeratei}
\item\label{problemBGa} Is $\many(L)\leq\many(X)$ necessarily true in this case?
\item\label{problemBGb} With the additional assumption that
\[X\in\set{B,C,D,E_0,E_1,F_0,G_0,H_0},\]
is $\many(L)\leq\many(X)$ necessarily true?
\end{enumeratei}
\end{problem}

If we could answer at least
 least part \eqref{problemBGb} of Problem~\ref{problemBG} affirmatively, 
then the proof of Theorem~\ref{thmreform} would
only require  the lemmas of (the present) Section~\ref{sectiontools} and an easy application of a straightforward computer program. Later, Remark~\ref{remarkWmthnK} and Example~\ref{exampledifficulty} will point out why Problem~\ref{problemBG} is not as easy as it may look at first sight.

\begin{remark}\label{remarkF0htp}
There are a lot of finite lattices $X$ such that for every finite lattice $L$, 
$\many(L)\leq \many(X)$ if $X$ is a subposet of $L$. For example, $X=F_0$ has this property.
\end{remark}

\begin{proof}[Proof of Remark~\ref{remarkF0htp}] Every finite chain obviously has the property above, whence there are ``a lot of'' such lattices. 
Although the proof of Lemma~\ref{lemmaonF0} will, in effect, establish the above property of $X=F_0$, we can present a short argument for this fact right now. (But this short argument is not independent from Lemma~\ref{lemmaonF0} and it relies on Theorem~\ref{thmreform} that has not yet been proved at this stage of the paper.)  For the sake of contradiction, suppose that $X=F_0$ is a subposet of $L$ but $\many(L) > \many(F_0)$. For a computer, it is straightforward to show that $\many(F_0)=83$; see Lemma~\ref{lemmakrlat} later. So $\many(L)>83$, and we obtain 
from Theorem~\ref{thmreform} that $L$ is planar. Hence, by Theorem~\ref{thmKellyRival}, $F_0$ cannot be a subposet of $L$, which is a contradiction.
\end{proof}

\begin{figure}[htb] 
\centerline
{\includegraphics[scale=1.0]{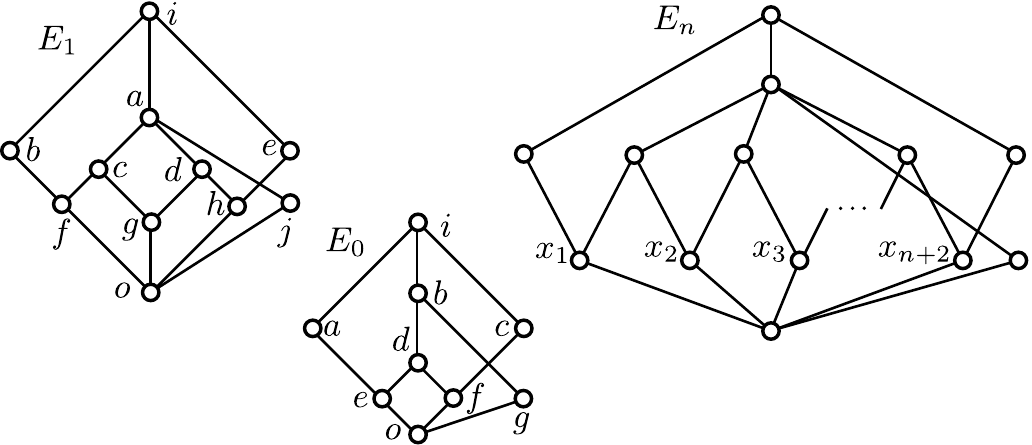}}
\caption{$E_n$ and, in particular, $E_0$ and $E_1$
\label{figE}}
\end{figure}%
\begin{figure}[htb] 
\centerline
{\includegraphics[scale=1.0]{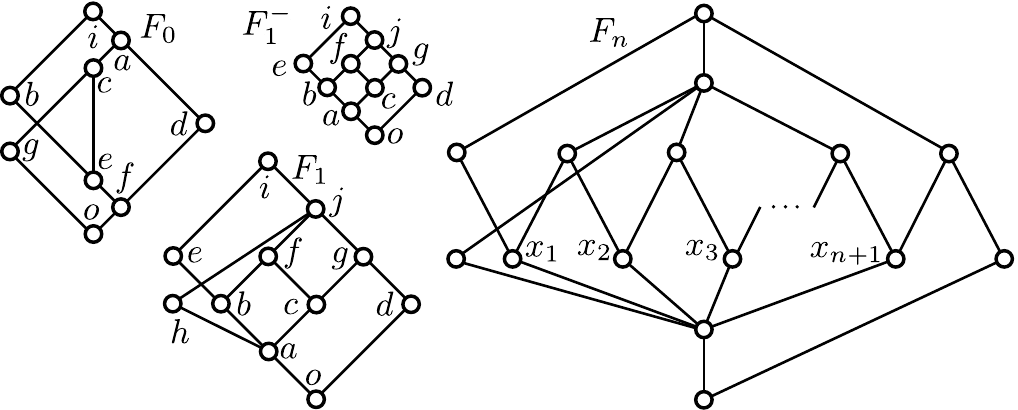}}
\caption{The encapsulated 2-ladder $\fem $, $F_n$,  and, in particular, $F_0$ and $F_1$
\label{figF}}
\end{figure}%
\begin{figure}[htb] 
\centerline
{\includegraphics[scale=1.0]{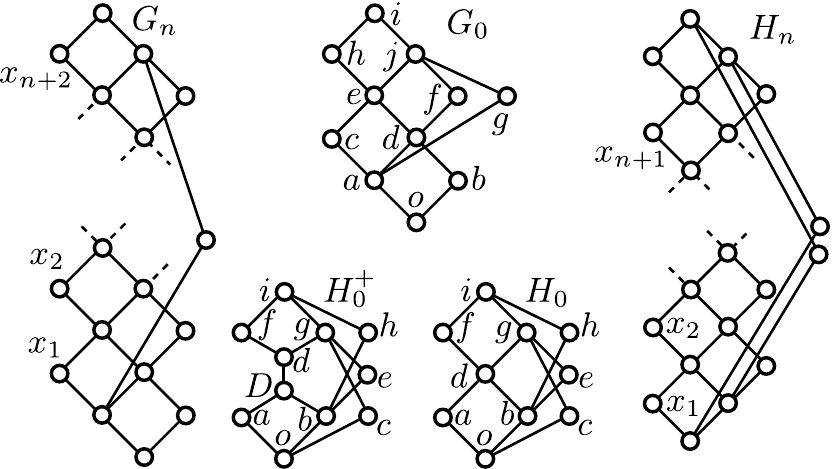}}
\caption{$G_n$ and $H_n$, in particular, $G_0$ and $H_0$, and the auxiliary lattice $H_0^+$
\label{figGH}}
\end{figure}%

\subsection{A computer program} Since it would be a very tedious task to compute $\many(X)$ manually even for the smallest lattice $X\in\krlist$, we have developed a straightforward computer program for Windows 10 to do it. This program, called \emph{subsize}, is downloadable from the authors website. The input of the program is an unformatted text file describing a finite binary partial algebra $\alg A=(A;F)$; there are several word processors that can produce such a file. 
In particular, the description of $\alg A$ includes a list of strings $x\ast y=z$ of length five where $\ast$ is an operation symbol in $F$, $(x,y)\in\dom{f_A}$ and $f_A(x,y)=z$; these strings are called \emph{constraints} in the input file.
The output, $\many(\alg A)$, is displayed on the screen and saved into a text file. The algorithm is trivial: the program lists all the $2^{|A|}$ subsets of $A$ and counts those that are closed with respect to all constraints.

\begin{remark} There are two kinds of difficulties we have to face.  First, we could not solve Problem~\ref{problemBG}; see the paragraph following it. Second, the running time of our program depends exponentially on the input size $|A|$. Hence, a lot of theoretical considerations are necessary before resorting to the program and what is even worse, many cases have to be input into the program. Because of the exponential time, it is not clear (and it is not hopeful) whether the appropriate cases could be found by a much more involved (and so less reliable) computer program without a lot of human work. So the program is simple, we believe it is reliable, and it is not to hard to write another program to test our input files. On the other hand, the exceptionally tedious work to find the appropriate cases and to create the input files needed several weeks. 
\end{remark}

However, it  is quite easy to obtain the following statement with the help of our computer program.

\begin{lemma}[on small Kelly-Rival lattices\label{lemmakrlat}]\ 
\begin{enumeratei}
\item\label{lemmakrlata} For the smallest lattices in $\krlist$, we have that
$\many(A_0)=74$, $\many(B)=54$, $\many(C)=68.5$, $\many(D)=76$, $\many(E_0)=60.5$, $\many(F_0)=83$, $\many(G_0)=54.25$, and $\many(H_0)=49.75$.
\item\label{lemmakrlatb} We also have that $\many(E_1)=31.125$ and $\many(F_1)=41.125$.
\end{enumeratei}
\end{lemma}

Except for its equality $\many(A_0)=74$, this lemma will not be used in the proof of Theorem~\ref{thmreform}. However, a part of this lemma will be used in the proof of Remark~\ref{remarksharp} below, and it is this lemma that tells us how the theorem was \emph{conjectured}. Even the proof of (part \eqref{lemmakrlata} of) this lemma requires more computation than a human is willing to carry out or check without a computer.

\begin{proof}[Proof of Remark \ref{remarksharp}] For $n=9$, 
the equality $\many(F_0)=83$ from Lemma~\ref{lemmakrlat} proves the validity of Remark \ref{remarksharp} since $F_0$ is not planar by Theorem~\ref{thmKellyRival}. Assume that $n>9$, let $C$ be an $(n-9)$-element chain, and let $L$ be the ordinal sum of $F_0$ and $C$. That is, $L$ is the disjoint union of its ideal $F_0$ and its filter $C$. By Theorem~\ref{thmKellyRival}, $L$ is not planar. Since a subset of $L$ is a subuniverse if and only if it is of the form $X\cup Y$ such that $X\in\Sub(F_0)$ and $Y\subseteq C$, it follows that 
\[|\Sub(L)|=|\Sub(F_0)|\cdot 2^{n-9}=(83\cdot 2^{|F_0]-8})\cdot  2^{n-9} = 83\cdot 2^{n-8},
\]
whereby $L$ has exactly $83\cdot 2^{n-8}-1$ sublattices, as required. 
\end{proof}

\begin{prooftechnique}\label{prooftechniqueXegy}
For Lemma~\ref{lemmakrlat} and also for all other statements that refer to the program or mention $\many(\dots)$, the corresponding input files are available from the author's website
\verb+http://www.math.u-szeged.hu/~czedli/+ . 
The output files proving these statements are also available there and they are attached as appendices to the \csakau{extended version of the}
paper.
\csakau{\footnote{\red{At present, THIS is the extended version.}}; see \url{https://arxiv.org/} . }
\end{prooftechnique}

\subsection{Lattice theoretical preparations}
The proof of Theorem~\ref{thmreform} will be organized as follows. 
Due to Theorem~\ref{thmKellyRival}, it suffices to show that for each lattice $X\in\krlist$, whenever $L$ is a lattice with \smany{} subuniverses (that is, $\many(L)>83$), then $X$ cannot be a subposet of $L$. Although we present some uniform arguments for several infinite sub-families of $\krlist$, separate arguments will be needed for most of the small lattices in $\krlist$. The following lemma is crucial.

\begin{lemma}[Antichain Lemma]\label{lemmacube}
If $\set{a_0,a_1,a_2}$ is a three-element antichain 
in a finite lattice with \smany{} subuniverses, then
\begin{enumeratei}
\item\label{lemmacubea}
There is a $k\in\set{0,1,2}$ such that
$a_0\vee a_1\vee a_2=\bigvee\set{a_i: i\in\set{0,1,2}\setminus\set k}$.
\item\label{lemmacubeb}
If $\set{i,j,k}=\set{0,1,2}$ and none of $a_i\vee a_j$ and $a_i\vee a_k$ is $a_0\vee a_1\vee a_k$, then $a_i\vee a_j\neq a_i\vee a_k$.
\end{enumeratei}
\end{lemma}

Part \eqref{lemmacubeb} of this lemma is trivial; we present it here to emphasize its implicit use in our considerations and in the input files of the program.

\begin{proof}
For the sake of contradiction, suppose that \eqref{lemmacubea} fails for a lattice $L$ with \smany{} subuniverses. Then $X:=\set{a_0\vee a_1, a_0\vee a_2, a_1\vee a_2}$ is a three-element antichain. It is well known that such an antichain generates a sublattice isomorphic to $A_0$, the eight-element boolean lattice; see, for example, Gr\"atzer~\cite[Lemma 73]{ggglt}. Combining Lemmas~\ref{lemmapartext} and \ref{lemmakrlat}, we obtain that  $\many(L)\leq \many(A_0)=74$, which contradicts the assumption that $\many(L)>83$.
\end{proof}

\begin{lemma}\label{lemmaBool8} 
If $L$ is a finite lattice with \smany{} subuniverses, then $A_0$ is not a subposet of $L$. 
\end{lemma}

\begin{proof} For the sake of contradiction, suppose that $A_0$ is a subposet of $L$ and $\many(L)>83$. Since $a,b,c$ play symmetric roles, Lemma~\ref{lemmacube}\eqref{lemmacubea} allows us to assume that $c\leq a\vee b\vee c=a\vee b$ in $L$. Then $c\leq a\vee b\leq e$ is a contradiction, as required.
\end{proof}

The following lemma needs a bit longer proof and the use of the program. This proof exemplifies many ideas that will be needed later.
Note that $K_5$, defined by Figure~\ref{figCDK}, is a sublattice of $G_n$ and $H_n$ for $n\geq 1$, this is why it deserves our attention.

\begin{lemma}\label{lemmaKot} 
If $L$ is a finite lattice with \smany{} subuniverses, then $K_5$ is not a subposet of $L$.
\end{lemma}

\begin{proof} 
For the sake of contradiction, suppose that $\many(L)>83$ but $K_5$ is a subposet of $L$. For the notation of the elements of $K_5$, see Figure~\ref{figCDK}.

\begin{preparation} We modify $K_5$ in $L$ if necessary. The operations $\vee$ and $\wedge$ will be understood in $L$.  We can assume that $e\vee f=g$, since otherwise we can replace $g$ by $e\vee f$. Of course, we have to show that this replacement results in an isomorphic subposet, but this is easy; analogous tasks will often be left to the reader. Namely,  $e\vee f\leq h$ would lead to $e\leq h$, a contradiction, while  $e\vee f\geq h$ combined with $g\geq e\vee f$ would lead to $g\geq h$, another contradiction. By duality, we also assume that $e\wedge f=c$. Next, we can assume $c\wedge d=b$ and, dually, $g\vee h=k$, because otherwise we can replace $b$ and $k$ by $c\wedge d$ and $g\vee h$, respectively. This is possible since, for example, $a\not\leq d$ implies that $a\not\leq c\wedge d$ while $c\wedge d\geq b$ and $a\not\geq b$ exclude that $a\geq c\wedge d$. 
In the next step, we assume similarly that $a\wedge b=o$ and $j\vee k=i$.   
Note that the equalities assumed so far and the comparability relations among the elements imply further equalities: $e\wedge d=e\wedge (f\wedge d)=(e\wedge f)\wedge d=c\wedge d = b$, 
$a\wedge d=a\wedge c\wedge d=a\wedge b=o$ and, dually, $e\vee h=k$ and $j\vee h=i$.  The set
\begin{align*}
T:=\{e\vee f&=g, e\wedge f=c, g\vee h=k, c\wedge d=b, a\wedge b=o,\cr
 j\vee k&=i, e\wedge d=b, a\wedge d=o, e\vee h=k,  j\vee h=i\} 
\end{align*}
defines a partial algebra $\alg K_5^{(0)}$ on the set $K_5$, which is a weak partial subalgebra of $L$. Note that the program calls the members of $T$ \emph{constraints}. 

\end{preparation}

\begin{computation}
The program proves that $\many(\alg K_5^{(0)})= 97.375$, which means that we are not ready yet. Thus, a whole hierarchy of cases have to be investigated in general. (Here, there will be only two cases.) The idea is that for incomparable elements $x$ and $y$, in notation, $x\parallel y$, such that $x\vee y$ or $x\wedge y$ is not defined in the partial algebra,  the argument splits into two cases: either $x\vee y$ (or $x\wedge y$) is one of the elements already present, or it is a new element of $L$ that we add to the partial algebra. In terms of the program, we add a new constraint with or without adding a new element. Also, when we add a constraint, then we also add its consequences similarly to the previous paragraph where, say, $e\wedge d=b$. Note that if an element had three covers or three lower covers, then we would use Lemma~\ref{lemmacube} to split a case into three subcases, but this 
technique will be used later, not in the present proof. 
A case with name  $\ast$ will be denoted by (C$\ast$).

\eset1 We assume that $c\vee d=f$ and $g\wedge h=f$. Then $e\vee d=e\vee c\vee d=e\vee f=g$ and, dually, $e\wedge h=c$. Adding these four constraints to the earlier ones, we get a new partial algebra $\alg K_5^{(1)}$, which is a weak subalgebra of $L$. The program yields that $\many(\alg K_5^{(1)})=79.1875$. 
Hence, $\many(L)\leq 79.1875$  by Lemma~\ref{lemmapartext},  contradicting the initial assumption that $\many(L)>83$. This excludes \azeset1.

Based on the argument for \azeset1{} above, to make our style more concise, let us agree to the following terminological issue, which will usually be used implicitly in the rest of the paper.

\begin{terminology}\label{terminSvL}
The cases we consider describe partial algebras, which are weak partial subalgebras of $L$; the  $\many$-values of these partial algebras will be called the  \emph{$\many$-values} of the corresponding cases. If the $\many$-value of a case is not greater than $83$, then the case in question is excluded.
\end{terminology}

\eset2 We assume that $c\vee d=:x < y:=g\wedge h$.  We remove $f$ from the weak partial algebra and add $x$ and $y$. We remove the constraints of $T$ that contain $f$ but we add the new constraints $c\vee d=x$, $g\wedge h=y$, 
$e\vee y=g$, and $e\wedge x=c$. The last two constraints we add follow from  $x\leq f\leq y$ and the previous constraints containing $f$. Note that the oval in Figure~\ref{figCDK} reminds us that now $\set f$ is replaced by $\set{x,y}$. Since the $\many$-value of the present situation is $80.5625$, \azeset2 is excluded.

After excluding  both cases, that is, all possible cases, the proof of the lemma is complete. \qedhere
\end{computation}
\end{proof}

Next, for later reference, we formulate a consequence, which trivially follows from Lemma~\ref{lemmaKot}.

\begin{corollary}\label{corolGnHn} If $L$ is a lattice with \smany{} subuniverses and $n\geq 1$, then none of $G_n$ and $H_n$ is a subposet of $L$. 
\end{corollary}

In order to formulate the following lemma about the \emph{encapsulated $2$-ladder} $\fem $ given in Figure~\ref{figF}, we need the following definition. This concept will be motivated by Corollary~\ref{corollaryF1} later.

\begin{definition}\label{defencapsladd}
Let $L$ and $K$ be a finite lattices. 
A mapping $\phi\colon K\to L$ will be called a \emph{\eqref{eqpbx2emb}-embedding} if 
\begin{equation}
\left.
\parbox{8.1cm}{$\phi$ is an order-embedding, $\phi(u)=\phi(v)\vee \phi(w)$ holds for every triplet $(u,v,w)\in K^3$ of distinct elements such that $u$ covers both $v$ and $w$,  and dually.}\right\}
\label{eqpbx2emb}
\end{equation}
\end{definition}
Note that if $v$ and $w$ are distinct elements covered by $u$ in $K$, then $u=v\vee_{K} w$, and the dual of this observation also holds. Hence, every lattice embedding is a \eqref{eqpbx2emb}-embedding but, clearly, not conversely.

\begin{lemma}[Encapsulated 2-ladder Lemma]\label{lemmaencaps2ladder}
If the encapsulated 2-ladder $\fem$ is a subposet of a lattice $L$, then it has a \eqref{eqpbx2emb}-embedding into $L$. 
\end{lemma}

\begin{proof} 
We can assume that $\fem \subseteq L$. The notation of the elements of $\fem $ is given in Figure~\ref{figF}. We are going to modify these elements in $L$ if necessary in order to obtain a  \eqref{eqpbx2emb}-embedding. The operations $\vee$ and $\wedge$ will be understood in $L$. 
First, we let $f':=b\vee c$. Since $b\not\leq g$, we have that $f'\not\leq g$. Since $f'\leq f$ and $f\not\geq g$, we obtain that $f'\not\geq g$. That is, $f'$ is incomparable with $g$; in notation, $f'\parallel g$. We obtain similarly that $f'\parallel x$ for all $x\in \fem $ such that $x\parallel f$. This allows us to replace $f$ by $f'$. To ease the notation, we will write $f$ instead of $f'$. So, $\fem$ is still a subposet of $L$ but now $f=b\vee c$. Next, we replace $c$ by $c':=f\wedge g\geq c$; then it is straightforward to see (or it follows by duality) that we still have a poset embedding. Since  $f=b\vee c \leq b\vee c'\leq f$, we have that $f=b\vee c'$. Thus, after writing $c$ instead of $c'$, the notation still gives a poset embedding of $\fem$ into $L$ with the progress that now $b\vee c=f$ and $f\wedge g=c$. We continue in the same way step by step, always defining a new poset embedding such that the already established equalities remain true; note that the order of adjusting the elements is not at all arbitrary. In the next step, we replace $b$ by $b':=e\wedge f\geq b$ and $g$ by $g':= c\vee d\leq g$ to add 
$b=e\wedge f$ and  $g= c\vee d$ to the list of valid equalities. 
We continue with setting $a=b\wedge c$ and $j=f\vee g$ similarly. Finally, redefining $i$ and $o$ as $e\vee j$ and $a\wedge d$, we complete the proof.
\end{proof}

Armed with Lemma~\ref{lemmaencaps2ladder}, we can give an easy proof of the following statement.

\begin{corollary}\label{corollaryF1} If $L$ is a lattice with \smany{} subuniverses, then  $F_1$ is not a subposet of $L$. 
\end{corollary}

\begin{proof} 
Suppose the contrary. Then $\fem$, which is a sublattice of $F_1$, is also a subposet of $L$. 
By Lemma~\ref{lemmaencaps2ladder}, we can assume that $\fem$ is a subposet of $L$ such that the inclusion map is a  \eqref{eqpbx2emb}-embedding. Hence, we know that
\begin{align}
 e\wedge f&=b,\,\, c\vee d=g,\,\, b\wedge c=a,\,\, f\vee g=j,
\label{LemmaF1a}\\
a\wedge d&=o,\,\, e\vee j=i ,\,\,  b\vee c=f,\,\, f\wedge g=c,\,\,\label{LemmaF1b}    
\\
b\vee g&=b\vee c\vee g=f\vee g=j,\,\, b\wedge g=b\wedge f\wedge g=b\wedge c =a,\,\,\label{LemmaF1c}\\
c\wedge e&=c\wedge f\wedge e=c\wedge b=a,\,\, f\vee d=f\vee c\vee d=f\vee g=j.\label{LemmaF1d}
\end{align}
The $\many$-value of the situation described by \eqref{LemmaF1a}--\eqref{LemmaF1d} is $81.75$. 
\end{proof}

The \emph{eight-element fence} is the poset formed by the eight empty-filled elements on the right of Figure~\ref{figCDK}. If we add the dashed line to its diagram, then we obtain the diagram of the \emph{eight-crown}. So the diagram of the eight-crown consists of the eight empty-filled elements $a,b,\dots,h$, seven solid edges and a dashed one. Note that the eight-crown is a subposet of $A_1$, see Figure~\ref{figAB}, but the eight-element fence is not.

\begin{lemma}\label{lemma8fence} If $L$ is a finite lattice with \smany{} subuniverses, then neither  the eight-element fence, nor the eight-crown is a subposet of $L$. 
\end{lemma}

\begin{proof} 
To ease the terminology in this proof, by the \emph{eight-poset} $P_8$ we shall mean either the eight-element fence, or the eight-crown; see  Figure~\ref{figCDK} for the notation of its elements. For the sake of contradiction, suppose that $\many(L)>83$ but $P_8$ is a subposet of $L$.

\begin{preparation} The set of \emph{atoms} and that of \emph{coatoms} of $P_8$ is $\set{a,c,e,g}$ and $\set{b,d,f,h}$, respectively. 
We claim that the subposet $P_8$ of $L$ can be chosen so that 
\begin{equation}\left.
\parbox{9.2cm}
{if $x$ and $y$ are distinct atoms of $P_8$ and $z\in P_8$   
such that $x\leq z$ and $y\leq z$, then $z=x\vee_L y$, and dually for coatoms.}\right\}
\label{eqpbxLfjTmCxk}
\end{equation}
In particular, \eqref{eqpbxLfjTmCxk} implies that the equalities
\begin{equation}
a\vee c=b,\,c\vee e=d,\, e\vee g=f,  b\wedge d=c,\, d\wedge f=e,\,\, f\wedge h=g
\label{eqgzerhBrPqznk}
\end{equation}
hold;
here and later in the proof, the lattice operations are understood in $L$.  In order to prove  \eqref{eqpbxLfjTmCxk}, we will modify the elements of $P_8$ one by one until all equalities listed in \eqref{eqpbxLfjTmCxk} hold.
By duality, it suffices to show that for each coatom $z$ of $P_8$
 covering two distinct atoms, $x$ and $y$ of $P_8$,  if we replace $z$ by $z':=x\vee y$, then the subposet  $(P_8\setminus\set{z})\cup \set{z'}$ of $L$ is still isomorphic to $P_8$ and, in addition to the progress $x\vee y=z'$, all the previously  valid equalities from \eqref{eqpbxLfjTmCxk} remain true if we replace $z$ by $z'$ in them. 

If $z$ is a meetand in an equality from \eqref{eqpbxLfjTmCxk} that holds in $L$, then the meet is $x$ or $y$, and  $x\leq z'\leq z$ or $y\leq z'\leq z$ shows that the equality remains true after replacing $z$ by $z'$. As a coatom of $P_8$, $z$ can be neither a joinand, nor a meet in an equality from \eqref{eqpbxLfjTmCxk}. Finally, the only stipulation of \eqref{eqpbxLfjTmCxk} with $z$ being a join is the equality with joinands $x$ and $y$; this fails with $z$ but becomes true after replacing $z$ by $z'$. 

Next, we show that the map $P_8\to (P_8\setminus\set{z})\cup \set{z'}$, defined by $z\mapsto z'$ and $u\mapsto u$ for $u\neq z$, is an order isomorphism. Let $u\in P_8\setminus\set{z}$. Since $z$ is a coatom of $P_8$, $z\not\leq u$. If we had $z'\leq u$, then $x\leq u$,  $y\leq u$, and $u\in P_8$ would give that $u=z$, contradicting $u\in P_8\setminus\set{z}$. That is, neither $u\leq z$, nor $u\leq z'$ holds. If $u\leq z'$, then we conclude  $u\leq z$ since
$z'<z$. 
Conversely, if $u\leq z$, then $u\in\set{x,y}$ since $x$ and $y$ are the only elements of $P_8$ below $z$, whereby 
$u\leq z'$. This shows that the map in question is an order isomorphism and completes the proof of \eqref{eqpbxLfjTmCxk}. Thus, we have also proved  \eqref{eqgzerhBrPqznk}.  

Next, we define $o:=c\wedge e$ and $i:=d\vee f$ in $L$. They are distinct new elements since $\set{a,c,e,g}$ and $\set{b,d,f,h}$ are antichains. We have  that 
\begin{equation}
d\vee f=i,\, c\wedge e=o,\, 
 b\wedge e=o,\ c\wedge f=o,\, d\vee g=i,\, f\vee c=i,
\label{eqmrqflDktNms}
\end{equation}
since the first two of these equalities are due to definitions and the rest are easy consequences; for example, 
$b\wedge e=b\wedge d\wedge e=c\wedge e=o$ while the rest follow by duality or symmetry.
\end{preparation}

\begin{computation} 
For 
the elements $a,b,\dots,h,o,i$ subject to \eqref{eqgzerhBrPqznk}  and \eqref{eqmrqflDktNms}, the $\many$-value is $84.5$; see Terminology~\ref{terminSvL}. In other words, we obtain with our usual technique (that is, using the program and  Lemma~\ref{lemmapartext}) that $\many(L)\leq 84.5$. Since this estimate is too week to derive a contradiction, we distinguish two cases.

\eset1 We assume that $b\vee d=i$. Then $b\vee e=i$
also holds since $b\vee e=b\vee c\vee e=b\vee d=i$. Adding these two equalities to  \eqref{eqgzerhBrPqznk}  and \eqref{eqmrqflDktNms}, the $\many$-value is 79,  which excludes this case. 

\eset2 We assume that $x:=b\vee d\neq i$. We also have that $b\vee e=x$ since $b\vee e=b\vee c\vee e=b\vee d$. Now we have eleven elements and, in addition to the two equalities just mentioned,  \eqref{eqgzerhBrPqznk},  and \eqref{eqmrqflDktNms}.
Since the $\many$-value is $77.25$, this case is also excluded. 

Both cases have been excluded, which proves Lemma~\ref{lemma8fence}.\qedhere
\end{computation}
\end{proof}

The lemma we have just proved trivially implies the following statement.

\begin{corollary}\label{corolAnEnpFnp} If $L$ is a lattice with \smany{} subuniverses and $n\geq 1$, then none of $A_n$,  $E_{n+1}$, and  $F_{n+1}$ is a subposet of $L$. 
\end{corollary}

In the rest of the paper, due to Corollaries~\ref{corolGnHn} and \ref{corolAnEnpFnp} and the Duality Principle, we need to exclude only finitely many members of the infinite list $\krlist$ as  subposets of a finite lattice $L$ with \smany{} subuniverses. After the proofs of Lemmas~\ref{lemmaKot} and \ref{lemma8fence}, our plan to exclude that a given member $X$ of $\krlist$ occurs as a subposet of a lattice $L$ with $\many(L)>83$ is the following. After assuming that $X$ is a subposet of $L$, first we need some lattice theoretical preparation to ensure a feasible computational time. In the second phase, we reduce the estimate on  $\many(L)$  by assuming equations and introducing new elements in a systematic way until we obtain that $\many(L)\leq 83$. In other words, we keep branching cases until all ``leaves of our parsing tree'' have $\many$-values at most 83. 
Unfortunately, this plan requires quite a lot of work; 
see Table~\eqref{tablazat} later. In the rest of the paper, we present  some of the details in order the give a better impression how our plan works. The rest of the details are 
given by the output files of our program and some of them in the extended version of the paper; see Proof Technique~\ref{prooftechniqueXegy} for their coordinates.

\begin{remark}\label{remarkWmthnK}
One may think of the following possibility: if $X\in\krlist$ is a subposet of $L$ with $\many(L)>83$, $c,e\in X$, and $c\vee_X e=g$, then 
either $c\vee_L e=g$ in $L$, or $x:=c\vee_L e<g\in L\setminus X$. If we could show that 
\begin{equation}\left.
\parbox{9cm}{if $X$ belongs to $\krlist$, then the second alternative (with $x$) \emph{always} yields a better (that is, smaller) estimate of $\many(L)$,}\right\}
\label{eqpbxsCndwthXx}
\end{equation}
then $X$ being a \emph{sublattice} would give the worst estimate but even this estimate would be sufficient to imply Theorem~\ref{thmreform} by  Lemma~\ref{lemmakrlat}.
We do not know if \eqref{eqpbxsCndwthXx} is true; the following example, in which $X$ happens not to be in $\krlist$, illustrates why \eqref{eqpbxsCndwthXx} and Problem~\ref{problemBG} are probably  difficult. 
\end{remark}

\begin{example}[Example to indicate difficulty]\label{exampledifficulty}
Let us denote by $X$ the subposet $\set{c,d,e,f,g,o,i}$ of $H_0$; see Figure~\ref{figGH}. 
Note that $X$ is a lattice but not a sublattice of $H_0$.
If $X$ is a subposet of a finite lattice $L$ such that
\begin{equation}
c\vee_L e=g\,\,\text{ and }\,\,g\wedge_L f=d,
\label{eqzhGtmVbpQntx}
\end{equation}
then $\many$ for the weak partial subalgebra of $L$ with base set $\set{c,d,e,f,g,o,i}$ and the equalities of \eqref{eqzhGtmVbpQntx} equals 192. So \eqref{eqzhGtmVbpQntx} is appropriate to show that $\many(L)\leq 192$. 
However, if drop  the first equality in \eqref{eqzhGtmVbpQntx}
and replace it by $c\vee_L e=x$, where $x<g$, then the weak partial subalgebra with base set $\set{c,d,e,f,g,o,i,x}$ and  equalities $c\vee e=x$ and $g\wedge f=d$ gives a worse estimate,
$\many(L)\leq 196$. 
\end{example}

\section{The rest of the lemmas and some proofs}\label{sectionrestlemmas}
In order to complete the proof of Theorem~\ref{thmreform}, we still need the following eight lemmas, in which $L$ denotes a finite lattice.

\begin{lemma}\label{lemmaonB} If $\many(L)>83$, then $B$ is not a subposet of $L$.
\end{lemma}

\begin{lemma}\label{lemmaonC} If $\many(L)>83$, then $C$ is not a subposet of $L$.
\end{lemma}

\begin{lemma}\label{lemmaonD} If $\many(L)>83$, then $D$ is not a subposet of $L$.
\end{lemma}

\begin{lemma}\label{lemmaonE0} If $\many(L)>83$, then $E_0$ is not a subposet of $L$.
\end{lemma}

\begin{lemma}\label{lemmaonE1} If $\many(L)>83$, then $E_1$ is not a subposet of $L$.
\end{lemma}

\begin{lemma}\label{lemmaonF0} If $\many(L)>83$, then $F_0$ is not a subposet of $L$.
\end{lemma}

\begin{lemma}\label{lemmaonG0} If $\many(L)>83$, then $G_0$ is not a subposet of $L$.
\end{lemma}

\begin{lemma}\label{lemmaonH0} If $\many(L)>83$, then $H_0$ is not a subposet of $L$.
\end{lemma}

\begin{proof}[Proof of Lemma~\ref{lemmaonF0}]
 For the sake of contradiction, suppose that $\many(L)>83$ but $F_0$ is a subposet of $L$. 

\begin{preparation} 
Unless otherwise stated, the lattice operations are understood in $L$; in notation, $x\vee y$ will mean $x\vee_L y$ and dually. Note that $F_0$ is a selfdual lattice and it has a unique dual automorphism
\[\begin{pmatrix}
i&a&b&c&d&e&f&g&o\cr
o&f&g&e&d&c&a&b&i
\end{pmatrix}.
\]
Since $e\vee_{F_0} g=c$, we have that $e\vee g\leq c$. If $c':=e\vee g<c$, then we  replace $c$ by $c'$. Observe that $e\not\leq d$ and $g\not\leq b$ imply that $c'\not\leq d$ and $c'\not\leq b$. Since $c'<c$ but $b\not\leq c$ and $d\not\leq c$, we also have that $b\not\leq c'$ and  $d\not\leq c'$. So it follows that the subposet $(F_0\setminus\set{c})\cup\set{c'}$ of $L$ is isomorphic to $F_0$.
 Hence, after replacing $c$ by $c'$ if necessary, we can assume that $e\vee g=c$. In the next step, after replacing $e$ by $e':=b\wedge c$, we assume that $b\wedge c=e$; we still have a subposet (isomorphic to) $F_0$. Clearly, $e\vee g=c$ remains valid, because  $c=e\vee g\leq e'\vee g\leq c$.
With $f':=e\wedge d\geq f$,  $f'\geq g$ would give that $e\geq g$ while $f'\leq g$ would lead to $f\leq g$. Hence, $f'\parallel g$. After replacing $f$ by $f'$ if necessary, we can assume that  $e\wedge d=f$. A dual argument allows us to assume that  $c\vee d=a$. In the next step, 
we can clearly assume that $a\vee b=i$ and $f\wedge g=o$.
To summarize, we have assumed that the inclusion map is a \eqref{eqpbx2emb}-embedding of $F_0$ into $L$, that is,
\begin{equation}
b\wedge c=e, \,\, e\vee g=c,\,\,c\vee d=a,\,\, d\wedge e=f,\,\, 
a\vee b=i,\,\,f\wedge g=o.
\label{eqfnnTtl}
\end{equation}
\end{preparation}

\begin{computation}
While splitting the possibilities into cases and subcases, we will benefit from the fact that both $F_0$ and \eqref{eqfnnTtl} are selfdual. We keep splitting (sub)cases to more specific subcases only as long as their $\many$-values are larger than 83; this tree-like splitting structure will have thirteen leaves, that is, thirteen subcases with small $\many$-values that cover all possibilities. Every case below is either \emph{evaluated}, that is, its $\many$-value is computed by the program, or the case is split further. Of courses, we have evaluated all cases to see which of them need further splitting, but  we present the $\many$-values only of the non-split cases, because only the thirteen evaluated cases are needed in the proof. 
The (sub)cases are denoted by strings. When a case (C$\vec{\textup x}$) is mentioned, 
all the ``ancestor cases'', that is, (C$\vec{\textup y}$) for all meaningful prefixes $\vec y$ of $\vec{\textup x}$ are automatically assumed.

\ceset{(C1)}: $b\vee c=i$ is assumed; then $b\vee g=b\vee e\vee g=b\vee c = i$ also holds. 

\ceset{(C1a)}: $e\wedge g=o$. Since this is the dual of the previous assumption, we are in a selfdual situation. Observe that $b\wedge g=b\wedge c\wedge g=e\wedge g=o$.

\ceset{(C1a.1)}: $d\vee e=a$, then $b\vee d=b\vee e\vee d=b\vee a=i$.

\ceset{(C1a.1a)}: $c\wedge d=f$;  then $d\wedge g=d\wedge c\wedge g=f\wedge g=o$. Again, we are in a selfdual situation.

\ceset{(C1a.1a.1)}: $a\wedge b=e$; then $b\wedge d=b\wedge a\wedge d=e\wedge d=f$. 

\ceset{(C1a.1a.1a)}: $f\vee g=c$; then $d\vee g=d\vee f\vee g=d\vee c=a$. (Note that this case describes the situation when $F_0$ is a sublattice of $L$.) Since the $\sigma$-value of this case is 83, 
$L$ has few subuniverses, whereby (C1a.1a.1a) is excluded.

\ceset{(C1a.1a.1b)}: $f\vee g=:x$ such that $x\neq c$. (The notation ``$=:$'' means that $x$ is defined as $f\vee g$ and $f\vee g=x$ is a new constraint.) 
Clearly,  $x<c$. Using the incomparabilities among the elements of $F_0$, it is straightforward to see that $x$ is a new element. (In what follows in the paper, an element with a new notation will always be distinct from the rest of elements, but usually this fact will not be mentioned and its straightforward verification will be omitted.) Since 
$c=e\vee g\leq e\vee x\leq c$, we have that $e\vee x=c$. Since the $\many$-value of this case is $74.25$, (C1a.1a.1b)  is excluded. 
Thus, the case (C1a.1a.1) is also is excluded. Since (C1a.1a) is
 seldfual, the dual of (C1a.1a.1) is also excluded; this will be used in the next case.  

\ceset{(C1a.1a.2)}: $a\wedge b=:x>e$ and $f\vee g=:y<c$. Since 
$c\wedge b=e$ and $e\vee g=c$, it follows easily that $c\wedge x=e$ and  $e\vee y=c$
Since the $\many$-value is now $68$, (C1a.1a.2) is excluded. Thus, (C1a.1a) is also excluded.

\ceset{(C1a.1b)}: $c\wedge d=:x>f$. Then  $e\wedge x=f$  since $e\wedge d=f$.

\ceset{(C1a.1b.1)}: $x\wedge g=o$. Then  $d\wedge g=g\wedge c\wedge d=g\wedge x=o$.

\ceset{(C1a.1b.1a)} $a\wedge b=e$. Now the $\many$-value is $78.25$, excluding this case.

\ceset{(C1a.1b.1b)}: $a\wedge b=:y>e$. This case is excluded again since its $\many$-value is $78.5$. Thus, (C1a.1b.1) is also excluded.

\ceset{(C1a.1b.2)}: $x\wedge g=:y>o$. Here $y$ is a new element since  $g>y>o$, and we have that $d\wedge g=d\wedge c\wedge g=x\wedge g=y$. This case is excluded, because its $\many$-value is $79.375$.  Thus, (C1a.1b) and so (C1a.1) are also excluded. Furthermore, since (C1a) is selfdual, we conclude that dual of (C1a.1) is also excluded; this fact will be used in the next case.

\ceset{(C1a.2)}: $d\vee e=:v<a$ and $c\wedge d=:u>f$. Observe that  $c\vee v=a$ and $e\wedge u=f$, since $c\vee d=a$ and $e\wedge d=f$. Let $x:=e\vee u$. In order to verify its novelty, observe that $e\leq x\leq c$ since $e<c$ and $u\leq c$. But $x=e$ would imply $u\leq e$, whence $u=e\wedge u=e\wedge c\wedge d= e\wedge d=f$, a contradiction. Also, $x=c$ would lead to $v=e\vee d= e\vee (u\vee d)= (e\vee u)\vee d=x\vee d=c\vee d=a$, a contradiction again. Hence, $e<x<c$, which implies easily that $x$ is a new element. 
We have that $d\vee x=v$ since $e \leq x\leq v$. Similarly, $b\wedge x=e$ since $e \leq x \leq c$.  Now the  $\many$-value of the situation is $66$, whereby this case is excluded. Thus, (C1a) is also excluded.

\ceset{(C1b)}: $e\wedge g=x >o$; then  $f\wedge x=o$ since $f\wedge g=o$.

\ceset{(C1b.1)}: $d\vee e=a$; then $b\vee d=b\vee e\vee d=b\vee a=i$.

\ceset{(C1b.1a)}: $c\wedge d=f$; then $d\wedge g=d\wedge c\wedge g=f\wedge g=o$.

\ceset{(C1b.1a.1)}: $a\wedge b=e$. Now $b\wedge d=b\wedge a\wedge d=e\wedge d=f$ and $\many=78.25$ excludes this case.

\ceset{(C1b.1a.2)}: $a\wedge b=:y>e$. Then  $c\wedge y=e$  since $c\wedge b=e$, and $\many= 72$ excludes this case.  Thus, (C1b.1a) is also excluded.

\ceset{(C1b.1b)}: $c\wedge d=:u>f$; then $e\wedge u=f$ since $e\wedge d=f$, and $\many=80$ excludes this case. Thus, (C1b.1) is also excluded.

\ceset{(C1b.2)}: $d\vee e=:v <a$; then $c\vee v=a$  since $c\vee d=a$.

\ceset{(C1b.2a)}: $c\wedge d=f$. Then $d\wedge g=d\wedge c\wedge g=f\wedge g=o$ and $\many=79.375$ excludes this case.

\ceset{(C1b.2b)}: $c\wedge d=:u>f$. Then  $e\wedge u=f$  since $e\wedge d=f$, and $\many=75.5$ excludes this case. Thus,  (C1b.2),  (C1b), and even  (C1) are excluded. Furthermore, since the underlying assumption, \eqref{eqfnnTtl}, is selfdual, the dual of (C1) is also excluded; this fact will be used below when (C2) is analysed.

\ceset{(C2)} $b\vee c=:t<i$ and $e\wedge g=:s>o$. Using $a\vee b=i$ and $f\wedge g=o$, we obtain that $a\vee t=i$ and $f\wedge s=o$. Also, $b\vee g=b\vee e\vee g=b\vee c=t$ 
and $b\wedge g=b\wedge c\wedge g=e\wedge g=s$, and so $\many=82.5$, excluding this case. All cases have been excluded, and the proof of Lemma~\ref{lemmaonF0} is complete. \qedhere
\end{computation}
\end{proof}

Note that the proof above required to compute an estimate for $\many(L)$ thirteen times. Let us call these thirteen values \emph{final $\many$-values}.  However, as mentioned previously,  many more values were needed  
to \emph{find} the proof. For example, the  $\many$-value of (C1a.1a.1) is $90.5$, and the inequality $90.5>83$ is the reason to split the case (C1a.1a.1) into subcases.

\begin{remark}[Notes on the proofs of Lemmas \ref{lemmaonB}--\ref{lemmaonH0}]\label{remarknotsPrf}
First, observe that  $\many(F_0)=83$ is the  largest $\many$-value occurring in Lemma~\ref{lemmakrlat}. Thus, Lemma~\ref{lemmaonF0} devoted to $F_0$ is the most crucial one in the paper. Since even the ``human part'' of its computer-assisted proof is long and threatens with unnoticed human errors, we have elaborated two separate proofs of Lemma~\ref{lemmaonF0}. One of these proofs is optimized in some sense and it has already been given, and it is also available from  the corresponding file \texttt{F0-output.txt}.
The other proof is less optimized and it is described only by its output file called \texttt{F0-alternative-output.txt}.

One might think that, compared to Lemma~\ref{lemmaonF0}, the seven other lemmas of this section are easier simply because while Lemma~\ref{lemmaonF0} is devoted to $F_0$ 
and $\many(F_0)=83$, the other lemmas deal with lattices $X\in\krlist$ with $\many(X)<83$. However, some of these lemmas need even more tedious proofs than Lemma~\ref{lemmaonF0}. Because such amount of straightforward technicalities would not be too exciting for the reader and because of space considerations, these proofs are not given in the concise version of this paper; some of them are  appendices in the extended version of the paper, and all of them are downloadable as output files of our computer program. Assuming that the reader shares our trust in our computer program or he writes another computer program, these files constitute \emph{complete} proofs.  In particular, these files include lots of comments that make them almost as detailed as the proof of Lemma~\ref{lemmaonF0}.
\end{remark}

\begin{remark}[On the lengths of the proofs of Lemmas \ref{lemmaonB}--\ref{lemmaonH0}]\label{remarknotdKzR}
The table below gives the numbers of final $\many$-values that our proofs, that is, the program output files, contain. We have already mentioned that the \ref{lemmaonF0}-labeled column gives 13. 
The $\ast$-labeled column refers to the second proof of Lemma~\ref{lemmaonF0} given in the downloadable file \texttt{F0-alternative-output.txt}.
\begin{equation}
\lower  0.0 cm
\vbox{\tabskip=0pt\offinterlineskip
\halign{\strut#&\vrule#\tabskip=4pt plus 1pt&
#\hfill& \vrule\vrule\vrule#&
\hfill#&\vrule#&
\hfill#&\vrule#&
\hfill#&\vrule#&
\hfill#&\vrule#&
\hfill#&\vrule#&
\hfill#&\vrule#&
\hfill#&\vrule#&
\hfill#&\vrule#&
\hfill#&\vrule\tabskip=0.0pt#&
#\hfill\vrule\vrule\cr
\vonal
&&Lemma&&\ref{lemmaonB}&&\ref{lemmaonC}&&\ref{lemmaonD}&&\ref{lemmaonE0}&&\ref{lemmaonE1}&&\ref{lemmaonF0}&&$\ast$&&\ref{lemmaonG0}&&\ref{lemmaonH0}&\cr
\vonal
&&$X\in\krlist$&&$B$&&$C$&&$D$&&$E_0$&&$E_1$&&$F_0$&&$F_0$&&$G_0$&&$H_0$&\cr
\vonal\vonal\vonal
&&$|\{$final $\many$-values$\}|$&&$11$&&$12$&&${5}$&&$37$&&$5$&&$13$&&$19$&&$24$&&$67$&\cr
\vonal}}
\label{tablazat}
\end{equation}

In order to explain some large numbers in the third row of the table, note the following. 
If $X\in\krlist$ contains no element with more than two covers or more than two lower covers, then the proof of the corresponding lemma is quite similar to that of Lemma~\ref{lemmaonF0}; of course, we can exploit duality only if $X$ itself is selfdual. Note that symmetries with respect to automorphisms can also be exploited. However, if there are elements with more than two lower covers or dually,  like in case of $H_0$, then there can be  cases that we split into three subcases according to Lemma~\ref{lemmacube} as follows. Let $a_0$, $a_1$, and $a_2$ be the three lower covers of an element $b$ in $X\in\krlist$; the case of upper covers is analogous. Let  $c:= a_0\vee  a_1\vee a_2$ in $L$.  (Usually, $b$ is meet-irreducible and we can let $c:=b$.) Then the following three subcases are considered. First, $a_0\vee a_1=c$. Second, $a_0\vee a_1:=x<c$ is a new element and $a_0\vee a_2=c$. Third, $a_0\vee a_1=:x<c$,  $a_0\vee a_2=:y<c$, $x\neq y$ are new elements, and $a_1\vee a_2=c$. It is not surprising now that this three-direction splitting leads to more final $\many$-values than the two-direction splittings in the proof of Lemma~\ref{lemmaonF0}. 
With an opposite effect, there is another factor related to the numbers of final $\many$-values. Namely,  $\many(X)<\many(F_0)=83$ for all   $X\in\krlist\setminus\set{F_0}$ that occur in the lemmas of this section, whereby we do not have to be so efficient for these $X$ as for $F_0$; simply because our lemmas for these $X$ state less than an affirmative answer to Problem~\ref{problemBG}\eqref{problemBGb}. To conclude Remark~\ref{remarknotdKzR}, we mention that there are many ways to prove the eight lemmas with the help of our program, and not much effort has been devoted to reduce the numbers in the third row of Table~\ref{tablazat}; such an effort would have required too much work. Some of these numbers might decrease in the future. 
\end{remark}

Finally, armed with our lemmas and corollaries, we are ready to present the concluding proof of the paper.

\begin{proof}[Proof of Theorem~\ref{thmreform}] For the sake of contradiction, suppose that $L$ is a finite lattice such that $\many(L)>83$. By Lemmas~\ref{lemmaBool8} and   \ref{lemmaonB}--\ref{lemmaonH0} and Corollaries~\ref{corolGnHn}, \ref{corollaryF1}, and  \ref{corolAnEnpFnp}, none of the lattices occurring as excluded subposets in these twelve statements is a subposet of $L$. Using $\many(\dual L)=\many(L)$ and applying these twelve statements to $\dual L$, we obtain that none of the duals of the excluded lattices is a subposet of $L$. Hence, no member of $\krlist$ is a subposet of $L$, and  
Theorem~\ref{thmKellyRival} implies that $L$ is planar, as required. 
\end{proof}

\clearpage
\centerline{\bf{{\LARGE APPENDICES}}}
\normalfont
\bigskip
\def\appendixhead{{Cz\'edli: Eighty-three sublattices / Appendices}}
\markboth\appendixhead\appendixhead

\noindent The rest of the paper is devoted to appendices, which are the output files of our computer program. These files are integral parts of the proof of the main result of the paper; they together with a sample input file are given in the following order; for each of them, we give the name of the appendix, the name of the downloadable output file, and we refer to the relevant statement(s) in the paper.

\begin{itemize}
\item $B$, \texttt{B-output.txt}; see Lemma~\ref{lemmaonB}
\item $C$, \texttt{C-output.txt}; see Lemma~\ref{lemmaonC}
\item $D$, \texttt{D-output.txt}; see Lemma~\ref{lemmaonD}
\item $E_0$, \texttt{E0-output.txt}; see Lemma~\ref{lemmaonE0}
\item $E_1$, \texttt{E1-output.txt}; see Lemma~\ref{lemmaonE1} 
\item Eight-element fence, \texttt{fence8.txt}; see Lemma~\ref{lemma8fence}
\item $F_0$, \texttt{F0-output.txt}; see Lemma~\ref{lemmaonF0}
\item $F_0$-alternative, \texttt{F0-alternative-output.txt}; see Lemma~\ref{lemmaonF0}
\item $F_1$, \texttt{F1-output.txt}; see Lemma~\ref{lemmaencaps2ladder} and Corollary~\ref{corollaryF1}
\item $G_0$, \texttt{G0-output.txt}; see Lemma~\ref{lemmaonG0}
\item $H_0$, \texttt{H0-output.txt}; see Lemma~\ref{lemmaonH0}
\item $K_5$, \texttt{K5-output.txt}; see Lemma~\ref{LemmaF1d}
\item Sample input; \texttt{0-readme-1st-and-sample-input.txt}; see Example~\ref{exampledifficulty}
\item Small Kelly-Rival lattices, \texttt{small-kelly-rival-lat.txt}; see Lemma~\ref{lemmakrlat}
\end{itemize}
The \emph{input files} are approximately of the same sizes as the corresponding output files and their contents are almost the same. In essence, the only difference is that the input files contain some \emph{commands} like ``\verb+\size+'', ``\verb+\constraints+'', etc., but they do not contain the $\many$-values.

\clearpage
\def\appendixhead{{Cz\'edli: Eighty-three sublattices / Appendix: $B$}}
\markboth\appendixhead\appendixhead
\centerline{\bf{{\LARGE Appendix: $B$}}}
\normalfont
\begin{verbatim}
Version of November 2, 2018; reformatted on May 6, 2019
B from Kelly-Rival: "Planar lattices"
Here we prove that B is not a subposet of L. Suppose the contrary.
We can and we will always assume that 
(#0:)  d+e=a e+g=c e+f=b a+c=i  d+c=i a+g=i d+g=i
  (After letting d+g=i, only the ad <-> cg symmetry remains!)
We know that two of a,b,c intersect at e.
**(Case 1) a*c=e. Then with d*g=:o, o<e
****(Subcase 1a)  a*b=e
****(Sub-subcase 1a1) f':=o+f>=e, 
     then o+f=b and f*o=:0 is the bottom
******(Sub-sub-subcase 1a1a)  a*f=0  then  d*f=0 
|A|=10, A(without commas)={iabcdefgo0}. Constraints:
d+e=a e+g=c e+f=b a+c=i  d+c=i a+g=i d+g=i;(#0)
a*c=e  d*g=o;(Case 1)
a*b=e;(1a)
o+f=b  f*o=0;(1a1)
a*f=0    d*f=0;(1a1a)
Result for A=Case1a1a:  |Sub(A)| = 315, that is,
sigma(A) = |Sub(A)|*2^(8-|A|) =  78.7500000000000000 .
    Few subuniverses, (Sub-sub-subcase 1a1a) is excluded

**(Case 1) a*c=e. Then with d*g=:o, o<e
****(Subcase 1a)  a*b=e
****(Sub-subcase 1a1) f':=o+f>=e, 
      then o+f=b and f*o=:0 is the bottom
******(Sub-sub-subcase 1a1b)  a*f=p >0 
|A|=11, A(without commas)={iabcdefgo0p}. Constraints:
d+e=a e+g=c e+f=b a+c=i  d+c=i a+g=i d+g=i;(#0)
a*c=e  d*g=o;(Case 1)
a*b=e;(1a)
o+f=b  f*o=0;(1a1)
a*f=p;(1a1b)
Result for A=Case1a1b:  |Sub(A)| = 634, that is,
sigma(A) = |Sub(A)|*2^(8-|A|) =  79.2500000000000000 .
    Few subuniverses, (Sub-sub-subcase 1a1b) is excluded
    Thus, (Sub-subcase 1a1) is excluded

**(Case 1) a*c=e. Then with d*g=:o, o<e
****(Subcase 1a)  a*b=e
****(Sub-subcase 1a2) f':=o+f\not\geq e, then let f:=f', 
      we still have b=e+f but now o is the bottom
    Let e*f=:p >o, then  a*f=a*b*f=e*f= p 
******(Sub-sub-subcase 1a2a) d+p=a 
|A|=10, A(without commas)={iabcdefgop}. Constraints:
d+e=a e+g=c e+f=b a+c=i  d+c=i a+g=i d+g=i;(#0)
a*c=e  d*g=o;(Case 1)
a*b=e;(1a)
e*f=p  a*f=p;(1a2)
d+p=a;(1a2a)
Result for A=Case1a2a:  |Sub(A)| = 328, that is,
sigma(A) = |Sub(A)|*2^(8-|A|) =  82.0000000000000000 .
    Few subuniverses, (Sub-sub-subcase 1a2a) is excluded

**(Case 1) a*c=e. Then with d*g=:o, o<e
****(Subcase 1a)  a*b=e
****(Sub-subcase 1a2) f':=o+f\not\geq e, then let f:=f', 
      we still have b=e+f but now o is the bottom
    Let e*f=:p >o, then  a*f=a*b*f=e*f= p 
******(Sub-sub-subcase 1a2b) d+p=:q<a, then q+e=a 
|A|=11, A(without commas)={iabcdefgopq}. Constraints:
d+e=a e+g=c e+f=b a+c=i  d+c=i a+g=i d+g=i;(#0)
a*c=e  d*g=o;(Case 1)
a*b=e;(1a)
e*f=p  a*f=p;(1a2)
d+p=q  q+e=a;(1a2b)
Result for A=Case1a2b:  |Sub(A)| = 597, that is,
sigma(A) = |Sub(A)|*2^(8-|A|) =  74.6250000000000000 .
    Few subuniverses, (Sub-sub-subcase 1a2b) is excluded
    Thus, (Sub-subcase 1a2) is excluded

**(Case 1) a*c=e. Then with d*g=:o, o<e
****(Subcase 1a)  a*b=e
****(Sub-subcase 1a3) e*f=o is the bottom. Then a*f=a*b*f=e*f=o
******(Sub-sub-subcase 1a3a)  a*f=p>o, then e*p=o
|A|=10, A(without commas)={iabcdefgop}. Constraints:
d+e=a e+g=c e+f=b a+c=i  d+c=i a+g=i d+g=i;(#0)
a*c=e  d*g=o;(Case 1)
a*b=e;(1a)
 e*f=o  a*f=o;(1a3)
 a*f=p  e*p=o;(1a3a)
Result for A=Case1a3a:  |Sub(A)| = 322, that is,
sigma(A) = |Sub(A)|*2^(8-|A|) =  80.5000000000000000 .
    Few subuniverses, (Sub-sub-subcase 1a3a) is excluded

**(Case 1) a*c=e. Then with d*g=:o, o<e
****(Subcase 1a)  a*b=e
****(Sub-subcase 1a3) e*f=o is the bottom. Then a*f=a*b*f=e*f=o
******(Sub-sub-subcase 1a3b)  a*f=o, then d*f=o
******(Sub-sub-sub-subcase 1a3b1)  f*c=o, then f*g=o
|A|=9, A(without commas)={iabcdefgo}. Constraints:
d+e=a e+g=c e+f=b a+c=i  d+c=i a+g=i d+g=i;(#0)
a*c=e  d*g=o;(Case 1)
a*b=e;(1a)
 e*f=o  a*f=o;(1a3)
 a*f=o   d*f=o;(1a3b)
 f*c=o   f*g=o;(1a3b1)
Result for A=Case1a3b1:  |Sub(A)| = 163, that is,
sigma(A) = |Sub(A)|*2^(8-|A|) =  81.5000000000000000 .
    Few subuniverses, Sub-sub-sub-subcase (1a3b1)  is excluded

**(Case 1) a*c=e. Then with d*g=:o, o<e
****(Subcase 1a)  a*b=e
****(Sub-subcase 1a3) e*f=o is the bottom. Then a*f=a*b*f=e*f= o
******(Sub-sub-subcase 1a3b)  a*f=o, then d*f=o
******(Sub-sub-sub-subcase 1a3b2)  f*c=:p>o, then e*p=o
|A|=10, A(without commas)={iabcdefgop}. Constraints:
d+e=a e+g=c e+f=b a+c=i  d+c=i a+g=i d+g=i;(#0)
a*c=e  d*g=o;(Case 1)
a*b=e;(1a)
 e*f=o  a*f=o;(1a3)
 a*f=o   d*f=o;(1a3b)
 f*c=p   e*p=o;(1a3b2)
Result for A=Case1a3b2:  |Sub(A)| = 296, that is,
sigma(A) = |Sub(A)|*2^(8-|A|) =  74.0000000000000000 .
    Few subuniverses, (Sub-sub-sub-subcase 1a3b2) is excluded
    Thus, (Sub-sub-subcase 1a3b) is excluded. 
    Thus, (Sub-subcase 1a3) is excluded. 
    Thus, (Subcase 1a) is excluded.

**(Case 1) a*c=e. Then with d*g=:o, o<e
****(Subcase 1b)  a*b=:u>e, then u*c=e
******(Sub-subcase 1b1) u+c=i, then b+c=i
|A|=10, A(without commas)={iabcdefgou}. Constraints:
d+e=a e+g=c e+f=b a+c=i  d+c=i a+g=i d+g=i;(#0)
a*c=e  d*g=o;(Case 1)
a*b=u  u*c=e;(1b)
u+c=i   b+c=i;(1b1)
Result for A=Case1b1:  |Sub(A)| = 318, that is,
sigma(A) = |Sub(A)|*2^(8-|A|) =  79.5000000000000000 .
    Few subuniverses, (Sub-subcase 1b1) is excluded

**(Case 1) a*c=e. Then with d*g=:o, o<e
****(Subcase 1b)  a*b=:u>e, then u*c=e
******(Sub-subcase 1b2) u+c=:v<i, then a+v=i
|A|=11, A(without commas)={iabcdefgouv}. Constraints:
d+e=a e+g=c e+f=b a+c=i  d+c=i a+g=i d+g=i;(#0)
a*c=e  d*g=o;(Case 1)
a*b=u  u*c=e;(1b)
 u+c=v   a+v=i;(1b2)
Result for A=Case1b2:  |Sub(A)| = 607, that is,
sigma(A) = |Sub(A)|*2^(8-|A|) =  75.8750000000000000 .
    Few subuniverses, (Sub-subcase 1b2) is excluded
    Thus, (Subcase 1b) is excluded
    Thus, (Case 1) is excluded

We still always assume that 
(#0:)  d+e=a e+g=c e+f=b a+c=i  d+c=i a+g=i d+g=i
  (After letting d+g=i, only the ad <-> cg symmetry remains!)
We know that two of a,b,c intersect at e. By ad <-> cg symmetry
the only remaining case is a*b=e. 
**(Case 2) a*b=e.  Then a*c=e is excluded by Subcase 1a, whence
a*c=:v>i.  It follows that v*b=e. With d*f=:o, o=d*a*b*f=o*e<=e
******(Subcase 2a)  v+b=:q<i 
|A|=11, A(without commas)={iabcdefgovq}. Constraints:
d+e=a e+g=c e+f=b a+c=i  d+c=i a+g=i d+g=i;(#0)
a*b=e  a*c=v   v*b=e  d*f=o;(Case2)
v+b=q;(Subcase 2a)
Result for A=Subcase2a:  |Sub(A)| = 618, that is,
sigma(A) = |Sub(A)|*2^(8-|A|) =  77.2500000000000000 .
    Few subuniverses, (Subcase 2a) is excluded.

(#0:)  d+e=a e+g=c e+f=b a+c=i  d+c=i a+g=i d+g=i
  (After letting d+g=i, only the ad <-> cg symmetry remains!)
We know that two of a,b,c intersect at e. By ad <-> cg symmetry
the only remaining case is a*b=e. 
**(Case 2) a*b=e. Then a*c=e is excluded by Subcase 1a, whence
a*c=:v>i.  It follows that v*b=e. With d*f=:o, o=d*a*b*f=o*e<=e
******(Subcase 2b)  v+b=i 
|A|=10, A(without commas)={iabcdefgov}. Constraints:
d+e=a e+g=c e+f=b a+c=i  d+c=i a+g=i d+g=i;(#0)
a*b=e  a*c=v   v*b=e  d*f=o;(Case2)
v+b=i;(Subcase 2b)
Result for A=Subcase2a:  |Sub(A)| = 316, that is,
sigma(A) = |Sub(A)|*2^(8-|A|) =  79.0000000000000000 .
    Few subuniverses, (Subcase 2b) is excluded.
    Thus, (Case 2) is excluded.
    So both possible cases are excluded.
    Therefore, B cannot be a subposet of L, as required. Q.e.d.

The computation took 78/1000 seconds.
\end{verbatim}
\normalfont

\clearpage
\def\appendixhead{{Cz\'edli: Eighty-three sublattices / Appendix: $C$}}
\markboth\appendixhead\appendixhead
\centerline{\bf{{\LARGE Appendix: $C$}}}
\normalfont
\begin{verbatim}
Version of November 4, 2018; reformatted on May 6, 2019
C from Kelly-Rival: "Planar lattices"
Here we prove that C is not a subposet of L. Suppose the contrary.
We can and we will always assume that 
(#0:) a*b=d, b*c=e, d*e=u, a*c=u, a*e=u, d*c=u, u*f=o, d*f=o
 Since we assumed that d*f=o, there is no symmetry for the rest.
Two of d,e,f give their join; 
      assume this join is b (otherwise modify b)
There will be three cases depending on which two of d,e,f give b.
**(Case 1) d+e=b, then a+c=a+d+e+c=a+b+c=i, the top element.
****(Subcase 1a) a+b=i, then a+e=a+d+e=i
******(Sub-subcase 1a1) b+c=i, then d+c= d+e+c=b+c=i 
********(Sub-sub-subcase 1a1a) d+f=b, then a+f=a+d+f=a+b=i 
|A|=9, A(without commas)={iabcdefuo}. Constraints:
a*b=d  b*c=e  d*e=u  a*c=u  a*e=u  d*c=u  u*f=o  d*f=o; (#0)
d+e=b  a+c=i; (Case 1)
 a+b=i  a+e=i; (Subcase 1a)
b+c=i  d+c=i;(Sub-subcase 1a1)
d+f=b  a+f=i; (Sub-sub-subcase 1a1a)
Result for A=Case1a1a:  |Sub(A)| = 164, that is,
sigma(A) = |Sub(A)|*2^(8-|A|) =  82.0000000000000000 .
   Few subuniverses, (Sub-sub-subcase 1a1a) is excluded.

(#0:) a*b=d, b*c=e, d*e=u, a*c=u, a*e=u, d*c=u, u*f=o, d*f=o
Two of d,e,f gives their join; 
       assume it is b (otherwise modify b)
**(Case 1) d+e=b, then a+c=a+d+e+c=a+b+c=i, the top element.
****(Subcase 1a) a+b=i, then a+e=a+d+e=i
******(Sub-subcase 1a1) b+c=i, then d+c= d+e+c=b+c=i 
********(Sub-sub-subcase 1a1b) d+f=:x<b, then x+e=b
|A|=10, A(without commas)={iabcdefuox}. Constraints:
a*b=d  b*c=e  d*e=u  a*c=u  a*e=u  d*c=u  u*f=o  d*f=o; (#0)
d+e=b  a+c=i; (Case 1)
 a+b=i  a+e=i; (Subcase 1a)
b+c=i  d+c=i;(Sub-subcase 1a1)
d+f=x  x+e=b; (Sub-sub-subcase 1a1b)
Result for A=Case1a1b:  |Sub(A)| = 299, that is,
sigma(A) = |Sub(A)|*2^(8-|A|) =  74.7500000000000000 .
   Few subuniverses, (Sub-sub-subcase 1a1b) is excluded.
   Thus, (Sub-subcase 1a1) is excluded.

(#0:) a*b=d, b*c=e, d*e=u, a*c=u, a*e=u, d*c=u, u*f=o, d*f=o
Two of d,e,f gives their join; 
       assume that is b (otherwise modify b)
**(Case 1) d+e=b, then a+c=a+d+e+c=a+b+c=i, the top element.
****(Subcase 1a) a+b=i, then a+e=a+d+e=i
******(Sub-subcase 1a2) b+c=:p<i, then a+p=i
********(Sub-sub-subcase 1a2a) d+f=b, then a+f=a+d+f=a+b=i
|A|=10, A(without commas)={iabcdefuop}. Constraints:
a*b=d  b*c=e  d*e=u  a*c=u  a*e=u  d*c=u  u*f=o  d*f=o; (#0)
d+e=b  a+c=i; (Case 1)
 a+b=i  a+e=i; (Subcase 1a)
 b+c=p  a+p=i; (Sub-subcase 1a2)
 d+f=b  a+f=i; (Sub-sub-subcase 1a2a)
Result for A=Case1a2a:  |Sub(A)| = 321, that is,
sigma(A) = |Sub(A)|*2^(8-|A|) =  80.2500000000000000 .
   Few subuniverses, (Sub-sub-subcase 1a2a) is excluded.

(#0:) a*b=d, b*c=e, d*e=u, a*c=u, a*e=u, d*c=u, u*f=o, d*f=o
Two of d,e,f gives their join; 
       assume it is b (otherwise modify b)
**(Case 1) d+e=b, then a+c=a+d+e+c=a+b+c=i, the top element.
****(Subcase 1a) a+b=i, then a+e=a+d+e=i
******(Sub-subcase 1a2) b+c=:p<i, then a+p=i
********(Sub-sub-subcase 1a2b) d+f=x<b, then x+e=d
|A|=11, A(without commas)={iabcdefuopx}. Constraints:
a*b=d  b*c=e  d*e=u  a*c=u  a*e=u  d*c=u  u*f=o  d*f=o; (#0)
d+e=b  a+c=i; (Case 1)
 a+b=i  a+e=i; (Subcase 1a)
 b+c=p  a+p=i; (Sub-subcase 1a2)
 d+f=x  x+e=d; (Sub-sub-subcase 1a2b)
Result for A=Case1a2b:  |Sub(A)| = 537, that is,
sigma(A) = |Sub(A)|*2^(8-|A|) =  67.1250000000000000 .
   Few subuniverses, (Sub-sub-subcase 1a2b) is excluded.
   Thus, (Sub-subcase 1a2)  is excluded.
   Thus, (Subcase 1a) is excluded.           

(#0:) a*b=d, b*c=e, d*e=u, a*c=u, a*e=u, d*c=u, u*f=o, d*f=o
Two of d,e,f gives their join; 
       assume it is b (otherwise modify b)
**(Case 1) d+e=b, then a+c=a+d+e+c=a+b+c=i, the top element.
****(Subcase 1b)a+b=r <i, then r+c=i and a+e=a+d+e=a+b=r
******(Sub-subcase 1b1) b+c=:p<i, then a+p=i and r+p=a+b+b+c=i
|A|=11, A(without commas)={iabcdefuorp}. Constraints:
a*b=d  b*c=e  d*e=u  a*c=u  a*e=u  d*c=u  u*f=o  d*f=o; (#0)
d+e=b  a+c=i; (Case 1)
a+b=r  r+c=i  a+e=r;(Subcase 1b)
 b+c=p  a+p=i  r+p=i; (Sub-subcase 1b1)
Result for A=Case1b1:  |Sub(A)| = 580, that is,
sigma(A) = |Sub(A)|*2^(8-|A|) =  72.5000000000000000 .
   Few subuniverses, (Sub-subcase 1b1) is excluded.

(#0:) a*b=d, b*c=e, d*e=u, a*c=u, a*e=u, d*c=u, u*f=o, d*f=o
Two of d,e,f gives their join; 
       assume it is b (otherwise modify b)
**(Case 1) d+e=b, then a+c=a+d+e+c=a+b+c=i, the top element.
****(Subcase 1b) a+b=r <i, then r+c=i and a+e=a+d+e=a+b=r
******(Sub-subcase 1b2) b+c=i, then  d+c=d+e+c=b+c=i
********(Sub-sub-subcase 1b2a) d+f=:x<b, then x+e=d
|A|=11, A(without commas)={iabcdefuorx}. Constraints:
a*b=d  b*c=e  d*e=u  a*c=u  a*e=u  d*c=u  u*f=o  d*f=o; (#0)
d+e=b  a+c=i; (Case 1)
 a+b=r  r+c=i  a+e=r;(Subcase 1b)
 b+c=i  d+c=i;(Sub-subcase 1b2)
 d+f=x  x+e=d; (Sub-sub-subcase 1b2a)
Result for A=Case1b2a:  |Sub(A)| = 527, that is,
sigma(A) = |Sub(A)|*2^(8-|A|) =  65.8750000000000000 .
   Few subuniverses, (Sub-subcase 1b2a) is excluded.

(#0:) a*b=d, b*c=e, d*e=u, a*c=u, a*e=u, d*c=u, u*f=o, d*f=o
Two of d,e,f gives their join; 
       assume it is b (otherwise modify b)
**(Case 1) d+e=b, then a+c=a+d+e+c=a+b+c=i, the top element.
****(Subcase 1b) a+b=r <i, then r+c=i and a+e=a+d+e=a+b=r
******(Sub-subcase 1b2) b+c=i, then  d+c=d+e+c=b+c=i
********(Sub-sub-subcase 1b2b) d+f=b, then a+f=a+d+f=a+b=r
|A|=10, A(without commas)={iabcdefuor}. Constraints:
a*b=d  b*c=e  d*e=u  a*c=u  a*e=u  d*c=u  u*f=o  d*f=o; (#0)
d+e=b  a+c=i; (Case 1)
 a+b=r  r+c=i  a+e=r;(Subcase 1b)
 b+c=i  d+c=i;(Sub-subcase 1b2)
 d+f=b  a+f=r;(Sub-sub-subcase 1b2b)
Result for A=Case1b2b:  |Sub(A)| = 307, that is,
sigma(A) = |Sub(A)|*2^(8-|A|) =  76.7500000000000000 .
   Few subuniverses, (Sub-subcase 1b2b) is excluded.
   Thus, (Sub-subcase 1b2) is excluded.         
   Thus, (Subcase 1b) is excluded.                  
   Thus, (Case 1) is excluded.        

(#0:) a*b=d, b*c=e, d*e=u, a*c=u, a*e=u, d*c=u, u*f=o, d*f=o
Two of d,e,f gives their join; 
       we assume it is b (otherwise modify b)
**(Case 2) d+f=b. Then  d+e=:x <b, and so x+f=b, 
           since (Case 1) has been excluded.
****(Subcase 2a) a+b=i, then a+f=a+d+f=a+b=i 
********(Sub-subcase 2a1) x*f=o  
|A|=11, A(without commas)={iabcdefoupx}. Constraints:
a*b=d  b*c=e  d*e=u  a*c=u  a*e=u  d*c=u  u*f=o  d*f=o; (#0)
d+f=b  d+e=x  x+f=b; (Case 2)
a+b=i  a+f=i;(Subcase 2a)
x*f=o;(Sub-subcase 2a1)
Result for A=Case2a1:  |Sub(A)| = 660, that is,
sigma(A) = |Sub(A)|*2^(8-|A|) =  82.5000000000000000 .
   Few subuniverses; (Sub-subcase 2a1) is excluded.

(#0:) a*b=d, b*c=e, d*e=u, a*c=u, a*e=u, d*c=u, u*f=o, d*f=o
Two of d,e,f gives their join; 
       we assume it is b (otherwise modify b)
**(Case 2) d+f=b. Then  d+e=:x <b, and so x+f=b, 
                  since (Case 1) has been excluded.
****(Subcase 2a) a+b=i, then a+f=a+d+f=a+b=i 
********(Sub-subcase 2a2) x*f=:y>o  
|A|=11, A(without commas)={iabcdefouxy}. Constraints:
a*b=d  b*c=e  d*e=u  a*c=u  a*e=u  d*c=u  u*f=o  d*f=o; (#0)
d+f=b  d+e=x  x+f=b; (Case 2)
a+b=i  a+f=i;(Subcase 2a)
 x*f=y;(Sub-subcase 2a2)
Result for A=Case2a2:  |Sub(A)| = 642, that is,
sigma(A) = |Sub(A)|*2^(8-|A|) =  80.2500000000000000 .
   Few subuniverses; (Sub-subcase 2a2) is excluded.
   Thus, (Subcase 2a) is excluded.   

(#0:) a*b=d, b*c=e, d*e=u, a*c=u, a*e=u, d*c=u, u*f=o, d*f=o
Two of d,e,f gives their join; 
       we assume it is b (otherwise modify b)
**(Case 2) d+f=b. Then  d+e=:x <b, and so x+f=b, 
                  since (Case 1) has been excluded.
****(Subcase 2b) a+b=:r< i, then a+f=a+d+f=a+b=r
********(Sub-subcase 2b1) x*f=:y>o  
|A|=12, A(without commas)={iabcdefourxy}. Constraints:
a*b=d  b*c=e  d*e=u  a*c=u  a*e=u  d*c=u  u*f=o  d*f=o; (#0)
d+f=b  d+e=x  x+f=b; (Case 2)
 a+b=r  a+f=r;(Subcase 2b)
 x*f=y;(Sub-subcase 2b1)
Result for A=Case2b1:  |Sub(A)| = 1284, that is,
sigma(A) = |Sub(A)|*2^(8-|A|) =  80.2500000000000000 .
   Few subuniverses; (Sub-subcase 2b1) is excluded.

(#0:) a*b=d, b*c=e, d*e=u, a*c=u, a*e=u, d*c=u, u*f=o, d*f=o
Two of d,e,f gives their join; 
       we assume it is b (otherwise modify b)
**(Case 2) d+f=b. Then  d+e=:x <b, and so 
     x+f=b, since (Case 1) has been excluded.
****(Subcase 2b) a+b=:r< i, then a+f=a+d+f=a+b=r
********(Sub-subcase 2b2) x*f=o  
|A|=11, A(without commas)={iabcdefourx}. Constraints:
a*b=d  b*c=e  d*e=u  a*c=u  a*e=u  d*c=u  u*f=o  d*f=o; (#0)
d+f=b  d+e=x  x+f=b; (Case 2)
 a+b=r  a+f=r;(Subcase 2b)
 x*f=o;(Sub-subcase 2b2)
Result for A=Case2b2:  |Sub(A)| = 660, that is,
sigma(A) = |Sub(A)|*2^(8-|A|) =  82.5000000000000000 .
   Few subuniverses; (Sub-subcase 2b2) is excluded.
   Thus, (Subcase 2b) is excluded.
   Thus, (Case 2) is excluded.

(#0:) a*b=d, b*c=e, d*e=u, a*c=u, a*e=u, d*c=u, u*f=o, d*f=o
Two of d,e,f gives their join; 
       we assume it is b (otherwise modify b)
**(Case 3) e+f=b. Then, since (Case 1), (Case 2) are excluded, 
d+e=:x <b and d+f=:z<b, and so x+f=b, e+z=b, and x+z=b
|A|=12, A(without commas)={iabcdefouxyz}. Constraints:
a*b=d  b*c=e  d*e=u  a*c=u  a*e=u  d*c=u  u*f=o  d*f=o; (#0)
e+f=b  d+e=x d+f=z  x+f=b e+z=b  x+z=b;(Case 3)
Result for A=Case3:  |Sub(A)| = 1120, that is,
sigma(A) = |Sub(A)|*2^(8-|A|) =  70.0000000000000000 .
   Few subuniverses; (Case 3) is excluded.
   There is no more case;        
   this proves that C cannot be a subposet of L.

The computation took 125/1000 seconds.

\end{verbatim}
\normalfont

\clearpage
\def\appendixhead{{Cz\'edli: Eighty-three sublattices / Appendix: $D$}}
\markboth\appendixhead\appendixhead

\centerline{\bf{{\LARGE Appendix: $D$}}}
\normalfont
\begin{verbatim}
Version of November 4, 2018; reformatted on May 6, 2019
D from Kelly-Rival: "Planar lattices"
Here we prove that D is not a subposet of L. Suppose the contrary.
We can and we will always assume that 
(#0:) d+c=a, d+g=a, c+e=b, g+e=b, f+g=c, 
      a+b=i, d+e=i, d+b=i, a+e=i, f*g=o
 We have an ab-be symmetry 
    and we can also assume that d*c*e=g, whence
 there are two cases to obtain f from the antichain d,c,e

**(Case 1) d*e=f
****(Subcase 1a) d*c=f 
******(Sub-subcase 1a1) c*e=f
******(Sub-sub-subcase 1a1a) a*b=c
         Remark: we do not use, say, a*e=a*b*e=c*e=f
|A|=9, A(without commas)={oiabcdefg}. Constraints:
d+c=a  d+g=a  c+e=b  g+e=b  f+g=c
       a+b=i  d+e=i  d+b=i  a+e=i  f*g=o;(#0)
 d*e=f;(Case 1)
 d*c=f;(Subcase 1a)
 c*e=f;(Sub-subcase 1a1)
 a*b=c;(Sub-sub-subcase 1a1a)
Result for A=Case1a1a:  |Sub(A)| = 160, that is,
sigma(A) = |Sub(A)|*2^(8-|A|) =  80.0000000000000000 .
 Few subuniverses, (Sub-sub-subcase 1a1a) is excluded.

**(Case 1) d*e=f
****(Subcase 1a) d*c=f 
******(Sub-subcase 1a1) c*e=f 
******(Sub-sub-subcase 1a1b) a*b=:x>c
|A|=10, A(without commas)={oiabcdefgx}. Constraints:
d+c=a  d+g=a  c+e=b  g+e=b  f+g=c
       a+b=i  d+e=i  d+b=i  a+e=i  f*g=o;(#0)
 d*e=f;(Case 1)
 d*c=f;(Subcase 1a)
 c*e=f;(Sub-subcase 1a1)
 a*b=x;(Sub-sub-subcase 1a1b)
Result for A=Case1a1b:  |Sub(A)| = 324, that is,
sigma(A) = |Sub(A)|*2^(8-|A|) =  81.0000000000000000 .
 Few subuniverses, (Sub-sub-subcase 1a1b) is excluded.
 Thus, (Sub-subcase 1a1) is excluded.

**(Case 1) d*e=f
****(Subcase 1a) d*c=f 
******(Sub-subcase 1a2) c*e=:y>f, then  d*y=f and y+g=c
|A|=10, A(without commas)={oiabcdefgy}. Constraints:
d+c=a  d+g=a  c+e=b  g+e=b  f+g=c
       a+b=i  d+e=i  d+b=i  a+e=i  f*g=o;(#0)
 d*e=f;(Case 1)
 d*c=f;(Subcase 1a)
 c*e=y     d*y=f  y+g=c;(Sub-subcase 1a2)
Result for A=Case1a2:  |Sub(A)| = 322, that is,
sigma(A) = |Sub(A)|*2^(8-|A|) =  80.5000000000000000 .
 Few subuniverses, (Sub-subcase 1a2) is excluded.
 Thus, (Subcase 1a) is excluded.

**(Case 1) d*e=f
****(Subcase 1b) d*c=:z>f, then z*e=f and z+g=c
    Since (Subcase 1a) is excluded, so is its mirror image, c*e=f
    If c*e=:y is comparable to c*d=z, 
    then f=d*e>=y or z is a contradiction.
    Hence, we also obtain that y and z is incomparable, and 
    c*e=y,  y*z=f, d*y=f, z*e=f, y+g=c
|A|=11, A(without commas)={oiabcdefgyz}. Constraints:
d+c=a  d+g=a  c+e=b  g+e=b  f+g=c
       a+b=i  d+e=i  d+b=i  a+e=i  f*g=o;(#0)
 d*e=f;(Case 1)
 d*c=z   z*e=f  z+g=c;(Subcase 1b)
 c*e=y   y*z=f  d*y=f  z*e=f  y+g=c;still from (Subcase 1b)
Result for A=Case1b:  |Sub(A)| = 531, that is,
sigma(A) = |Sub(A)|*2^(8-|A|) =  66.3750000000000000 .
 Few subuniverses, (Subcase 1b) is excluded.  
 Thus, (Case 1) is excluded.

**(Case 2) d*c=f, then d*g=d*c*g=f*g=o and 
           (since Case 1 is included) d*e=:u>f
and so u*g=u*d*g=u*o=o, u*c=u*d*c=u*f=f 
|A|=10, A(without commas)={oiabcdefgu}. Constraints:
d+c=a  d+g=a  c+e=b  g+e=b  f+g=c
       a+b=i  d+e=i  d+b=i  a+e=i  f*g=o;(#0)
 d*c=f  d*g=o d*e=u u*g=o u*c=f;(Case 2)
Result for A=Case2:  |Sub(A)| = 326, that is,
sigma(A) = |Sub(A)|*2^(8-|A|) =  81.5000000000000000 .
 Few subuniverses, (2) is excluded.
 Therefore,  D cannot be a subposet of L, as required.

The computation took 47/1000 seconds.
\end{verbatim}
\normalfont

\clearpage
\def\appendixhead{{Cz\'edli: Eighty-three sublattices / Appendix: $E_0$}}
\markboth\appendixhead\appendixhead

\centerline{\bf{{\LARGE Appendix: $E_0$}}}
\normalfont
\begin{verbatim}
Version of November 1, 2018; reformatted on May 6, 2019
E_0 from Kelly-Rival: "Planar lattices"
Here we prove that E_0 is not a subposet of L. 
Suppose the contrary.
We can and we will always assume that 
(#0:) a*d=e, c*d=f, e+f=d, d+g=b 
e*f=:u (we do not know yet if u is the bottom element=o), a*c=u, 
e*c=u, a*f=u, e+g=b (:#0)
~~Case (1): u is not o. (Observe that u is parallel to g !)
u*g=o (C1), and at least one of e*g, f*g is o. 
~~~~Subcase (1a): both  e*g=o and f*g=o
~~~~~~Sub-subcase (1a1) a+b=:i a+c=i b+c=i      
|A|=10, A(without commas)={oiabcdefgu}. Constraints:
a*d=e c*d=f e+f=d d+g=b e*f=u  a*c=u   e*c=u a*f=u  e+g=b;(#0)
u*g=o;(C1)
e*g=o  f*g=o;(C1a)
a+b=i a+c=i b+c=i;(C1a1)
Result for A=Case1a1:  |Sub(A)| = 298, that is,
sigma(A) = |Sub(A)|*2^(8-|A|) =  74.5000000000000000 .
   Few subuniverses, Sub-subcase (1a1) above is excluded.

~~Case (1): u is not o and u is parallel to g, whence 
u*g=o (C1), and at least one of e*g, f*g is o. 
~~~~Subcase (1a): both  e*g=o and f*g=o
~~~~~~Sub-subcase (1a2) a+b=i, a+c=i, b+c=:v<i. (Then a+v=i)   
|A|=11, A(without commas)={oiabcdefguv}. Constraints:
a*d=e c*d=f e+f=d d+g=b e*f=u  a*c=u   e*c=u a*f=u  e+g=b;(#0)
u*g=o;(C1)
e*g=o  f*g=o;(C1a)
a+b=i  a+c=i  b+c=v  a+v=i;(C1a2)
Result for A=Case1a2:  |Sub(A)| = 564, that is,
sigma(A) = |Sub(A)|*2^(8-|A|) =  70.5000000000000000 .
   Few subuniverses, Sub-subcase (1a2) above is excluded.

~~Case (1): u is not o and u is parallel to g, whence 
u*g=o (C1), and at least one of e*g, f*g is o. 
~~~~Subcase (1a): both  e*g=o and f*g=o
~~~~~~Sub-subcase (1a3) a+b=i b+c=i, a+c=:v<i. (Then b+v=i)   
|A|=11, A(without commas)={oiabcdefguv}. Constraints:
a*d=e c*d=f e+f=d d+g=b e*f=u  a*c=u   e*c=u a*f=u  e+g=b;(#0)
u*g=o;(C1)
e*g=o  f*g=o;(C1a)
a+b=i b+c=i  a+c=v  b+v=i;(C1a3)
Result for A=Case1a3:  |Sub(A)| = 543, that is,
sigma(A) = |Sub(A)|*2^(8-|A|) =  67.8750000000000000 .
   Few subuniverses, Sub-subcase (1a3) above is excluded.

~~Case (1): u is not o and u is parallel to g, whence 
u*g=o (C1), and at least one of e*g, f*g is o. 
~~~~Subcase (1a): both  e*g=o and f*g=o
~~~~~~Sub-subcase (1a4) a+c=i, b+c=i, a+b=:v. (Then c+v=i)   
|A|=11, A(without commas)={oiabcdefguv}. Constraints:
a*d=e c*d=f e+f=d d+g=b e*f=u  a*c=u   e*c=u a*f=u  e+g=b;(#0)
u*g=o;(C1)
e*g=o  f*g=o;(C1a)
a+c=i b+c=i a+b=v  c+v=i;(C1a4)
Result for A=Case1a4:  |Sub(A)| = 563, that is,
sigma(A) = |Sub(A)|*2^(8-|A|) =  70.3750000000000000 .
   Few subuniverses, Sub-subcase (1a4) above is excluded.

~~Case (1): u is not o and u is parallel to g, whence 
u*g=o (C1), and at least one of e*g, f*g is o. 
~~~~Subcase (1a): both  e*g=o and f*g=o
~~~~~~Sub-subcase (1a5) a+b=:v, b+c=:w, a+c=i.
                        (Then v+w=i, a+w=i, c+v=i)   
|A|=12, A(without commas)={oiabcdefguvw}. Constraints:
a*d=e c*d=f e+f=d d+g=b e*f=u  a*c=u   e*c=u a*f=u  e+g=b;(#0)
u*g=o;(C1)
e*g=o  f*g=o;(C1a)
a+b=v  b+c=w a+c=i  v+w=i a+w=i c+v=i;(C1a5)
Result for A=Case1a5:  |Sub(A)| = 952, that is,
sigma(A) = |Sub(A)|*2^(8-|A|) =  59.5000000000000000 .
   Few subuniverses, Sub-subcase (1a5) above is excluded.

~~Case (1): u is not o and u is parallel to g, whence 
u*g=o (C1), and at least one of e*g, f*g is o. 
~~~~Subcase (1a): both  e*g=o and f*g=o
~~~~~~Sub-subcase (1a6) a+c=:v, b+c=:w, a+b=i.
                        (Then v+w=i, a+w=i, b+v=i)   
|A|=12, A(without commas)={oiabcdefguvw}. Constraints:
a*d=e c*d=f e+f=d d+g=b e*f=u  a*c=u   e*c=u a*f=u  e+g=b;(#0)
u*g=o;(C1)
e*g=o  f*g=o;(C1a)
a+c=v  b+c=w a+b=i  v+w=i a+w=i b+v=i;(C1a6)
Result for A=Case1a6:  |Sub(A)| = 947, that is,
sigma(A) = |Sub(A)|*2^(8-|A|) =  59.1875000000000000 .
   Few subuniverses, Sub-subcase (1a6) above is excluded.

~~Case (1): u is not o and u is parallel to g, whence 
u*g=o (C1), and at least one of e*g, f*g is o. 
~~~~Subcase (1a): both  e*g=o and f*g=o
~~~~~~Sub-subcase (1a7) a+b=:v, a+c=:w, b+c=i.
          (Then v+w=i, b+w=i, c+v=i)   
|A|=12, A(without commas)={oiabcdefguvw}. Constraints:
a*d=e c*d=f e+f=d d+g=b e*f=u  a*c=u   e*c=u a*f=u  e+g=b;(#0)
u*g=o;(C1)
e*g=o  f*g=o;(C1a)
a+b=v  a+c=w b+c=i  v+w=i b+w=i c+v=i;(C1a7)
Result for A=Case1a7:  |Sub(A)| = 946, that is,
sigma(A) = |Sub(A)|*2^(8-|A|) =  59.1250000000000000 .
   Few subuniverses, Sub-subcase (1a7) above is excluded.
   ! Subcase (1a) is excluded!

~~Case (1): u is not o and u is parallel to g, whence 
u*g=o (C1), and at least one of e*g, f*g is o. 
~~~~Subcase (1b): only e*g=o but f*g=:t>o. 
       Then u*t=o, u*g=o, e*t=o
|A|=11, A(without commas)={oiabcdefgut}. Constraints:
a*d=e c*d=f e+f=d d+g=b e*f=u  a*c=u   e*c=u a*f=u  e+g=b;(#0)
u*g=o;(C1)
e*g=o  f*g=t  u*t=o u*g=o e*t=o;(1b)
Result for A=Case1b:  |Sub(A)| = 612, that is,
sigma(A) = |Sub(A)|*2^(8-|A|) =  76.5000000000000000 .
   Few subuniverses, Subcase (1b) above is excluded.

~~Case (1): u is not o and u is parallel to g, whence 
u*g=o (C1), and at least one of e*g, f*g is o. 
~~~~Subcase (1c): only f*g=o but e*g=:t>o. Then u*g=o u*t=o f*t=o
|A|=11, A(without commas)={oiabcdefgut}. Constraints:
a*d=e c*d=f e+f=d d+g=b e*f=u  a*c=u   e*c=u a*f=u  e+g=b;(#0)
u*g=o;(C1)
f*g=o   e*g=t  u*g=o u*t=o  f*t=o;(1c)
Result for A=Case1c:  |Sub(A)| = 620, that is,
sigma(A) = |Sub(A)|*2^(8-|A|) =  77.5000000000000000 .
   Few subuniverses, Subcase (1c) above is excluded.
   Case (1) is excluded! 

Remark: since u>=g is impossible, and u<g would mean u=o, and
         u parallel g has been excluded, we get that u=o.
~~Case (2): e*g=u and f*g=u.
           (So u<g, u is the bottom, and still e*f=u)
~~~~~~Subcase (2a) a+b=:i a+c=i b+c=i      
|A|=10, A(without commas)={oiabcdefgu}. Constraints:
a*d=e c*d=f e+f=d d+g=b e*f=u  a*c=u   e*c=u a*f=u  e+g=b;(#0)
e*g=u  f*g=u;(Case2)
 a+b=i a+c=i b+c=i;(2a)
Result for A=Case2a:  |Sub(A)| = 324, that is,
sigma(A) = |Sub(A)|*2^(8-|A|) =  81.0000000000000000 .
   Few subuniverses, Subcase (2a) above is excluded.

~~Case (2): e*g=u and f*g=u.
           (So u<g, u is the bottom, 
            and still e*f=u)
~~~~~~Subcase (2b) a+b=:i a+c=i b+c=:x<i      
|A|=11, A(without commas)={oiabcdefgux}. Constraints:
a*d=e c*d=f e+f=d d+g=b e*f=u  a*c=u   e*c=u a*f=u  e+g=b;(#0)
e*g=u  f*g=u;(Case2)
a+b=i a+c=i b+c=x;(2b)
Result for A=Case2b:  |Sub(A)| = 632, that is,
sigma(A) = |Sub(A)|*2^(8-|A|) =  79.0000000000000000 .
  Few subuniverses, Subcase (2b) above is excluded.

~~Case (2): e*g=u and f*g=u. (So u<g, u is the bottom, 
            and still e*f=u)
~~~~~~Subcase (2c) a+b=:i a+c=x<i  b+c=i      
|A|=11, A(without commas)={oiabcdefgux}. Constraints:
a*d=e c*d=f e+f=d d+g=b e*f=u  a*c=u   e*c=u a*f=u  e+g=b;(#0)
e*g=u  f*g=u;(Case2)
a+b=i a+c=x  b+c=i;(2c)
Result for A=Case2c:  |Sub(A)| = 632, that is,
sigma(A) = |Sub(A)|*2^(8-|A|) =  79.0000000000000000 .
  Few subuniverses, Subcase (2c) above is excluded.

~~Case (2): e*g=u and f*g=u. (So u<g, u is the bottom, 
            and still e*f=u)
~~~~~~Subcase (2d) a+b=:x<i  a+c=i b+c=i      
|A|=11, A(without commas)={oiabcdefgux}. Constraints:
a*d=e c*d=f e+f=d d+g=b e*f=u  a*c=u   e*c=u a*f=u  e+g=b;(#0)
e*g=u  f*g=u;(Case2)
 a+b=x  a+c=i b+c=i;(2d)
Result for A=Case2d:  |Sub(A)| = 632, that is,
sigma(A) = |Sub(A)|*2^(8-|A|) =  79.0000000000000000 .
  Few subuniverses, Subcase (2d) above is excluded.

~~Case (2): e*g=u and f*g=u.
            (so u<g, u is the bottom, and still e*f=u)
~~~~~~Subcase (2e) a+b=:x<i  a+c=:y<i  b+c=i    
           Remark: we do not have to use x+y=i !  
|A|=12, A(without commas)={oiabcdefguxy}. Constraints:
a*d=e c*d=f e+f=d d+g=b e*f=u  a*c=u   e*c=u a*f=u  e+g=b;(#0)
e*g=u  f*g=u;(Case2)
 a+b=x  a+c=y  b+c=i;(2e)
Result for A=Case2e:  |Sub(A)| = 1248, that is,
sigma(A) = |Sub(A)|*2^(8-|A|) =  78.0000000000000000 .
  Few subuniverses, Subcase (2e) above is excluded.

~~Case (2): e*g=u and f*g=u.
            (so u<g, u is the bottom, and still e*f=u)
~~~~~~Subcase (2f) a+b=:x<i  a+c=i  b+c:=y<i    
|A|=12, A(without commas)={oiabcdefguxy}. Constraints:
a*d=e c*d=f e+f=d d+g=b e*f=u  a*c=u   e*c=u a*f=u  e+g=b;(#0)
e*g=u  f*g=u;(Case2)
a+b=x  a+c=i  b+c=y;(2f)
Result for A=Case2f:  |Sub(A)| = 1248, that is,
sigma(A) = |Sub(A)|*2^(8-|A|) =  78.0000000000000000 .
  Few subuniverses, Subcase (2f) above is excluded.

~~Case (2): e*g=u and f*g=u.
            (so u<g, u is the bottom, and still e*f=u)
~~~~~~Subcase (2g) a+b=i  a+c=:x<i  b+c:=y<i    
|A|=12, A(without commas)={oiabcdefguxy}. Constraints:
a*d=e c*d=f e+f=d d+g=b e*f=u  a*c=u   e*c=u a*f=u  e+g=b;(#0)
e*g=u  f*g=u;(Case2)
a+b=i  a+c=x  b+c=y;(2g)
Result for A=Case2g:  |Sub(A)| = 1248, that is,
sigma(A) = |Sub(A)|*2^(8-|A|) =  78.0000000000000000 .
  Few subuniverses, Subcase (2g) above is excluded.
  Since no more possibility for the joins of a,b,c, we get that
  Case (2) is excluded.

~~Case (3): e*g=u and f*g:=v>u.
            (u=e*f is still the bottom) Clearly, e*v=u.
~~~~~~Subcase (3a) a+b=:i a+c=i b+c=i      
|A|=11, A(without commas)={oiabcdefguv}. Constraints:
a*d=e c*d=f e+f=d d+g=b e*f=u  a*c=u   e*c=u a*f=u  e+g=b;(#0)
e*g=u  f*g=v  e*v=u;(Case3)
a+b=i a+c=i b+c=i;(3a)
Result for A=Case3a:  |Sub(A)| = 600, that is,
sigma(A) = |Sub(A)|*2^(8-|A|) =  75.0000000000000000 .
  Few subuniverses, Subcase (3a) above is excluded.

~~Case (3): e*g=u and f*g:=v>u.
            (u=e*f is still the bottom) Clearly, e*v=u.
~~~~~~Subcase (3b) a+b=x a+c=i b+c=i. (x<i)
|A|=12, A(without commas)={oiabcdefguvx}. Constraints:
a*d=e c*d=f e+f=d d+g=b e*f=u  a*c=u   e*c=u a*f=u  e+g=b;(#0)
e*g=u  f*g=v  e*v=u;(Case3)
 a+b=x a+c=i b+c=i;(3b)
Result for A=Case3b:  |Sub(A)| = 1172, that is,
sigma(A) = |Sub(A)|*2^(8-|A|) =  73.2500000000000000 .
  Few subuniverses, Subcase (3b) above is excluded.

~~Case (3): e*g=u and f*g:=v>u.
            (u=e*f is still the bottom) Clearly, e*v=u.
~~~~~~Subcase (3c) a+b=i a+c=x b+c=i. (x<i)
|A|=12, A(without commas)={oiabcdefguvx}. Constraints:
a*d=e c*d=f e+f=d d+g=b e*f=u  a*c=u   e*c=u a*f=u  e+g=b;(#0)
e*g=u  f*g=v  e*v=u;(Case3)
a+b=i a+c=x b+c=i;(3c)
Result for A=Case3c:  |Sub(A)| = 1172, that is,
sigma(A) = |Sub(A)|*2^(8-|A|) =  73.2500000000000000 .
  Few subuniverses, Subcase (3c) above is excluded.

~~Case (3): e*g=u and f*g:=v>u.
            (u=e*f is still the bottom) Clearly, e*v=u.
~~~~~~Subcase (3d) a+b=i a+c=i b+c=x. (x<i)
|A|=12, A(without commas)={oiabcdefguvx}. Constraints:
a*d=e c*d=f e+f=d d+g=b e*f=u  a*c=u   e*c=u a*f=u  e+g=b;(#0)
e*g=u  f*g=v  e*v=u;(Case3)
a+b=i a+c=i b+c=x;(3d)
Result for A=Case3d:  |Sub(A)| = 1172, that is,
sigma(A) = |Sub(A)|*2^(8-|A|) =  73.2500000000000000 .
  Few subuniverses, Subcase (3d) above is excluded.

~~Case (3): e*g=u and f*g:=v>u.
            (u=e*f is still the bottom) Clearly, e*v=u.
~~~~~~Subcase (3e) a+b=x a+c=y b+c=i. (x,y<i)
|A|=13, A(without commas)={oiabcdefguvxy}. Constraints:
a*d=e c*d=f e+f=d d+g=b e*f=u  a*c=u   e*c=u a*f=u  e+g=b;(#0)
e*g=u  f*g=v  e*v=u;(Case3)
a+b=x a+c=y b+c=i;(3e)
Result for A=Case3e:  |Sub(A)| = 2316, that is,
sigma(A) = |Sub(A)|*2^(8-|A|) =  72.3750000000000000 .
  Few subuniverses, Subcase (3e) above is excluded.

~~Case (3): e*g=u and f*g:=v>u.
            (u=e*f is still the bottom) Clearly, e*v=u.
~~~~~~Subcase (3f) a+b=x a+c=i b+c=y. (x,y<i)
|A|=13, A(without commas)={oiabcdefguvxy}. Constraints:
a*d=e c*d=f e+f=d d+g=b e*f=u  a*c=u   e*c=u a*f=u  e+g=b;(#0)
e*g=u  f*g=v  e*v=u;(Case3)
a+b=x a+c=i b+c=y;(3f)
Result for A=Case3f:  |Sub(A)| = 2316, that is,
sigma(A) = |Sub(A)|*2^(8-|A|) =  72.3750000000000000 .
  Few subuniverses, Subcase (3f) above is excluded.

~~Case (3): e*g=u and f*g:=v>u.
            (u=e*f is still the bottom) Clearly, e*v=u.
~~~~~~Subcase (3g) a+b=i a+c=x b+c=y. (x,y<i)
|A|=13, A(without commas)={oiabcdefguvxy}. Constraints:
a*d=e c*d=f e+f=d d+g=b e*f=u  a*c=u   e*c=u a*f=u  e+g=b;(#0)
e*g=u  f*g=v  e*v=u;(Case3)
a+b=i a+c=x b+c=y;(3g)
Result for A=Case3g:  |Sub(A)| = 2316, that is,
sigma(A) = |Sub(A)|*2^(8-|A|) =  72.3750000000000000 .
  Few subuniverses, Subcase (3g) above is excluded.
           No more possibility for the joins of a,b,c, whence
  Case (3) is excluded.

~~Case (4): e*g=v>u and f*g:=u.
            (u=e*f is still the bottom) Clearly, f*v=u.
~~~~~~Subcase (4a) a+b=:i a+c=i b+c=i      
|A|=11, A(without commas)={oiabcdefguv}. Constraints:
a*d=e c*d=f e+f=d d+g=b e*f=u  a*c=u   e*c=u a*f=u  e+g=b;(#0)
e*g=v  f*g=u  f*v=u;(Case4)
a+b=i a+c=i b+c=i;(4a)
Result for A=Case4a:  |Sub(A)| = 606, that is,
sigma(A) = |Sub(A)|*2^(8-|A|) =  75.7500000000000000 .
  Few subuniverses, Subcase (4a) above is excluded.

~~Case (4): e*g=v>u and f*g:=u.
            (u=e*f is still the bottom) Clearly, f*v=u.
~~~~~~Subcase (4b) a+b=x a+c=i b+c=i. (x<i)
|A|=12, A(without commas)={oiabcdefguvx}. Constraints:
a*d=e c*d=f e+f=d d+g=b e*f=u  a*c=u   e*c=u a*f=u  e+g=b;(#0)
e*g=v  f*g=u  f*v=u;(Case4)
 a+b=x a+c=i b+c=i;(4b)
Result for A=Case4b:  |Sub(A)| = 1184, that is,
sigma(A) = |Sub(A)|*2^(8-|A|) =  74.0000000000000000 .
  Few subuniverses, Subcase (4b) above is excluded.

~~Case (4): e*g=v>u and f*g:=u.
            (u=e*f is still the bottom) Clearly, f*v=u.
~~~~~~Subcase (4c) a+b=i a+c=x b+c=i. (x<i)
|A|=12, A(without commas)={oiabcdefguvx}. Constraints:
a*d=e c*d=f e+f=d d+g=b e*f=u  a*c=u   e*c=u a*f=u  e+g=b;(#0)
e*g=v  f*g=u  f*v=u;(Case4)
a+b=i a+c=x b+c=i;(4c)
Result for A=Case4c:  |Sub(A)| = 1184, that is,
sigma(A) = |Sub(A)|*2^(8-|A|) =  74.0000000000000000 .
  Few subuniverses, Subcase (4c) above is excluded.

~~Case (4): e*g=v>u and f*g:=u.
            (u=e*f is still the bottom) Clearly, f*v=u.
~~~~~~Subcase (4d) a+b=i a+c=i b+c=x. (x<i)
|A|=12, A(without commas)={oiabcdefguvx}. Constraints:
a*d=e c*d=f e+f=d d+g=b e*f=u  a*c=u   e*c=u a*f=u  e+g=b;(#0)
e*g=v  f*g=u  f*v=u;(Case4)
a+b=i a+c=i b+c=x;(4d)
Result for A=Case4d:  |Sub(A)| = 1184, that is,
sigma(A) = |Sub(A)|*2^(8-|A|) =  74.0000000000000000 .
  Few subuniverses, Subcase (4d) above is excluded.

~~Case (4): e*g=v>u and f*g:=u.
            (u=e*f is still the bottom) Clearly, f*v=u.
~~~~~~Subcase (4e) a+b=x a+c=y b+c=i. (x,y<i)
|A|=13, A(without commas)={oiabcdefguvxy}. Constraints:
a*d=e c*d=f e+f=d d+g=b e*f=u  a*c=u   e*c=u a*f=u  e+g=b;(#0)
e*g=v  f*g=u  f*v=u;(Case4)
a+b=x a+c=y b+c=i;(4e)
Result for A=Case4e:  |Sub(A)| = 2340, that is,
sigma(A) = |Sub(A)|*2^(8-|A|) =  73.1250000000000000 .
  Few subuniverses, Subcase (4e) above is excluded.

~~Case (4): e*g=v>u and f*g:=u.
            (u=e*f is still the bottom) Clearly, f*v=u.
~~~~~~Subcase (4f) a+b=x a+c=i b+c=y. (x,y<i)
|A|=13, A(without commas)={oiabcdefguvxy}. Constraints:
a*d=e c*d=f e+f=d d+g=b e*f=u  a*c=u   e*c=u a*f=u  e+g=b;(#0)
e*g=v  f*g=u  f*v=u;(Case4)
a+b=x a+c=i b+c=y;(4f)
Result for A=Case4f:  |Sub(A)| = 2340, that is,
sigma(A) = |Sub(A)|*2^(8-|A|) =  73.1250000000000000 .
  Few subuniverses, Subcase (4f) above is excluded.

~~Case (4): e*g=v>u and f*g:=u.
            (u=e*f is still the bottom) Clearly, f*v=u.
~~~~~~Subcase (4g) a+b=i a+c=x b+c=y. (x,y<i)
|A|=13, A(without commas)={oiabcdefguvxy}. Constraints:
a*d=e c*d=f e+f=d d+g=b e*f=u  a*c=u   e*c=u a*f=u  e+g=b;(#0)
e*g=v  f*g=u  f*v=u;(Case4)
a+b=i a+c=x b+c=y;(4g)
Result for A=Case4g:  |Sub(A)| = 2340, that is,
sigma(A) = |Sub(A)|*2^(8-|A|) =  73.1250000000000000 .
  Few subuniverses, Subcase (4g) above is excluded.
           No more possibility for the joins of a,b,c, whence
  Case (4) is excluded.

~~Case (5): e*g=v>u and f*g:=w>u.
            (e*f=u is still the bottom;  e*w=u, f*v=u).
~~~~~~Subcase (5a) a+b=:i a+c=i b+c=i      
|A|=12, A(without commas)={oiabcdefguvw}. Constraints:
a*d=e c*d=f e+f=d d+g=b e*f=u  a*c=u   e*c=u a*f=u  e+g=b;(#0)
e*g=v  f*g=w  e*w=u  f*v=u;(Case5)
a+b=i a+c=i b+c=i;(5a)
Result for A=Case5a:  |Sub(A)| = 1094, that is,
sigma(A) = |Sub(A)|*2^(8-|A|) =  68.3750000000000000 .
  Few subuniverses, Subcase (5a) above is excluded.

~~Case (5): e*g=v>u and f*g:=w>u.
            (e*f=u is still the bottom;  e*w=u, f*v=u).
~~~~~~Subcase (5b) a+b=x a+c=i b+c=i. (x<i)
|A|=13, A(without commas)={oiabcdefguvwx}. Constraints:
a*d=e c*d=f e+f=d d+g=b e*f=u  a*c=u   e*c=u a*f=u  e+g=b;(#0)
e*g=v  f*g=w  e*w=u  f*v=u;(Case5)
 a+b=x a+c=i b+c=i;(5b)
Result for A=Case5b:  |Sub(A)| = 2138, that is,
sigma(A) = |Sub(A)|*2^(8-|A|) =  66.8125000000000000 .
  Few subuniverses, Subcase (5b) above is excluded.

~~Case (5): e*g=v>u and f*g:=w>u.
            (e*f=u is still the bottom;  e*w=u, f*v=u).
~~~~~~Subcase (5c) a+b=i a+c=x b+c=i. (x<i)
|A|=13, A(without commas)={oiabcdefguvwx}. Constraints:
a*d=e c*d=f e+f=d d+g=b e*f=u  a*c=u   e*c=u a*f=u  e+g=b;(#0)
e*g=v  f*g=w  e*w=u  f*v=u;(Case5)
a+b=i a+c=x b+c=i;(5c)
Result for A=Case5c:  |Sub(A)| = 2138, that is,
sigma(A) = |Sub(A)|*2^(8-|A|) =  66.8125000000000000 .
  Few subuniverses, Subcase (5c) above is excluded.

~~Case (5): e*g=v>u and f*g:=w>u.
            (e*f=u is still the bottom;  e*w=u, f*v=u).
~~~~~~Subcase (5d) a+b=i a+c=i b+c=x. (x<i)
|A|=13, A(without commas)={oiabcdefguvwx}. Constraints:
a*d=e c*d=f e+f=d d+g=b e*f=u  a*c=u   e*c=u a*f=u  e+g=b;(#0)
e*g=v  f*g=w  e*w=u  f*v=u;(Case5)
a+b=i a+c=i b+c=x;(5d)
Result for A=Case5d:  |Sub(A)| = 2138, that is,
sigma(A) = |Sub(A)|*2^(8-|A|) =  66.8125000000000000 .
  Few subuniverses, Subcase (5d) above is excluded.

~~Case (5): e*g=v>u and f*g:=w>u.
            (e*f=u is still the bottom;  e*w=u, f*v=u).
~~~~~~Subcase (5e) a+b=x a+c=y b+c=i. (x,y<i)
|A|=14, A(without commas)={oiabcdefguvwxy}. Constraints:
a*d=e c*d=f e+f=d d+g=b e*f=u  a*c=u   e*c=u a*f=u  e+g=b;(#0)
e*g=v  f*g=w  e*w=u  f*v=u;(Case5)
a+b=x a+c=y b+c=i;(5e)
Result for A=Case5e:  |Sub(A)| = 4226, that is,
sigma(A) = |Sub(A)|*2^(8-|A|) =  66.0312500000000000 .
  Few subuniverses, Subcase (5e) above is excluded.

~~Case (5): e*g=v>u and f*g:=w>u.
            (e*f=u is still the bottom;  e*w=u, f*v=u).
~~~~~~Subcase (5f) a+b=x a+c=i b+c=y. (x,y<i)
|A|=14, A(without commas)={oiabcdefguvwxy}. Constraints:
a*d=e c*d=f e+f=d d+g=b e*f=u  a*c=u   e*c=u a*f=u  e+g=b;(#0)
e*g=v  f*g=w  e*w=u  f*v=u;(Case5)
a+b=x a+c=i b+c=y;(5f)
Result for A=Case4f:  |Sub(A)| = 4226, that is,
sigma(A) = |Sub(A)|*2^(8-|A|) =  66.0312500000000000 .
  Few subuniverses, Subcase (5f) above is excluded.

~~Case (5): e*g=v>u and f*g:=w>u.
            (e*f=u is still the bottom;  e*w=u, f*v=u).
~~~~~~Subcase (5g) a+b=i a+c=x b+c=y. (x,y<i)
|A|=14, A(without commas)={oiabcdefguvwxy}. Constraints:
a*d=e c*d=f e+f=d d+g=b e*f=u  a*c=u   e*c=u a*f=u  e+g=b;(#0)
e*g=v  f*g=w  e*w=u  f*v=u;(Case5)
a+b=i a+c=x b+c=y;(5g)
Result for A=Case5g:  |Sub(A)| = 4226, that is,
sigma(A) = |Sub(A)|*2^(8-|A|) =  66.0312500000000000 .
  Few subuniverses, Subcase (5g) above is excluded.
           No more possibility for the joins of a,b,c, whence
  Case (5) is excluded.
There are no more cases. Therefore, E_0 cannot be subposet of L,
as required.

The computation took 172/1000 seconds.
\end{verbatim}
\normalfont

\clearpage
\def\appendixhead{{Cz\'edli: Eighty-three sublattices / Appendix: $E_1$}}
\markboth\appendixhead\appendixhead
\centerline{\bf{{\LARGE Appendix: $E_1$}}}
\normalfont
\begin{verbatim}
Version of January 1, 2019; reformatted on May 7, 2019
E_1 from Kelly-Rival: "Planar lattices"; its edges are
of og oh oj fb fc gc gd hd he ja bi ca da ei ai  
Suppose, for contradiction, that  E_1 is a subposet of L
and L has many subuniverses.
Note the symmetry: the reflection accross the o-g-a-j line 
except for the fixed point j

First, we can always assume (after lifting some elements):
(always) b*c=f c*d=g d*e=h\w Then, in after lowering, 
(always) f+g=c g+h=d      
Second, a is the join of two of c,d,j; 
        i is the join of two of b,a,e,  (#1)
Let o:=f*g*h;  maybe 0:=o*j<0 !!!, note that
o=f*h since f*h=f*c*d*h=f*g*h (#2)
SUBSIZE version Dec 29, 2018 (started at 0:12:40) reports:
 [ Supported by the Hungarian Research Grant KH 126581,
                                (C) Gabor Czedli, 2018 ]


|A|=11, A(without commas)={ioabcdefghj}. Constraints:
(always) b*c=f c*d=g d*e=h; Then  after lowering
(always) f+g=c g+h=d
(always  see #2) f*h=o; Symmetric case
 (C1)c+d=a; Symmetric case
    (then)f+d=a c+h=a;since f+d=f+g+d=c+d c+h=c+g+h=c+d
  (C1a) f*g=o
    (then) b*g=o f*d=o;since b*g=b*c*g=f*g  f*d=f*c*d=f*g
   (C1a1) g*h=o
     (then) g*e=o c*h=o;since g*e=g*d*e=g*h  c*h=c*d*h=g*h
Result for A=(E1/C1a1):  |Sub(A)| = 656, that is,
sigma(A) = |Sub(A)|*2^(8-|A|) =  82.0000000000000000 .
Few subuniverses, (C1a1) is excluded.

|A|=12, A(without commas)={ioabcdefghjx}. Constraints:
(always) b*c=f c*d=g d*e=h; Then  after lowering
(always) f+g=c g+h=d; Symmetric case
(always  see #2) f*h=o; Symmetric case
 (C1)c+d=a; Symmetric case
    (then)f+d=a c+h=a;since f+d=f+g+d=c+d c+h=c+g+h=c+d
  (C1a) f*g=o
    (then) b*g=o f*d=o;since b*g=b*c*g=f*g  f*d=f*c*d=f*g
   (C1a2) g*h=x; new x>o
     (then) f*x=o; since f*h=o
     (and) b*x=o; since b*x=b*c*x=f*x
Result for A=(E1/C1a2):  |Sub(A)| = 1208, that is,
sigma(A) = |Sub(A)|*2^(8-|A|) =  75.5000000000000000 .
Few subuniverses, (C1a2) is excluded.
Thus, (C1a) is excluded. Since (C1) is symmetric,
g*h=o is also excluded! Furthermore, since f*g*h=o,
p:=f*g and q:=g*h are parallel !

|A|=13, A(without commas)={ioabcdefghjpq}. Constraints:
(always) b*c=f c*d=g d*e=h; Then  after lowering
(always) f+g=c g+h=d; Symmetric case
(always  see #2) f*h=o; Symmetric case
 (C1)c+d=a; Symmetric case
    (then)f+d=a c+h=a;since f+d=f+g+d=c+d c+h=c+g+h=c+d
  (C1b) f*g=p f*h=q; see after (C1a2) above
    (then) f*q=o p*h=o p*q=o; since #2
Result for A=(E1/C1b):  |Sub(A)| = 2292, that is,
sigma(A) = |Sub(A)|*2^(8-|A|) =  71.6250000000000000 .
Few subuniverses, (C1b) is excluded.
Thus, (C1) is excluded.

|A|=12, A(without commas)={ioabcdefghjr}. Constraints:
(always) b*c=f c*d=g d*e=h; Then  after lowering
(always) f+g=c g+h=d
(always  see #2) f*h=o; Symmetric case
 (C2)c+d=r; r<a  symmetric case
    (then)f+d=r c+h=r;since f+d=f+g+d=c+d c+h=c+g+h=c+d
    (and) r+j=a
  (C2a) f*g=o
    (then) b*g=o f*d=o;since b*g=b*c*g=f*g  f*d=f*c*d=f*g
Result for A=(E1/C2a):  |Sub(A)| = 1206, that is,
sigma(A) = |Sub(A)|*2^(8-|A|) =  75.3750000000000000 .
Few subuniverses, (C2a) is excluded. But (C2) is symmetric,
whence g*h=o is also excluded. Thus, p:=f*g and q:=g*h
are parallel and (#2) applies !!! 

|A|=14, A(without commas)={ioabcdefghjrpq}. Constraints:
(always) b*c=f c*d=g d*e=h; Then  after lowering
(always) f+g=c g+h=d
(always  see #2) f*h=o; Symmetric case
 (C2)c+d=r; r<a  symmetric case
    (then)f+d=r c+h=r;since f+d=f+g+d=c+d c+h=c+g+h=c+d
    (and) r+j=a
  (C2b) f*g=p g*h=q; see above  after (C2a)
    (then) f*q=o p*h=o p*q=o; by #2
Result for A=(E1/C2b):  |Sub(A)| = 3946, that is,
sigma(A) = |Sub(A)|*2^(8-|A|) =  61.6562500000000000 .
Few subuniverses, (C2b) is excluded.
Thus, (C2) is excluded.
All cases have been excluded, q.e.d.

The computation took 32/1000 seconds.
\end{verbatim}
\normalfont

\clearpage
\def\appendixhead{{Cz\'edli: Eighty-three sublattices / Appendix: Eight-element fence}}
\markboth\appendixhead\appendixhead
\centerline{\bf{{\LARGE Appendix: Eight-element fence}}}
\normalfont
\begin{verbatim}
Version of December 31, 2018, revised May 7, 2019
Fence_8 is the 8-element fence; its edges are 
(we write xy to denote that y covers x):
ab cb cd ed ef gf gh and we add o and i with oc oe and  di fi
|A|=10, A(without commas)={abcdefghoi}. Constraints:
a+c=b c+e=d e+g=f; the joins exist only for neighboring atoms
b*d=c d*f=e f*h=g; the meets exist only for neighboring coatoms
 (also) d+f=i c*e=o; since we define d+f=:i and  c*e=:o
   (so) b*e=o c*f=o; since b*e=b*d*e=c*e=o and c*f=c*d*f=c*e=o
   (and) d+g=i f+c=i; since d+g=d+e+g=d+f=i and f+c=f+e+c=f+d=i
Result for A=fence_8:  |Sub(A)| = 338, that is,
sigma(A) = |Sub(A)|*2^(8-|A|) =  84.5000000000000000 .
We are not ready yet.

Fence_8 is the 8-element fence; its edges are 
  (we write xy to denote that y covers x):
ab cb cd ed ef gf fh and we add o and i with oc oe and  di fi
|A|=10, A(without commas)={abcdefghoi}. Constraints:
a+c=b c+e=d e+g=f; the joins exist only for neighboring atoms
b*d=c d*f=e f*h=g; the meets exist only for neighboring atoms
 (also) d+f=i c*e=o; since we define d+f=:i and  c*e=:o
   (so) b*e=o c*f=o; since b*e=b*d*e=c*e=o and c*f=c*d*f=c*e=o
   (and) d+g=i f+c=i; since d+g=d+e+g=d+f=i and f+c=f+e+c=f+d=i
 (C1) b+d=i
    (so) b+e=i; since b+e=b+c+e=b+d=i
Result for A=fence_8:  |Sub(A)| = 316, that is,
sigma(A) = |Sub(A)|*2^(8-|A|) =  79.0000000000000000 .
This case is excluded.

Fence_8 is the 8-element fence; its edges are 
  (we write xy to denote that y covers x):
ab cb cd ed ef gf fh and we add o and i with oc oe and  di fi
|A|=11, A(without commas)={abcdefghoix}. Constraints:
a+c=b c+e=d e+g=f; the joins exist only for neighboring atoms
b*d=c d*f=e f*h=g; the meets exist only for neighboring atoms
 (also) d+f=i c*e=o; since we define d+f=:i and  c*e=:o
   (so) b*e=o c*f=o; since b*e=b*d*e=c*e=o and c*f=c*d*f=c*e=o
   (and) d+g=i f+c=i; since d+g=d+e+g=d+f=i and f+c=f+e+c=f+d=i
 (C2) b+d=x; since b+d=:x is distinct from i
   (so) b+e=x; since b+e=b+c+e=b+d=x
Result for A=fence_8:  |Sub(A)| = 618, that is,
sigma(A) = |Sub(A)|*2^(8-|A|) =  77.2500000000000000 .
This case is excluded.
Both cases are excluded, q.e.d.

The computation took 15/1000 seconds.
\end{verbatim}
\normalfont

\clearpage
\def\appendixhead{{Cz\'edli: Eighty-three sublattices / Appendix: $F_0$}}
\markboth\appendixhead\appendixhead
\centerline{\bf{{\LARGE Appendix: $F_0$}}}
\normalfont
\begin{verbatim}
 Version of December 30, 2018
F_0 from Kelly-Rival: "Planar lattices"; its edges are
       ai bi ca da eb ec fe fd gc of og
Here we prove that F_0 is not a subposet of L, provided 
L has many subuniverses. We can and will always assume 
that b*c=e, e+g=c, c+d=a, d*e=f, a+b=i, f*g=o; this is a
selfdual assumption and F_0 is a selfdual lattice. 
There will be several cases, whose assumptions will be 
additional to the ones above.
|A|=9, A(without commas)={oiabcdefg}. Constraints:
(always) b*c=e e+g=c c+d=a d*e=f a+b=i f*g=o; Selfdual!
(C1) b+c=i
  (then) b+g=i; since b+g=b+e+g=b+c
 (C1a) e*g=o; the dual of (C1)
   (then) b*g=o; since b*g=b*c*g=e*g. Selfdual situation!
  (C1a.1) d+e=a
    (then) b+d=i; since b+d=b+e+d=b+a
   (C1a.1a) c*d=f
      (then) d*g=o; since d*g=d*c*g=f*g. Selfdual situation!
    (C1a.1a.1) a*b=e
       (then) b*d=f; since b*d=b*a*d=e*d
     (C1a.1a.1a) f+g=c
        (then) d+g=a; since d+g=d+f+g=d+c
Result for A=(F_0/C1a.1a.1a):  |Sub(A)| = 166, that is,
sigma(A) = |Sub(A)|*2^(8-|A|) =  83.0000000000000000 .
   Few subuniverses, (C1a.1a.1a) is excluded.

|A|=10, A(without commas)={oiabcdefgx}. Constraints:
(always) b*c=e e+g=c c+d=a d*e=f a+b=i f*g=o; see the paper
(C1) b+c=i
  (then) b+g=i; since b+g=b+e+g=b+c
 (C1a) e*g=o; the dual of (C1)
   (then) b*g=o; since b*g=b*c*g=e*g. Selfdual situation!
  (C1a.1) d+e=a
    (then) b+d=i; since b+d=b+e+d=b+a
   (C1a.1a) c*d=f
      (then) d*g=o; since d*g=d*c*g=f*g. Selfdual situation!
    (C1a.1a.1) a*b=e
       (then) b*d=f; since b*d=b*a*d=e*d
     (C1a.1a.1b) f+g=x ; x<c
        (then) e+x=c; since c=e+g<=e+x<=c
Result for A=(F_0/1a.1a.1b):  |Sub(A)| = 297, that is,
sigma(A) = |Sub(A)|*2^(8-|A|) =  74.2500000000000000 .
   Few subuniverses, (C1a.1a.1b) is excluded.
   Thus, (C1a.1a.1) is excluded. Since (C1a.1a) is
   seldfual, the dual of (C1a.1a.1) is also excluded (*1)

|A|=11, A(without commas)={oiabcdefgxy}. Constraints:
(always) b*c=e e+g=c c+d=a d*e=f a+b=i f*g=o; see the paper
(C1) b+c=i
  (then) b+g=i; since b+g=b+e+g=b+c
 (C1a) e*g=o; the dual of (C1)
   (then) b*g=o; since b*g=b*c*g=e*g. Selfdual situation!
  (C1a.1) d+e=a
    (then) b+d=i; since b+d=b+e+d=b+a
   (C1a.1a) c*d=f
      (then) d*g=o; since d*g=d*c*g=f*g. Selfdual situation!
    (C1a.1a.2) a*b=x f+g=y; x>e  y<c  (*1) is used.
       (then) c*x=e e+y=c ; since c*b=e and e+g=c
Result for A=(F_0/1a.1a.2):  |Sub(A)| = 544, that is,
sigma(A) = |Sub(A)|*2^(8-|A|) =  68.0000000000000000 .
   Few subuniverses, (C1a.1a.2) is excluded.
   Thus, (C1a.1a) is excluded.

|A|=10, A(without commas)={oiabcdefgx}. Constraints:
(always) b*c=e e+g=c c+d=a d*e=f a+b=i f*g=o; see the paper
(C1) b+c=i
  (then) b+g=i; since b+g=b+e+g=b+c
 (C1a) e*g=o; the dual of (C1)
   (then) b*g=o; since b*g=b*c*g=e*g. Selfdual situation!
  (C1a.1) d+e=a
    (then) b+d=i; since b+d=b+e+d=b+a
   (C1a.1b) c*d=x ; x>f
     (then) e*x=f ; since e*d=f
    (C1a.1b.1) x*g=o
      (then) d*g=o ; since d*g=g*c*d=g*x=o
     (C1a.1b.1a) a*b=e
Result for A=(F_0/1a.1b.1a):  |Sub(A)| = 313, that is,
sigma(A) = |Sub(A)|*2^(8-|A|) =  78.2500000000000000 .
   Few subuniverses, (C1a.1b.1a) is excluded.
   

|A|=11, A(without commas)={oiabcdefgxy}. Constraints:
(always) b*c=e e+g=c c+d=a d*e=f a+b=i f*g=o; see the paper
(C1) b+c=i
  (then) b+g=i; since b+g=b+e+g=b+c
 (C1a) e*g=o; the dual of (C1)
   (then) b*g=o; since b*g=b*c*g=e*g. Selfdual situation!
  (C1a.1) d+e=a
    (then) b+d=i; since b+d=b+e+d=b+a
   (C1a.1b) c*d=x ; x>f  x not>g
     (then) e*x=f ; since e*d=f
    (C1a.1b.1) x*g=o
      (then) d*g=o ; since d*g=g*c*d=g*x=o
     (C1a.1b.1b) a*b=y ; y>e
Result for A=(F_0/1a.1b.1b):  |Sub(A)| = 628, that is,
sigma(A) = |Sub(A)|*2^(8-|A|) =  78.5000000000000000 .
   Few subuniverses, (C1a.1b.1b) is excluded.
   Thus, (C1a.1b.1) is excluded. 

|A|=11, A(without commas)={oiabcdefgxy}. Constraints:
(always) b*c=e e+g=c c+d=a d*e=f a+b=i f*g=o; see the paper
(C1) b+c=i
  (then) b+g=i; since b+g=b+e+g=b+c
 (C1a) e*g=o; the dual of (C1)
   (then) b*g=o; since b*g=b*c*g=e*g. Selfdual situation!
  (C1a.1) d+e=a
    (then) b+d=i; since b+d=b+e+d=b+a
   (C1a.1b) c*d=x ; x>f  x not>g
     (then) e*x=f ; since e*d=f
    (C1a.1b.2) x*g=y ; g>y>o  so y is a new element
      (then) d*g=y ; since d*g=d*c*g=x*g=y
Result for A=(F_0/1a.1b.2):  |Sub(A)| = 635, that is,
sigma(A) = |Sub(A)|*2^(8-|A|) =  79.3750000000000000 .
   Few subuniverses, (C1a.1b.2) is excluded.
   Thus, (C1a.1b) is excluded. 
   Thus, (C1a.1) is excluded.   Since (C1a) is selfdual,
   the dual of (C1a.1) is also excluded!  (*2)

|A|=12, A(without commas)={oiabcdefguvx}. Constraints:
(always) b*c=e e+g=c c+d=a d*e=f a+b=i f*g=o; see the paper
(C1) b+c=i
  (then) b+g=i; since b+g=b+e+g=b+c
 (C1a) e*g=o; the dual of (C1)
   (then) b*g=o; since b*g=b*c*g=e*g. Selfdual situation!
  (C1a.2) d+e=v c*d=u; v<a  u>f  (*2) was used.
    (then) c+v=a e*u=f; since c+d=a and e*d=f
    (and) e+u=x d+x=v b*x=e  ; since x:=e+u<=c*v
Result for A=(F_0/1a.2):  |Sub(A)| = 1056, that is,
sigma(A) = |Sub(A)|*2^(8-|A|) =  66.0000000000000000 .
   Few subuniverses, (C1a.2) is excluded.
   Thus, (C1a) is excluded.

|A|=10, A(without commas)={oiabcdefgx}. Constraints:
(always) b*c=e e+g=c c+d=a d*e=f a+b=i f*g=o; see the paper
(C1) b+c=i
  (then) b+g=i; since b+g=b+e+g=b+c
 (C1b) e*g=x; x>o
   (then) f*x=o ; since f*g=o
  (C1b.1) d+e=a
    (then) b+d=i; since b+d=b+e+d=b+a
   (C1b.1a) c*d=f
      (then) d*g=o; since d*g=d*c*g=f*g
    (C1b.1a.1) a*b=e
       (then) b*d=f; since b*d=b*a*d=e*d
Result for A=(F_0/1b.1a.1):  |Sub(A)| = 313, that is,
sigma(A) = |Sub(A)|*2^(8-|A|) =  78.2500000000000000 .
   Few subuniverses, (C1b.1a.1) is excluded.

|A|=11, A(without commas)={oiabcdefgxy}. Constraints:
(always) b*c=e e+g=c c+d=a d*e=f a+b=i f*g=o; see the paper
(C1) b+c=i
  (then) b+g=i; since b+g=b+e+g=b+c
 (C1b) e*g=x; x>o
   (then) f*x=o ; since f*g=o
  (C1b.1) d+e=a
    (then) b+d=i; since b+d=b+e+d=b+a
   (C1b.1a) c*d=f
      (then) d*g=o; since d*g=d*c*g=f*g
    (C1b.1a.2) a*b=y ; y>e
      (then) c*y=e ; since c*b=e
Result for A=(F_0/1b.1a.2):  |Sub(A)| = 576, that is,
sigma(A) = |Sub(A)|*2^(8-|A|) =  72.0000000000000000 .
   Few subuniverses, (C1b.1a.2) is excluded.
   Thus, (C1b.1a) is excluded.

|A|=11, A(without commas)={oiabcdefgxu}. Constraints:
(always) b*c=e e+g=c c+d=a d*e=f a+b=i f*g=o; see the paper
(C1) b+c=i
  (then) b+g=i; since b+g=b+e+g=b+c
 (C1b) e*g=x; x>o
   (then) f*x=o ; since f*g=o
  (C1b.1) d+e=a
    (then) b+d=i; since b+d=b+e+d=b+a
   (C1b.1b) c*d=u ; u>f
     (then) e*u=f ; since e*d=f
Result for A=(F_0/1b.1b):  |Sub(A)| = 640, that is,
sigma(A) = |Sub(A)|*2^(8-|A|) =  80.0000000000000000 .
   Few subuniverses, (C1b.1b) is excluded.
   Thus, (C1b.1) is excluded.

|A|=11, A(without commas)={oiabcdefgxv}. Constraints:
(always) b*c=e e+g=c c+d=a d*e=f a+b=i f*g=o; see the paper
(C1) b+c=i
  (then) b+g=i; since b+g=b+e+g=b+c
 (C1b) e*g=x; x>o
   (then) f*x=o ; since f*g=o
  (C1b.2) d+e=v ; v<a
    (then) c+v=a ; since c+d=a
   (C1b.2a) c*d=f
     (then) d*g=o; since d*g=d*c*g=f*g.
Result for A=(F_0/1b.2a):  |Sub(A)| = 635, that is,
sigma(A) = |Sub(A)|*2^(8-|A|) =  79.3750000000000000 .
   Few subuniverses, (C1b.2a) is excluded.

|A|=12, A(without commas)={oiabcdefgxvu}. Constraints:
(always) b*c=e e+g=c c+d=a d*e=f a+b=i f*g=o; see the paper
(C1) b+c=i
  (then) b+g=i; since b+g=b+e+g=b+c
 (C1b) e*g=x; x>o
   (then) f*x=o ; since f*g=o
  (C1b.2) d+e=v ; v<a
    (then) c+v=a ; since c+d=a
   (C1b.2b) c*d=u ; u>f
     (then) e*u=f ; since e*d=f
Result for A=(F_0/1b.2b):  |Sub(A)| = 1208, that is,
sigma(A) = |Sub(A)|*2^(8-|A|) =  75.5000000000000000 .
   Few subuniverses, (C1b.2b) is excluded.
   Thus, (C1b.2) is excluded.
   Thus, (C1b) is excluded.
   Thus, (C1) is excluded.
   (always) is selfdual, so the dual of (C1) is excluded (*3)

|A|=11, A(without commas)={oiabcdefgst}. Constraints:
(always) b*c=e e+g=c c+d=a d*e=f a+b=i f*g=o; see the paper
(C2) b+c=t e*g=s; t<i  s>o  (*3) was used. Selfdual situation!
  (then) a+t=i f*s=o ; since a+b=i and f*g=o
  (and) b+g=t b*g=s; since b+g=b+e+g=b+c and b*g=b*c*g=e*g
Result for A=(F_0/2):  |Sub(A)| = 660, that is,
sigma(A) = |Sub(A)|*2^(8-|A|) =  82.5000000000000000 .
   Few subuniverses, (C2) is excluded.
   All cases have been excluded, q.e.d.

The computation took 78/1000 seconds.
\end{verbatim}
\normalfont

\clearpage
\def\appendixhead{{Cz\'edli: Eighty-three sublattices / Appendix: $F_0$-alternative}}
\markboth\appendixhead\appendixhead
\centerline{\bf{{\LARGE Appendix: $F_0$-alternative}}}
\normalfont
\begin{verbatim}
Version of November 28, 2018; reformatted May 7, 2019
F_0 from Kelly-Rival: "Planar lattices", an alternative approach 
The edges are ai bi ca da eb ec fe fd gc of og
 To prove that F_0 is not a subposet of L, suppose the contrary.
We can and we will ALWAYS assume that 
         c=e+g, e=b*c, f=d*e, a=c+d, i=a+b, o=f*g
There will be several cases
Letting u:=b+c, v:=c*d, w:=d+g, we have u<=i, v>=f, w<=a.
Clearly, a+u=i, e*v=f, c+w=a. 
We need to deal with Cases (1), (2),  ... (8) below. Namely,
Depending on how many of the elements u, v, and w are new, we 
are going to with Case (1) (=none of these elements are new); 
Cases (2)-(4) (= two are new); Cases (5)-(7) (=exactly one these 
elements are new); and Case (8) (none of them is new).
(We also need to deal with subcases for most of the Cases.)

~~Case (1): each of b+c=:u, c*d=:v, and  d+g=:w is a new element. 
   Observe: a+u=i, e*v=f, c+w=a
|A|=12, A(without commas)={oiabcdefguvw}. Constraints:
e+g=c b*c=e d*e=f c+d=a a+b=i f*g=o
 b+c=u c*d=v d+g=w
  a+u=i e*v=f c+w=a
Result for A=Case1:  |Sub(A)| = 1243, that is,
sigma(A) = |Sub(A)|*2^(8-|A|) =  77.6875000000000000 .
         Few subuniverses, Case (1) above is excluded.

~~Case (2): b+c=:u=i but c*d=:v and d+g=:w are new elements; then 
            e*v=f and c+w=a
~~~~Subcase (2/a):  e+v=t is also a new elements; then, clearly,
       t<c, g+t=c, and d*t=v.
|A|=12, A(without commas)={oiabcdefgvwt}. Constraints:
e+g=c b*c=e d*e=f c+d=a a+b=i f*g=o
 b+c=i c*d=v d+g=w    e*v=f  c+w=a
  e+v=t   g+t=c d*t=v
Result for A=Case2/a:  |Sub(A)| = 1094, that is,
sigma(A) = |Sub(A)|*2^(8-|A|) =  68.3750000000000000 .
         Few subuniverses, Subcase (2/a) above is excluded.

~~Case (2): b+c=i, c*d=:v, and d+g=:w, then  e*v=f and c+w=a
 (In what follows, every element is new unless otherwise stated.)
~~~~Subcase (2/b): t:=e+v=c 
|A|=11, A(without commas)={oiabcdefgvw}. Constraints:
e+g=c b*c=e d*e=f c+d=a a+b=i f*g=o
 b+c=i c*d=v d+g=w   e*v=f c+w=a
  e+v=c
Result for A=Case2/b:  |Sub(A)| = 645, that is,
sigma(A) = |Sub(A)|*2^(8-|A|) =  80.6250000000000000 .
         Few subuniverses, Subcase (2/b) above is excluded.
         Thus, Case (2) is excluded.

~~Case (3): c*d=:v=f but b+c=:u and d+g=_w are new; 
            then a+u=i and c+w=a 
~~~~Subcase (3/a):a*u=:t is new, then c<t,  b+t=u and t+d=a
|A|=12, A(without commas)={oiabcdefguwt}. Constraints:
e+g=c b*c=e d*e=f c+d=a a+b=i f*g=o
 c*d=f b+c=u d+g=w  a+u=i c+w=a
  a*u=t   b+t=u t+d=a
Result for A=Case3/a:  |Sub(A)| = 1036, that is,
sigma(A) = |Sub(A)|*2^(8-|A|) =  64.7500000000000000 .
         Few subuniverses, Subcase (3/a) above is excluded.

~~Case (3): c*d=:v=f but b+c=:u and d+g=:w are new; 
            then a+u=i and c+w=a 
~~~~Subcase (3/b): a*u is not a new element; 
                   then, clearly,  a*u=c.
|A|=11, A(without commas)={oiabcdefguw}. Constraints:
e+g=c b*c=e d*e=f c+d=a a+b=i f*g=o
 c*d=f b+c=u d+g=w  a+u=i c+w=a
  a*u=c
Result for A=Case3/b:  |Sub(A)| = 610, that is,
sigma(A) = |Sub(A)|*2^(8-|A|) =  76.2500000000000000 .
         Few subuniverses, Subcase (3/b) above is excluded.
         Thus, Case (3) is excluded.

~~Case (4): b+c=:u and  c*d=:v are new but w:=d+g=a is not;
            then, clearly, a+u=i and e*v=f.
~~~~Subcase (4/a):   a*u=:t is also a new element. 
                     Then, clearly, b+t=u and t+d=a.
|A|=12, A(without commas)={oiabcdefguvt}. Constraints:
e+g=c b*c=e d*e=f c+d=a a+b=i f*g=o
 b+c=u c*d=v d+g=a   a+u=i e*v=f
  a*u=t   b+t=u  t+d=a
Result for A=Case4/a:  |Sub(A)| = 1057, that is,
sigma(A) = |Sub(A)|*2^(8-|A|) =  66.0625000000000000 .
         Few subuniverses, Subcase (4/a) above is excluded.

~~Case (4): b+c=:u and  c*d=:v are new but w:=d+g=a is not.
            Then,  clearly, a+u=i and e*v=f.
~~~~Subcase (4/b):  t:=a*u is not new. (Clearly, a*u=c).
|A|=11, A(without commas)={oiabcdefguv}. Constraints:
e+g=c b*c=e d*e=f c+d=a a+b=i f*g=o
 b+c=u c*d=v d+g=a   a+u=i  e*v=f
  a*u=c
Result for A=Case4/b:  |Sub(A)| = 601, that is,
sigma(A) = |Sub(A)|*2^(8-|A|) =  75.1250000000000000 .
         Few subuniverses, Subcase (4/b) above is excluded.
         Thus, Case (4) is excluded.

~~Case (5): b+c=:u is new, but v:=c*d=f and w:=d+g=a, 
             then a+u=i.
~~~~Subcase (5/a): a*u=:t is also new, 
                   then c<t gives that  b+t=u and t+d=a.
|A|=11, A(without commas)={oiabcdefgut}. Constraints:
e+g=c b*c=e d*e=f c+d=a a+b=i f*g=o
 b+c=u  c*d=f d+g=a   a+u=i
  a*u=t    b+t=u  t+d=a
Result for A=Case5/a:  |Sub(A)| = 592, that is,
sigma(A) = |Sub(A)|*2^(8-|A|) =  74.0000000000000000 .
         Few subuniverses, Subcase (5/a) above is excluded.

~~Case (5): b+c=:u is new, but v:=c*d=f and w:=d+g=a, then a+u=i.
~~~~Subcase (5/b): a*u=:t is not new. 
                   (Clearly, a*u=c. Observe: a*b=b*a*u=b*c=e.)
|A|=10, A(without commas)={oiabcdefgu}. Constraints:
e+g=c b*c=e d*e=f c+d=a a+b=i f*g=o
 b+c=u  c*d=f d+g=a   a+u=i
  a*u=c  a*b=e
Result for A=Case5/b:  |Sub(A)| = 322, that is,
sigma(A) = |Sub(A)|*2^(8-|A|) =  80.5000000000000000 .
         Few subuniverses, Subcase (5/b) above is excluded.
         Thus, Case (5) is excluded.

~~Case (6): c*d=:v is new but u:=b+c=i and w:=d+g=a. 
            Then, clearly,  e*v=f
~~~~Subcase (6/a): e+v=:t is new. Since e<t<c, t*d=v and t+g=c,
|A|=11, A(without commas)={oiabcdefgvt}. Constraints:
e+g=c b*c=e d*e=f c+d=a a+b=i f*g=o
 c*d=v b+c=i d+g=a   e*v=f
  e+v=t  t*d=v t+g=c
Result for A=Case6/a:  |Sub(A)| = 605, that is,
sigma(A) = |Sub(A)|*2^(8-|A|) =  75.6250000000000000 .
         Few subuniverses, Subcase (6/a) above is excluded.

~~Case (6): c*d=:v is new but u:=b+c=i and w:=d+g=a. 
            Then, cclearly,  e*v=f.
~~~~Subcase (6/b): t:=e+v=c is old. Then, clearly, 
                   b+v=b+e+v=b+c=i, b*v=b*c*d=e*v=f
|A|=10, A(without commas)={oiabcdefgv}. Constraints:
e+g=c b*c=e d*e=f c+d=a a+b=i f*g=o
 c*d=v b+c=i d+g=a   e*v=f
  e+v=c  b+v=i b*v=f
Result for A=Case6/a:  |Sub(A)| = 328, that is,
sigma(A) = |Sub(A)|*2^(8-|A|) =  82.0000000000000000 .
         Few subuniverses, Subcase (6/b) above is excluded.
         Thus, Case (6) is excluded.

~~Case (7): d+g=:w is new, u:=b+c=i and v:=c*d=f are old. 
            Then c+w=a and b+g=b+e+g=b+c=i.
     Let c*w=:t. Since f<=t, t<>g. So g<t<c, t<w, t is new, and
     t+d=w, e+t=c, b+t=b+e+t=b+c=i)
|A|=11, A(without commas)={oiabcdefgwt}. Constraints:
e+g=c b*c=e d*e=f c+d=a a+b=i f*g=o
 d+g=w b+c=i c*d=f  c+w=a b+g=i
  c*w=t  t+d=w  e+t=c b+t=i
Result for A=Case7:  |Sub(A)| = 576, that is,
sigma(A) = |Sub(A)|*2^(8-|A|) =  72.0000000000000000 .
         Few subuniverses, Case (7) above is excluded.

~~Case (8) u:=b+c=i, v:=c*d=f and w:=d+g=a are old, 
           then b+g=b+e+g=b+c=i.                 
~~~~Subcase (8/a): suppose, for contradiction, that 
            a*b=:t>e is new. Then t*c=e.
~~~~~~Sub-subcase (8/a/1): suppose, for contradiction, 
        that c+t=:s is new, then t+g=t+e+g=t+c=s and  b*s=t.
|A|=11, A(without commas)={oiabcdefgts}. Constraints:
e+g=c b*c=e d*e=f c+d=a a+b=i f*g=o
 b+c=i c*d=f d+g=a   b+g=i
  a*b=t  t*c=e
   c+t=s    t+g=s   b*s=t
Result for A=Case8a1:  |Sub(A)| = 596, that is,
sigma(A) = |Sub(A)|*2^(8-|A|) =  74.5000000000000000 .
         Few subuniverses, Subsubcase (8/a/1) above is excluded.

~~Case (8) u:=b+c=i, v:=c*d=f and w:=d+g=a are old, 
           then b+g=b+e+g=b+c=i.                 
~~~~Subcase (8/a): suppose, for contradiction, 
                   that a*b=:t>e is new. then t*c=e.
~~~~~~Armed with not (8/a/1): s:=c+t=a is old, 
             then t+g=t+e+g=t+c=a
|A|=10, A(without commas)={oiabcdefgt}. Constraints:
e+g=c b*c=e d*e=f c+d=a a+b=i f*g=o
 b+c=i c*d=f d+g=a   b+g=i
  a*b=t  t*c=e
   c+t=a  t+g=a
Result for A=Case8a2:  |Sub(A)| = 326, that is,
sigma(A) = |Sub(A)|*2^(8-|A|) =  81.5000000000000000 .
         Few subuniverses, Sub-subcase (8/a) above is excluded.

~~Case (8) u:=b+c=i, v:=c*d=f and w:=d+g=a, 
           then b+g=b+e+g=b+c=i and since not (8/a),
          a*b=e
~~~~Subcase (8/b) suppose, for contradiction, that e+d=:t<a is
             new; then c+t=a and t+g=e+d+g=e+=e+a=a.
|A|=10, A(without commas)={oiabcdefgt}. Constraints:
e+g=c b*c=e d*e=f c+d=a a+b=i f*g=o
 b+c=i c*d=f d+g=a   b+g=i  a*b=e
  e+d=t  c+t=a t+g=a
Result for A=Case8b:  |Sub(A)| = 325, that is,
sigma(A) = |Sub(A)|*2^(8-|A|) =  81.2500000000000000 .
      Few subuniverses, Case (8/b) above is excluded.
      Thus, after that Cases (8/a) and (8/b) are excluded, 
      we have that a*b=e and e+d=a. 
      We are in Case (8), so the two "pentagons" are sublattices.
      Since b*d=b*a*d=e*d=f and b+d=b+e+d=b+a=i, we conclude that 
      {a,b,c,d,e,f,i} is a sublattice of L, not only that of F_0,
      that is
 (#1) a+b=i a*b=e, b+c=i b*c=e b+d=i b*d=f, d+e=a d*e=f

~~Case (8) b+c=i, c*d=f and d+g=a, then b+g=i, and we know that 
      (#1) a+b=i a*b=e, b+c=i b*c=e b+d=i b*d=f, d+e=a d*e=f
          Note that the two lines above are not disjoint 
          but this is not a problem.
~~~~Subcase (8/c): f+g=:t is new. 
                   Then e+t=c, b+t=b+e+t=b+c=i and t+d=t+g+d=a.
|A|=10, A(without commas)={oiabcdefgt}. Constraints:
e+g=c b*c=e d*e=f c+d=a a+b=i f*g=o
 b+c=i c*d=f d+g=a  b+g=i
   a+b=i a*b=e  b+c=i b*c=e b+d=i b*d=f  d+e=a d*e=f
    f+g=t  e+t=c b+t=i  t+d=a
Result for A=Case8c:  |Sub(A)| = 294, that is,
sigma(A) = |Sub(A)|*2^(8-|A|) =  73.5000000000000000 .
  Few subuniverses, Subcase (8/c) above is excluded.
  So, in addition to the equalities in Case (8) and (#1), we have
    (#2) f+g=c   
  Armed with these equalities, we continue with

~~~~Subcase (8/d): e*g=:t>o. (Observe that f*t=o.)
|A|=10, A(without commas)={oiabcdefgt}. Constraints:
e+g=c b*c=e d*e=f c+d=a a+b=i f*g=o
 b+c=i c*d=f d+g=a  b+g=i
   a+b=i a*b=e  b+c=i b*c=e b+d=i b*d=f  d+e=a d*e=f
    f+g=c
     e*g=t  f*t=o
Result for A=Case8d:  |Sub(A)| = 293, that is,
sigma(A) = |Sub(A)|*2^(8-|A|) =  73.2500000000000000 .
   Few subuniverses, Subcase (8d) above is excluded. So,  e*g=o,
   and now we have three pentagons. That is, in addition to (#1),
   and Case (8), now {o,f,g,e,c} is a pentagon, that is,
 (#3) f+g=c  f*g=o e+g=c e*g=o, and we also have b*g=b*c*g=e*g=o

~~~~Subcase (8/e): d*g=:t>o, in addition to Case (8), 
                   (#1) and (#3). Then
    b*g=b*c*g=e*g=o, f*t=f*g*t=o, e*t=e*g*t=o,
    b*t=b*g*t=b*c*g*t=e*g*t=o   holds.
|A|=10, A(without commas)={oiabcdefgt}. Constraints:
e+g=c b*c=e d*e=f c+d=a a+b=i f*g=o
 b+c=i c*d=f d+g=a  b+g=i
  a+b=i a*b=e  b+c=i b*c=e b+d=i b*d=f  d+e=a d*e=f
   f+g=c  f*g=o e+g=c e*g=o   b*g=o
    d*g=t   b*g=o f*t=o e*t=o b*t=o
Result for A=Case8e:  |Sub(A)| = 276, that is,
sigma(A) = |Sub(A)|*2^(8-|A|) =  69.0000000000000000 .
     Few subuniverses, Subcase (8e) is excluded, whereby d*g=o

~~~~Subcase (8/f): d*g=o since Subcase (8e) is excluded and the 
   equalities from Case (8), (#1) and (#3) hold.
   (Note that Subcase (8f) deals with the situation where F_0 is
    a sublattice, not just a subposet of L.)
|A|=10, A(without commas)={oiabcdefgt}. Constraints:
e+g=c b*c=e d*e=f c+d=a a+b=i f*g=o
 b+c=i c*d=f d+g=a  b+g=i
  a+b=i a*b=e  b+c=i b*c=e b+d=i b*d=f  d+e=a d*e=f
   f+g=c  f*g=o e+g=c e*g=o   b*g=o
    d*g=o
Result for A=Case8f:  |Sub(A)| = 332, that is,
sigma(A) = |Sub(A)|*2^(8-|A|) =  83.0000000000000000 .
       Few subuniverses, Subcase (8f) above is excluded.  
       Thus, it is excluded that F_0 is a subposet of L.  Q.e.d.

The computation took 140/1000 seconds.
\end{verbatim}
\normalfont

\clearpage
\def\appendixhead{{Cz\'edli: Eighty-three sublattices / Appendix: $F_1$}}
\markboth\appendixhead\appendixhead
\centerline{\bf{{\LARGE Appendix: $F_1$}}}
\normalfont
\begin{verbatim}
Version of January 2, 2019; reformatted May 7, 2019
In the paper, see the corollary of the
Encapsulated 2-ladder Lemma.

|A|=10, A(without commas)={ioabcdefgj}. Constraints:
(encapsulated 2-ladder) e*f=b c+d=g b*c=a f+g=j a*d=o e+j=i
(encapsulted 2-ladder) b+c=f f*g=c
  (then) b+g=j ; since b+g=b+c+g=f+g
  (and)  b*g=a ; since b*g=b*f*g=b*c
  (and)  c*e=a ; since c*e=c*f*e=c*b
  (and)  f+d=j ; since f+d=f+c+d=f+g
Result for A=(F_1^-):  |Sub(A)| = 327, that is,
sigma(A) = |Sub(A)|*2^(8-|A|) =  81.7500000000000000 .
Few subuniverses. This settles F_1. Q.e.d.

The computation took 0/1000 seconds.
\end{verbatim}
\normalfont

\clearpage
\def\appendixhead{{Cz\'edli: Eighty-three sublattices / Appendix: $G_0$}}
\markboth\appendixhead\appendixhead
\centerline{\bf{{\LARGE Appendix: $G_0$}}}
\normalfont

\begin{verbatim}
Version of Nov, 27, 2018; reformatted on May 6, 2019
G_0 from Kelly-Rival: "Planar lattices"; its edges are
      oa ob ac ad ag  bd ce de df eh ej fj gj hi ji
  Here we prove that G_0 is not a subposet of L. 
  Suppose the contrary.
As usual, we can and will assume that a+b=i and h*j=o.
Decreasing b to c+d+e if necessary, 
we assume also that b=c+d+e.
Then increasing c to a*b, we can assume that a*b=c. 
Similarly, after increasing g if necessary, we can 
assume that c*d=g. However, then
a*d=a*(b*d)=(a*b)*d=c*d=g,  so we have that a*d=g. 
Next, after increasing h if necessary, we have that 
h=f*g*e. Finally, after increasing o, h*j=o .
Thus, the permanent assumption throughout is
(#0): a+b=i a*b=c c*d=g a*d=g h*j=o, and keep in mind that
(#1): b=c+d+e  h=f*g*e

**(Case 1)[of 3 by #1] c+d=b, then a+d=a+c+d=a+b=i
                            this gives 133
****(Sub1case 1a) f+g=c, then f+d=f+g+d=c+d=c; 
                            this gives 106.75
******(Sub2case 1a1) h+j=g, then f+j=f+h+j=f+g=c; 
                            this gives 96.25
********(Sub3case 1a1a)[of 3 by #1] f*g=h, 
                then f*d=f*c*d=f*g=h; giving 86.25
**********(Sub4case 1a1a1) f*e=h
|A|=11, A(without commas)={oiabcdefghj}. Constraints:
a+b=i  a*b=c  c*d=g  a*d=g  h*j=o;(#0); this gives 150.5
 c+d=b  a+d=i;(Case 1)[of 3]
  f+g=c   f+d=c;(Sub1case 1a)
   h+j=g  f+j=c;(Sub2case 1a1)
    f*g=h  f*d=h;(Sub3case 1a1a)
     f*e=h;(Sub4case 1a1a1)
Result for A=Case-1a1a1:  |Sub(A)| = 650, that is,
sigma(A) = |Sub(A)|*2^(8-|A|) =  81.2500000000000000 .
     Few subuniverses, (Sub4case 1a1a1) is excluded.

**(Case 1)[of 3 by #1] c+d=b, then a+d=a+c+d=a+b=i; 
                              this gives 133
****(Sub1case 1a) f+g=c, then f+d=f+g+d=c+d=c; 
                              this gives 106.75
******(Sub2case 1a1) h+j=g, then f+j=f+h+j=f+g=c;
                            this gives 96.25
********(Sub3case 1a1a)[of 3 by #1] f*g=h, 
                  then f*d=f*c*d=f*g=h; giving 86.25
**********(Sub4case 1a1a2) f*e:=u> h
|A|=12, A(without commas)={oiabcdefghju}. Constraints:
a+b=i  a*b=c  c*d=g  a*d=g  h*j=o;(#0); this gives 150.5
 c+d=b  a+d=i;(Case 1)[of 3]
  f+g=c   f+d=c;(Sub1case 1a)
   h+j=g  f+j=c;(Sub2case 1a1)
    f*g=h  f*d=h;(Sub3case 1a1a)
     f*e=u;(Sub4case 1a1a2)
Result for A=Case-1a1a2:  |Sub(A)| = 1284, that is,
sigma(A) = |Sub(A)|*2^(8-|A|) =  80.2500000000000000 .
     Few subuniverses, (Sub4case 1a1a2) is excluded.
     Thus, (Sub3case 1a1a) is excluded

**(Case 1)[of 3 by #1] c+d=b, then a+d=a+c+d=a+b=i; 
                              this gives 133
****(Sub1case 1a) f+g=c, then f+d=f+g+d=c+d=c; 
                         this gives 106.75
******(Sub2case 1a1) h+j=g, then f+j=f+h+j=f+g=c;
                            this gives 96.25
********(Sub3case 1a1b)[of 3 by #1] f*e=h; 
         since 1a1a is excluded, f*g:=u>h and u*e=h
|A|=12, A(without commas)={oiabcdefghju}. Constraints:
a+b=i  a*b=c  c*d=g  a*d=g  h*j=o;(#0); this gives 150.5
 c+d=b  a+d=i;(Case 1)[of 3]
  f+g=c   f+d=c;(Sub1case 1a)
   h+j=g  f+j=c;(Sub2case 1a1)
    f*e=h   f*g=u  u*e=h;(Sub3case 1a1b)
Result for A=Case-1a1b:  |Sub(A)| = 1128, that is,
sigma(A) = |Sub(A)|*2^(8-|A|) =  70.5000000000000000 .
     Few subuniverses, (Sub3case 1a1b) is excluded.

**(Case 1)[of 3 by #1] c+d=b, then a+d=a+c+d=a+b=i; 
                              this gives 133
****(Sub1case 1a) f+g=c, then f+d=f+g+d=c+d=c;
                         this gives 106.75
******(Sub2case 1a1) h+j=g, then f+j=f+h+j=f+g=c;
                            this gives 96.25
********(Sub3case 1a1c)[of 3 by #1] g*e=h;
          since 1a1a is excluded, f*g:=u>h and u*e=h
|A|=12, A(without commas)={oiabcdefghju}. Constraints:
a+b=i  a*b=c  c*d=g  a*d=g  h*j=o;(#0); this gives 150.5
 c+d=b  a+d=i;(Case 1)[of 3]
  f+g=c   f+d=c;(Sub1case 1a)
   h+j=g  f+j=c;(Sub2case 1a1)
    g*e=h   f*g=u  u*e=h;(Sub3case 1a1c)
Result for A=Case-1a1c:  |Sub(A)| = 1084, that is,
sigma(A) = |Sub(A)|*2^(8-|A|) =  67.7500000000000000 .
     Few subuniverses, (Sub3case 1a1c)) is excluded.
     Thus, (Sub2case 1a1) is excluded 

**(Case 1)[of 3 by #1] c+d=b, then a+d=a+c+d=a+b=i;
                              this gives 133
****(Sub1case 1a) f+g=c, then f+d=f+g+d=c+d=c;
                         this gives 106.75
******(Sub2case 1a2) h+j=:v<g;
                         this gives 99.125
********(Sub3case 1a2a)[of 3 by #1] f*g=h; then f*v=h
|A|=12, A(without commas)={oiabcdefghjv}. Constraints:
a+b=i  a*b=c  c*d=g  a*d=g  h*j=o;(#0); this gives 150.5
 c+d=b  a+d=i;(Case 1)[of 3]
  f+g=c   f+d=c;(Sub1case 1a)
   h+j=v;(Sub2case 1a2)
    f*g=h   f*v=h;(Sub3case 1a2a)
Result for A=Case-1a2a:  |Sub(A)| = 1322, that is,
sigma(A) = |Sub(A)|*2^(8-|A|) =  82.6250000000000000 .
     Few subuniverses, (Sub3case 1a2a) is excluded.

**(Case 1)[of 3 by #1] c+d=b, then a+d=a+c+d=a+b=i; 
                              this gives 133
****(Sub1case 1a) f+g=c, then f+d=f+g+d=c+d=c;
                         this gives 106.75
******(Sub2case 1a2) h+j=:v<g;
                         this gives 99.125
********(Sub3case 1a2b)[of 3 by #1] f*e=h;
          since 1a2a is excluded, f*g=:u>h, u*e=h
|A|=13, A(without commas)={oiabcdefghjvu}. Constraints:
a+b=i  a*b=c  c*d=g  a*d=g  h*j=o;(#0); this gives 150.5
 c+d=b  a+d=i;(Case 1)[of 3]
  f+g=c   f+d=c;(Sub1case 1a)
   h+j=v;(Sub2case 1a2)
    f*e=h  f*g=u u*e=h;(Sub3case 1a2b)
Result for A=Case-1a2b:  |Sub(A)| = 2320, that is,
sigma(A) = |Sub(A)|*2^(8-|A|) =  72.5000000000000000 .
     Few subuniverses, (Sub3case 1a2b) is excluded.

**(Case 1)[of 3 by #1] c+d=b, then a+d=a+c+d=a+b=i; 
                              this gives 133
****(Sub1case 1a) f+g=c, then f+d=f+g+d=c+d=c;
                         this gives 106.75
******(Sub2case 1a2) h+j=:v<g; this gives 99.125
********(Sub3case 1a2c)[of 3 by #1] g*e=h;
          since 1a2a is excluded, f*g=:u>h, u*e=h
|A|=13, A(without commas)={oiabcdefghjvu}. Constraints:
a+b=i  a*b=c  c*d=g  a*d=g  h*j=o;(#0); this gives 150.5
 c+d=b  a+d=i;(Case 1)[of 3]
  f+g=c   f+d=c;(Sub1case 1a)
   h+j=v;(Sub2case 1a2)
    g*e=h  f*g=u  u*e=h;(Sub3case 1a2c)
Result for A=Case-1a2c:  |Sub(A)| = 2256, that is,
sigma(A) = |Sub(A)|*2^(8-|A|) =  70.5000000000000000 .
     Few subuniverses, (Sub3case 1a2c) is excluded.
     Thus, (Sub2case 1a2) is excluded.
     Thus, (Sub1case 1a) is excluded.

**(Case 1)[of 3 by #1] c+d=b, then a+d=a+c+d=a+b=i;
                              this gives 133
****(Sub1case 1b) f+g=:w<c, then w*d=g;
                            this gives 107.625
******(Sub2case 1b1) h+j=g, then f+j=f+h+j=f+g=w;
                            this gives 95.875
********(Sub3case 1b1a)[of 3 by #1] f*g=h, then 
          f*d=f*c*d=f*g=h; this gives 84.375
**********(Sub4case 1b1a1) f*e=h
|A|=12, A(without commas)={oiabcdefghjw}. Constraints:
a+b=i  a*b=c  c*d=g  a*d=g  h*j=o;(#0); this gives 150.5
 c+d=b  a+d=i;(Case 1)[of 3]
  f+g=w  w*d=g;(Sub1case 1b)
   h+j=g   f+j=w;(Sub2case 1b1)
    f*g=h   f*d=h;(Sub3case 1b1a)
     f*e=h;(Sub4case 1b1a1)
Result for A=Case-1b1a1:  |Sub(A)| = 1272, that is,
sigma(A) = |Sub(A)|*2^(8-|A|) =  79.5000000000000000 .
     Few subuniverses, (Sub4case 1b1a1) is excluded.

**(Case 1)[of 3 by #1] c+d=b, then a+d=a+c+d=a+b=i;
                              this gives 133
****(Sub1case 1b) f+g=:w<c, then w*d=g;
                            this gives 107.625
******(Sub2case 1b1) h+j=g, then f+j=f+h+j=f+g=w;
                            this gives 95.875
********(Sub3case 1b1a)[of 3 by #1] f*g=h, then 
          f*d=f*c*d=f*g=h; this gives 84.375
**********(Sub4case 1b1a2) f*e=u>h
|A|=13, A(without commas)={oiabcdefghjwu}. Constraints:
a+b=i  a*b=c  c*d=g  a*d=g  h*j=o;(#0); this gives 150.5
 c+d=b  a+d=i;(Case 1)[of 3]
  f+g=w  w*d=g;(Sub1case 1b)
   h+j=g   f+j=w;(Sub2case 1b1)
    f*g=h   f*d=h;(Sub3case 1b1a)
     f*e=u;(Sub4case 1b1a2)
Result for A=Case-1b1a2:  |Sub(A)| = 2499, that is,
sigma(A) = |Sub(A)|*2^(8-|A|) =  78.0937500000000000 .
     Few subuniverses, (Sub4case 1b1a2) is excluded.
     Thus,  (Sub3case 1b1a) is excluded.

**(Case 1)[of 3 by #1] c+d=b, then a+d=a+c+d=a+b=i;
                              this gives 133
****(Sub1case 1b) f+g=:w<c, then w*d=g;
                            this gives 107.625
******(Sub2case 1b1) h+j=g, then f+j=f+h+j=f+g=w;
                            this gives 95.875
********(Sub3case 1b1b)[of 3 by #1]  f*e=h; 
          since 1b1a is excluded, f*g:=u>h and u*e=h
|A|=13, A(without commas)={oiabcdefghjwu}. Constraints:
a+b=i  a*b=c  c*d=g  a*d=g  h*j=o;(#0); this gives 150.5
 c+d=b  a+d=i;(Case 1)[of 3]
  f+g=w  w*d=g;(Sub1case 1b)
   h+j=g   f+j=w;(Sub2case 1b1)
    f*e=h  f*g=u  u*e=h;(Sub3case 1b1b)
Result for A=Case-1b1b:  |Sub(A)| = 2230, that is,
sigma(A) = |Sub(A)|*2^(8-|A|) =  69.6875000000000000 .
     Few subuniverses,(Sub3case 1b1b) is excluded.

**(Case 1)[of 3 by #1] c+d=b, then a+d=a+c+d=a+b=i;
                              this gives 133
****(Sub1case 1b) f+g=:w<c, then w*d=g;
                            this gives 107.625
******(Sub2case 1b1) h+j=g, then f+j=f+h+j=f+g=w;
                            this gives 95.875
********(Sub3case 1b1c)[of 3 by #1]  g*e=h;
          since 1b1a is excluded, f*g:=u>h and u*e=h
|A|=13, A(without commas)={oiabcdefghjwu}. Constraints:
a+b=i  a*b=c  c*d=g  a*d=g  h*j=o;(#0); this gives 150.5
 c+d=b  a+d=i;(Case 1)[of 3]
  f+g=w  w*d=g;(Sub1case 1b)
   h+j=g   f+j=w;(Sub2case 1b1)
    g*e=h  f*g=u  u*e=h;(Sub3case 1b1c))
Result for A=Case-1b1c:  |Sub(A)| = 2148, that is,
sigma(A) = |Sub(A)|*2^(8-|A|) =  67.1250000000000000 .
     Few subuniverses,(Sub3case 1b1c) is excluded.
     Thus, (Sub2case 1b1) is excluded.

**(Case 1)[of 3 by #1] c+d=b, then a+d=a+c+d=a+b=i; 
                              this gives 133
****(Sub1case 1b) f+g=:w<c, then w*d=g; 
                            this gives 107.625
******(Sub2case 1b2)   h+j=:v<g this gives 99.9375
********(Sub3case 1b2a)[of 3 by #1] f*g=h; then f*v=h
|A|=13, A(without commas)={oiabcdefghjwv}. Constraints:
a+b=i  a*b=c  c*d=g  a*d=g  h*j=o;(#0); this gives 150.5
 c+d=b  a+d=i;(Case 1)[of 3]
  f+g=w  w*d=g;(Sub1case 1b)
   h+j=v;(Sub2case 1b2)
    f*g=h  f*v=h;(Sub3case 1b2a)
Result for A=Case-1b2a:  |Sub(A)| = 2622, that is,
sigma(A) = |Sub(A)|*2^(8-|A|) =  81.9375000000000000 .
     Few subuniverses,(Sub3case 1b2a) is excluded.

**(Case 1)[of 3 by #1] c+d=b, then a+d=a+c+d=a+b=i; 
                              this gives 133
****(Sub1case 1b) f+g=:w<c, then w*d=g; 
                            this gives 107.625
******(Sub2case 1b2)   h+j=:v<g this gives 99.9375
********(Sub3case 1b2b)[of 3 by #1] f*e=h; 
          since 1b2a is excluded, f*g=:u>h, u*e=h
|A|=14, A(without commas)={oiabcdefghjwvu}. Constraints:
a+b=i  a*b=c  c*d=g  a*d=g  h*j=o;(#0); this gives 150.5
 c+d=b  a+d=i;(Case 1)[of 3]
  f+g=w  w*d=g;(Sub1case 1b)
   h+j=v;(Sub2case 1b2)
    f*e=h  f*g=u u*e=h;(Sub3case 1b2b)
Result for A=Case-1b2b:  |Sub(A)| = 4626, that is,
sigma(A) = |Sub(A)|*2^(8-|A|) =  72.2812500000000000 .
     Few subuniverses,(Sub3case 1b2b) is excluded.

**(Case 1)[of 3 by #1] c+d=b, then a+d=a+c+d=a+b=i; 
                              this gives 133
****(Sub1case 1b) f+g=:w<c, then w*d=g; this gives 107.625
******(Sub2case 1b2)   h+j=:v<g this gives 99.9375
********(Sub3case 1b2c)[of 3 by #1] g*e=h; 
          since 1b2a is excluded, f*g=:u>h, u*e=h
|A|=14, A(without commas)={oiabcdefghjwvu}. Constraints:
a+b=i  a*b=c  c*d=g  a*d=g  h*j=o;(#0); this gives 150.5
 c+d=b  a+d=i;(Case 1)[of 3]
  f+g=w  w*d=g;(Sub1case 1b)
   h+j=v;(Sub2case 1b2)
     g*e=h  f*g=u  u*e=h;(Sub3case 1b2c)
Result for A=Case-1b2c:  |Sub(A)| = 4530, that is,
sigma(A) = |Sub(A)|*2^(8-|A|) =  70.7812500000000000 .
     Few subuniverses,(Sub3case 1b2c) is excluded.
     Thus, (Sub2case 1b2) is excluded.
     Thus, (Sub1case 1b) is excluded.
     Thus, (Case 1) is excluded.

**(Case 2)[of 3 by #1] c+e=b, not Case1 gives c+d=:t<b, 
            so t+e=b; giving 108.5
****(Sub1case 2a) f+g=c, then f+d=f+g+d=c+d=t; 
                         this gives 87.5
******(Sub2case 2a1) h+j=g, then f+j=f+h+j=f+g=c
|A|=12, A(without commas)={oiabcdefghjt}. Constraints:
a+b=i  a*b=c  c*d=g  a*d=g  h*j=o;(#0); this gives 150.5
 c+e=b   c+d=t  t+e=b;(Case 2)
  f+g=c   f+d=t;(Sub1case 2a)
   h+j=g  f+j=c;(Sub2case 2a1)
Result for A=Case-2a1:  |Sub(A)| = 1248, that is,
sigma(A) = |Sub(A)|*2^(8-|A|) =  78.0000000000000000 .
     Few subuniverses, (Sub2case 2a1) is excluded.

**(Case 2)[of 3 by #1] c+e=b, not Case1 gives 
     c+d=:t<b, so t+e=b; giving 108.5
****(Sub1case 2a) f+g=c, then f+d=f+g+d=c+d=t; 
                         this gives 87.5
******(Sub2case 2a2)   h+j=:v<g;
|A|=13, A(without commas)={oiabcdefghjtv}. Constraints:
a+b=i  a*b=c  c*d=g  a*d=g  h*j=o;(#0); this gives 150.5
 c+e=b   c+d=t  t+e=b;(Case 2)
  f+g=c   f+d=t;(Sub1case 2a)
   h+j=v;(Sub2case 2a2)
Result for A=Case-2a2:  |Sub(A)| = 2600, that is,
sigma(A) = |Sub(A)|*2^(8-|A|) =  81.2500000000000000 .
     Few subuniverses, (Sub2case 2a2) is excluded.
     Thus, (Sub1case 2a) is excluded.

**(Case 2)[of 3 by #1] c+e=b, not Case1 gives c+d=:t<b, 
    so t+e=b; giving 108.5
****(Sub1case 2b)  f+g=:w<c, then w*d=g; this gives 87.0825
******(Sub2case 2b1) h+j=g, then f+j=f+h+j=f+g=w
|A|=13, A(without commas)={oiabcdefghjtw}. Constraints:
a+b=i  a*b=c  c*d=g  a*d=g  h*j=o;(#0); this gives 150.5
 c+e=b   c+d=t  t+e=b;(Case 2)
  f+g=w   w*d=g;(Sub1case 2b)
   h+j=g  f+j=w;(Sub2case 2b1)
Result for A=Case-2b1:  |Sub(A)| = 2490, that is,
sigma(A) = |Sub(A)|*2^(8-|A|) =  77.8125000000000000 .
     Few subuniverses, (Sub2case 2b1) is excluded.

**(Case 2)[of 3 by #1] c+e=b, not Case1 gives c+d=:t<b, 
    so t+e=b; giving 108.5
****(Sub1case 2b)  f+g=:w<c, then w*d=g; this gives 87.0825
******(Sub2case 2b2) h+j=:v<g
|A|=14, A(without commas)={oiabcdefghjtwv}. Constraints:
a+b=i  a*b=c  c*d=g  a*d=g  h*j=o;(#0); this gives 150.5
 c+e=b   c+d=t  t+e=b;(Case 2)
  f+g=w   w*d=g;(Sub1case 2b)
   h+j=v;(Sub2case 2b2)
Result for A=Case-2b2:  |Sub(A)| = 5174, that is,
sigma(A) = |Sub(A)|*2^(8-|A|) =  80.8437500000000000 .
     Few subuniverses, (Sub2case 2b2) is excluded.
     Thus, (Sub1case 2b) is excluded.
     Thus, (Case 2) is excluded.

**(Case 3)[of 3 by #1] d+e=b, not Case1 gives c+d=:t<b, 
    so t+e=b; giving 110.25
****(Sub1case 3a)  f+g=c, then f+d=f+g+d=c+d=t; 
                          this gives 91.875
******(Sub2case 3a1) h+j=g, then f+j=f+h+j=f+g=c
|A|=12, A(without commas)={oiabcdefghjt}. Constraints:
a+b=i  a*b=c  c*d=g  a*d=g  h*j=o;(#0); this gives 150.5
 d+e=b   c+d=t  t+e=b;*(Case 3)
  f+g=c   f+d=t;(Sub1case 3a)
   h+j=g   f+j=c;(Sub2case 3a1)
Result for A=Case-3a1:  |Sub(A)| = 1312, that is,
sigma(A) = |Sub(A)|*2^(8-|A|) =  82.0000000000000000 .
     Few subuniverses, (Sub2case 3a1) is excluded.

**(Case 3)[of 3 by #1] d+e=b, not Case1 gives c+d=:t<b, 
    so t+e=b; giving 110.25
****(Sub1case 3a)  f+g=c, then f+d=f+g+d=c+d=t; 
                          this gives 91.875
******(Sub2case 3a2)  h+j=:v<g; this gives 85.3125
******(Sub3case 3a2a)  f*e=h
|A|=13, A(without commas)={oiabcdefghjtv}. Constraints:
a+b=i  a*b=c  c*d=g  a*d=g  h*j=o;(#0); this gives 150.5
 d+e=b   c+d=t  t+e=b;*(Case 3)
  f+g=c   f+d=t;(Sub1case 3a)
   h+j=v;(Sub2case 3a2)
    f*e=h;(Sub3case 3a2a)
Result for A=Case-3a2a:  |Sub(A)| = 2466, that is,
sigma(A) = |Sub(A)|*2^(8-|A|) =  77.0625000000000000 .
     Few subuniverses, (Sub3case 3a2a) is excluded.

**(Case 3)[of 3 by #1] d+e=b, not Case1 gives c+d=:t<b, 
    so t+e=b; giving 110.25
****(Sub1case 3a)  f+g=c, then f+d=f+g+d=c+d=t; 
                          this gives 91.875
******(Sub2case 3a2)  h+j=:v<g; this gives 85.3125
******(Sub3case 3a2b)  f*e=x>h
|A|=14, A(without commas)={oiabcdefghjtvx}. Constraints:
a+b=i  a*b=c  c*d=g  a*d=g  h*j=o;(#0); this gives 150.5
 d+e=b   c+d=t  t+e=b;*(Case 3)
  f+g=c   f+d=t;(Sub1case 3a)
   h+j=v;(Sub2case 3a2)
    f*e=x;(Sub3case 3a2b)
Result for A=Case-3a2b:  |Sub(A)| = 5031, that is,
sigma(A) = |Sub(A)|*2^(8-|A|) =  78.6093750000000000 .
     Few subuniverses, (Sub3case 3a2b) is excluded.
     Thus, (Sub2case 3a2) is excluded.    
     (Sub1case 3a) is excluded.

**(Case 3)[of 3 by #1] d+e=b, not Case1 gives c+d=:t<b, 
    so t+e=b; giving 110.25
****(Sub1case 3b) f+g=:w<c, then w*d=g; this gives 89.6875
******(Sub2case 3b1) h+j=g, then f+j=f+h+j=f+g=w
|A|=13, A(without commas)={oiabcdefghjtw}. Constraints:
a+b=i  a*b=c  c*d=g  a*d=g  h*j=o;(#0); this gives 150.5
 d+e=b   c+d=t  t+e=b;*(Case 3)
  f+g=w   w*d=g;(Sub1case 3b)
   h+j=g   f+j=w;(Sub2case 3b1)
Result for A=Case-3b1:  |Sub(A)| = 2558, that is,
sigma(A) = |Sub(A)|*2^(8-|A|) =  79.9375000000000000 .
     Few subuniverses, (Sub2case 3b1) is excluded.

**(Case 3)[of 3 by #1] d+e=b, not Case1 gives c+d=:t<b, 
    so t+e=b; giving 110.25
****(Sub1case 3b) f+g=:w<c, then w*d=g; this gives 89.6875
******(Sub2case 3b2) h+j=:v<g; this gives 83.28125
********(Sub2case 3b2a) f*e=h;
|A|=14, A(without commas)={oiabcdefghjtwv}. Constraints:
a+b=i  a*b=c  c*d=g  a*d=g  h*j=o;(#0); this gives 150.5
 d+e=b   c+d=t  t+e=b;*(Case 3)
  f+g=w   w*d=g;(Sub1case 3b)
   h+j=v;(Sub2case 3b2)
    f*e=h;(Sub2case 3b2a)
Result for A=Case-3b2a:  |Sub(A)| = 4810, that is,
sigma(A) = |Sub(A)|*2^(8-|A|) =  75.1562500000000000 .
     Few subuniverses, (Sub2case 3b2a) is excluded.

**(Case 3)[of 3 by #1] d+e=b, not Case1 gives c+d=:t<b, 
    so t+e=b; giving 110.25
****(Sub1case 3b) f+g=:w<c, then w*d=g; this gives 89.6875
******(Sub2case 3b2) h+j=:v<g; this gives 83.28125
********(Sub2case 3b2b) f*e=:x>h;
|A|=15, A(without commas)={oiabcdefghjtwvx}. Constraints:
a+b=i  a*b=c  c*d=g  a*d=g  h*j=o;(#0); this gives 150.5
 d+e=b   c+d=t  t+e=b;*(Case 3)
  f+g=w   w*d=g;(Sub1case 3b)
   h+j=v;(Sub2case 3b2)
    f*e=x;(Sub2case 3b2b)
Result for A=Case-3b2b:  |Sub(A)| = 9815, that is,
sigma(A) = |Sub(A)|*2^(8-|A|) =  76.6796875000000000 .
     Few subuniverses, (Sub2case 3b2b) is excluded.
     Thus, (Sub2case 3b2) is excluded.
     Thus, (Sub1case 3b) is excluded.
     Thus, (Case 3) is excluded.
     Therefore, G_0 cannot be a subset of L, a contradiction as required.

The computation took 156/1000 seconds.
\end{verbatim}

\normalfont

\clearpage
\def\appendixhead{{Cz\'edli: Eighty-three sublattices / Appendix: $H_0$}}
\markboth\appendixhead\appendixhead

\centerline{\bf{{\LARGE Appendix: $H_0$}}}
\normalfont

\begin{verbatim}
H_0 from Kelly-Rival: "Planar lattices"; its edges are
      oa ob oc ad bd be bh cg df dg eg fi gi hi
Suppose, for contradiction, that  H_0 is a subposet of L
and L has many subuniverses. We will indicate by ' 
where the argument splits into 3 subcases (according
to the Antichain Lemma in the paper); otherwise the
argument splits into two subcases. 

We can always assume (after changing some elements if
necessary) that 
     f*g=d  (always) , and   
     (#1)  g=d+c+e, b=d*e*h,  i=f+g+h, o=a*b*c  
To benefit from the selfduality of H_0, either we assume that 
a+b=d, then (o,a,b,c,d,e,h,g,f,i) \mapsto 
            (i,f,g,h,d,e,c,b,a,o) is a dual automorphism,
or we let a+b=:D and (o,a,b,c,d,e,k,h,g,f,i) \mapsto
                     (i,f,g,h,k,e,d,c,b,a,o) is such.
When possible, each sub-sub-...-case is followed by its dual.
SUBSIZE version Dec 29, 2018 (started at 1:8:43) reports:

|A|=10, A(without commas)={oiabcdefgh}. Constraints:
(always) f*g=d
 (C1)a+b=d
  (C1.a')f+g=i
   (C1.a1')a*b=o
    (C1.a1.a') d+e=g ;then f+e=f+d+e=f+g=i  a+e=a+b+e=d+e=g
       (so) f+e=i a+e=g
     (C1.a1.a1') d*e=b;then a*e=a*d*e=a*b=o  f*e=f*g*e=d*e=b
        (so) a*e=o f*e=b
      (C1.a1.a1.a) d*h=b ; then a*h=a*d*h=a*b=o
         (so)  a*h=o
       (C1.a1.a1.a1) d+c=g; then f+c=f+d+c=f+g=i
          (so)  f+c=i
        (C1.a1.a1.a1.a) e*h=b
Result for A=(H0/C1.a1.a1.a1.a):  |Sub(A)| = 317, that is,
sigma(A) = |Sub(A)|*2^(8-|A|) =  79.2500000000000000 .
  Few subuniverses, (C1.a1.a1.a1.a) is excluded.

|A|=11, A(without commas)={oiabcdefghx}. Constraints:
(always) f*g=d
(C1) a+b=d
 (C1.a')  f+g=i
  (C1.a1') a*b=o
   (C1.a1.a') d+e=g; then f+e=f+d+e=f+g=i and a+e=a+b+e=d+e=g
      (so) f+e=i  a+e=g
    (C1.a1.a1') d*e=b; then a*e=a*d*e=a*b=o f*e=f*g*e=d*e=b
        (so) a*e=o f*e=b
     (C1.a1.a1.a) d*h=b; then a*h=a*d*h=a*b=o
        (so) a*h=o
       (C1.a1.a1.a1)   d+c=g; then f+c=f+d+c=f+g=i
         (so) f+c=i
        (C1.a1.a1.a1.b) e*h=x; e*h=:x>b
Result for A=(H0/C1.a1.a1.a1.b):  |Sub(A)| = 616, that is,
sigma(A) = |Sub(A)|*2^(8-|A|) =  77.0000000000000000 .
  Few subuniverses, (C1.a1.a1.a1.b) is excluded.
  Thus, (C1.a1.a1.a1) is excluded

|A|=11, A(without commas)={oiabcdefghy}. Constraints:
(always) f*g=d
 (C1) a+b=d
  (C1.a')  f+g=i
   (C1.a1') a*b=o
    (C1.a1.a')   d+e=g; then f+e=f+d+e=f+g=i  a+e=a+b+e=d+e=g
     (so) f+e=i  a+e=g
     (C1.a1.a1') d*e=b; then a*e=a*d*e=a*b=o  f*e=f*g*e=d*e=b
         (so) a*e=o f*e=b
      (C1.a1.a1.a)    d*h=b; then a*h=a*d*h=a*b=o
         (so) a*h=o
       (C1.a1.a1.a2)d+c=y ; d+c=:y<g  then  y+e=g by #1
         (so) y+e=g ; by #1
Result for A=(H0/C1.a1.a1.a2):  |Sub(A)| = 648, that is,
sigma(A) = |Sub(A)|*2^(8-|A|) =  81.0000000000000000 .
  Few subuniverses, (C1.a1.a1.a2) is excluded.
  Thus, (C1.a1.a1.a) is excluded and  so is its dual, 
  since the set of stipulations before (C1.a1.a1.a) is 
  selfdual; these facts will used in (C1.a1.a1.b) below.

|A|=12, A(without commas)={oiabcdefghxy}. Constraints:
(always) f*g=d
 (C1) a+b=d
  (C1.a')   f+g=i
   (C1.a1') a*b=o
    (C1.a1.a')d+e=g; then f+e=f+d+e=f+g=i  a+e=a+b+e=d+e=g
      (so) f+e=i a+e=g
     (C1.a1.a1') d*e=b; then a*e=a*d*e=a*b=o  f*e=f*g*e=d*e=b
       (so) a*e=o f*e=b
      (C1.a1.a1.b) d*h=x d+c=y; then y<g  x*e=b  and y+e=g
         (so)  x*e=b  y+e=g
Result for A=(H0/C1.a1.a1.b):  |Sub(A)| = 1274, that is,
sigma(A) = |Sub(A)|*2^(8-|A|) =  79.6250000000000000 .
 Few subuniverses, (C1.a1.a1.b) is excluded.
 Thus, (C1.a1.a1') is excluded 
 (At present, we cannot say that its dual, (C1.a1.a'), is
 excluded; we can exclude only one of them now.)

|A|=12, A(without commas)={oiabcdefghxy}. Constraints:
(always) f*g=d
 (C1) a+b=d
  (C1.a') f+g=i
   (C1.a1') a*b=o
    (C1.a1.a') d+e=g; then f+e=f+d+e=f+g=i and a+e=a+b+e=d+e=g
      (so)  f+e=i  a+e=g
     (C1.a1.a2') d*e=x d*h=b;then x>b  f*e=f*g*e=d*e and #1
       (gives) f*e=x x*h=b
      (C1.a1.a2.a)  d*h=b; then a*h=a*d*h=a*b=o
        (so) a*h=o
       (C1.a1.a2.a1)  d+c=g; then f+c=f+d+c=f+g=i
         (so) f+c=i
Result for A=(H0/C1.a1.a2.a1):  |Sub(A)| = 1226, that is,
sigma(A) = |Sub(A)|*2^(8-|A|) =  76.6250000000000000 .
      Few subuniverses, (C1.a1.a2.a1) is excluded.

|A|=12, A(without commas)={oiabcdefghxz}. Constraints:
(always) f*g=d
 (C1) a+b=d
  (C1.a') f+g=i
   (C1.a1') a*b=o
    (C1.a1.a') d+e=g; then f+e=f+d+e=f+g=i and a+e=a+b+e=d+e=g
      (so) f+e=i a+e=g
     (C1.a1.a2') d*e=x d*h=b;then x>b f*e=f*g*e=d*e=x  x*h=b
       (so) f*e=x x*h=b
      (C1.a1.a2.a)   d*h=b; then a*h=a*d*h=a*b=o
       (C1.a1.a2.a2)  d+c=z; then d+c=:z<g and e+z=g by (#1}
        (so) e+z=g
Result for A=(H0/C1.a1.a2.a2):  |Sub(A)| = 1218, that is,
sigma(A) = |Sub(A)|*2^(8-|A|) =  76.1250000000000000 .
 Few subuniverses, (C1.a1.a2.a2) is excluded. Thus,
 (C1.a1.a2.a) is excluded and dually; this will be used
                 in (C1.a1.a2.b) below

|A|=13, A(without commas)={oiabcdefghxuv}. Constraints:
(always) f*g=d
 (C1) a+b=d
  (C1.a') f+g=i
   (C1.a1') a*b=o
    (C1.a1.a') d+e=g; then f+e=f+d+e=f+g=i and a+e=a+b+e=d+e=g
      (so) f+e=i a+e=g
     (C1.a1.a2') d*e=x d*h=b;then x>b f*e=f*g*e=d*e=x  x*h=b
       (so) f*e=x x*h=b
      (C1.a1.a2.b) d*h=u d+c=v; d*h=:u>b  d+c=:v<g
        (then #1 gives) u*e=b v+e=g
Result for A=(H0/C1.a1.a2.b):  |Sub(A)| = 2123, that is,
sigma(A) = |Sub(A)|*2^(8-|A|) =  66.3437500000000000 .
      Few subuniverses, (C1.a1.a2.b) is excluded.
      Thus, (C1.a1.a2') is excluded. 

|A|=12, A(without commas)={oiabcdefghxy}. Constraints:
(always) f*g=d
 (C1) a+b=d
  (C1.a') f+g=i
   (C1.a1') a*b=o
    (C1.a1.a') d+e=g; then f+e=f+d+e=f+g=i and a+e=a+b+e=d+e=g
      (so) f+e=i a+e=g
     (C1.a1.a3') d*e=x d*h=y e*h=b;then x>b  y>b and
       (#1 gives) x*h=b  y*e=b  x*y=b
Result for A=(H0/C1.a1.a3'):  |Sub(A)| = 1202, that is,
sigma(A) = |Sub(A)|*2^(8-|A|) =  75.1250000000000000 .
 Few subuniverses, (C1.a1.a3') is excluded. So (C1.a1.a') is
 excluded and dually; to be used in (C1.a1.b'), (C1.a1.c').

|A|=12, A(without commas)={oiabcdefghuv}. Constraints:
(always) f*g=d
 (C1) a+b=d
  (C1.a') f+g=i
   (C1.a1') a*b=o
    (C1.a1.b')d+e=v d+c=g d*e=u d*h=b;
      (then) v+c=g   u*h=b ; by #1
      (and) a+e=v f+c=i; since a+e=a+b+e=d+e  f+c=f+d+c=f+g
      (and) f*e=u a*h=o; since f*e=f*g*e=d*e  a*h=a*d*h=a*b
Result for A=(H0/C1.a1.b'):  |Sub(A)| = 1121, that is,
sigma(A) = |Sub(A)|*2^(8-|A|) =  70.0625000000000000 .
 Few subuniverses, (C1.a1.b') is excluded. This fact and its
 dual will be used below, according to (#1)

|A|=14, A(without commas)={oiabcdefghuvUV}. Constraints:
(always) f*g=d
 (C1) a+b=d
  (C1.a') f+g=i
   (C1.a1') a*b=o
    (C1.a1.c') d+e=v d+c=V e+c=g d*e=u d*h=U e*h=b
      (then) v+c=g e+V=g v+V=g  u*h=b e*U=b u*U=b; by #1
      (and) a+e=v f*e=u;since a+e=a+b+e=d+e=v  f*e=f*g*e=d*e=u
Result for A=(H0/C1.a1.c'):  |Sub(A)| = 3105, that is,
sigma(A) = |Sub(A)|*2^(8-|A|) =  48.5156250000000000 .
  Few subuniverses, (C1.a1.c') is excluded.
  Thus, (C1.a1') is excluded.

|A|=11, A(without commas)={oiabcdefghp}. Constraints:
(always) f*g=d
 (C1) a+b=d
  (C1.a') f+g=i
   (C1.a2') a*b=p  a*c=o; i.e.  a*b=:p>o and a*c=o
     (then) p*c=o;; at present  there is no duality.
    (C1.a2.a') d+e=g; then f+e=f+d+e=f+g=i and a+e=a+b+e=d+e=g
       (gives)  f+e=i a+e=g
     (C1.a2.a1')d*e=b;then a*e=a*d*e=a*b=p and f*e=f*g*e=d*e=b
        (so)  a*e=p  f*e=b
      (C1.a2.a1.a) d*h=b
        (then) a*h=p ; since a*h=a*d*h=a*b=p
Result for A=(H0/C1.a2.a1.a):  |Sub(A)| = 651, that is,
sigma(A) = |Sub(A)|*2^(8-|A|) =  81.3750000000000000 .
  Few subuniverses, (C1.a2.a1.a) is excluded.

|A|=12, A(without commas)={oiabcdefghpx}. Constraints:
(always) f*g=d
 (C1) a+b=d
  (C1.a')  f+g=i
   (C1.a2') a*b=p a*c=o; p>o  there is no duality now
     (then) p*c=o
    (C1.a2.a')d+e=g; then f+e=f+d+e=f+g=i and a+e=a+b+e=d+e=g
      (gives) f+e=i  a+e=g
     (C1.a2.a1') d*e=b;then a*e=a*d*e=a*b=p  f*e=f*g*e=d*e=b
       (give) a*e=p  f*e=b
      (C1.a2.a1.b)  d*h=x; then x>b and x*e=b by #1
        (so)  x*e=b; by #1
Result for A=(H0/C1.a2.a1.b):  |Sub(A)| = 1261, that is,
sigma(A) = |Sub(A)|*2^(8-|A|) =  78.8125000000000000 .
  Few subuniverses, (C1.a2.a1.b) is excluded.
  Thus, (C1.a2.a1') is excluded.

|A|=13, A(without commas)={oiabcdefghpxy}. Constraints:
(always) f*g=d
 (C1) a+b=d
  (C1.a') f+g=i
   (C1.a2') a*b=p a*c=o; p>o and there is no duality;
     (then) p*c=o
    (C1.a2.a') d+e=g;then f+e=f+d+e=f+g=i and a+e=a+b+e=d+e=g
        (give) f+e=i a+e=g
     (C1.a2.a2') d*e=x d*h=b; then x>b
       (and) f*e=x a*h=p;since f*e=f*g*e=d*e  a*h=a*d*h=a*b=p
       (and) x*h=b ; by #1
Result for A=(H0/C1.a2.a2'):  |Sub(A)| = 2384, that is,
sigma(A) = |Sub(A)|*2^(8-|A|) =  74.5000000000000000 .
      Few subuniverses, (C1.a2.a2') is excluded.

|A|=14, A(without commas)={oiabcdefghpxXy}. Constraints:
(always) f*g=d
 (C1) a+b=d
  (C1.a') f+g=i
   (C1.a2') a*b=p a*c=o; p>o  now there is no duality.
    (C1.a2.a') d+e=g;then f+e=f+d+e=f+g=i and a+e=a+b+e=d+e=g
      (so) f+e=i  a+e=g
     (C1.a2.a3') d*e=x d*h=X e*h=b; then d*e=:x>b  d*h=:X>b
       (and) f*e=x x*h=b; since  f*e=f*g*e=d*e=x and x*h=b
       (and)  e*X=b x*X=b;  by #1
Result for A=(H0/C1.a2.a3'):  |Sub(A)| = 4366, that is,
sigma(A) = |Sub(A)|*2^(8-|A|) =  68.2187500000000000 .
      Few subuniverses, (C1.a2.a3') is excluded.
      Thus, (C1.a2.a') is excluded.

|A|=12, A(without commas)={oiabcdefghpv}. Constraints:
(always) f*g=d
 (C1) a+b=d
  (C1.a')  f+g=i
   (C1.a2') a*b=p a*c=o; p>o  no duality
     (and) p*c=o
    (C1.a2.b') d+e=v d+c=g;then v<g
      (and) v+c=g; by #1  and a+e=a+b+e=d+e=v  f+c=f+d+c=f+g=i
      (gives) a+e=v f+c=i
     (C1.a2.b1) p+c=g;
       (then) a+c=g b+c=g e+c=g
      (C1.a2.b1.a) f*h=b
        (then) d*h=b
Result for A=(H0/C1.a2.b1.a):  |Sub(A)| = 1139, that is,
sigma(A) = |Sub(A)|*2^(8-|A|) =  71.1875000000000000 .
      Few subuniverses, (C1.a2.b1.a)  is excluded.

|A|=13, A(without commas)={oiabcdefghpvx}. Constraints:
(always) f*g=d
 (C1) a+b=d
  (C1.a') f+g=i
   (C1.a2') a*b=p a*c=o;  then p>o  there is no duality
     (and)  p*c=o
    (C1.a2.b') d+e=v  d+c=g; then v<g  #1
      (gives) v+c=g; and a+e=a+b+e=d+e=v  f+c=f+d+c=f+g=i
      (yields) a+e=v f+c=i
     (C1.a2.b1) p+c=g
       (then) a+c=g b+c=g e+c=g
      (C1.a2.b1.b) f*h=x ; then x>b
Result for A=(H0/C1.a2.b1.b):  |Sub(A)| = 2577, that is,
sigma(A) = |Sub(A)|*2^(8-|A|) =  80.5312500000000000 .
   Few subuniverses, (C1.a2.b1.b)  is excluded.
   Thus,  (C1.a2.b1) is excluded.

|A|=14, A(without commas)={oiabcdefghpvxy}. Constraints:
(always) f*g=d
 (C1) a+b=d
  (C1.a') f+g=i
   (C1.a2') a*b=p a*c=o; then p>o  there is no duality
     (and) p*c=o
    (C1.a2.b') d+e=v d+c=g;then v<g
      (and) v+c=g; and a+e=a+b+e=d+e=v  f+c=f+d+c=f+g=i
      (that is) a+e=v  f+c=i
     (C1.a2.b2) p+c=y; then y<g
      (C1.a2.b2.a) f*h=b;
        (then) d*h=b
Result for A=(H0/C1.a2.b2.a):  |Sub(A)| = 4570, that is,
sigma(A) = |Sub(A)|*2^(8-|A|) =  71.4062500000000000 .
   Few subuniverses, (C1.a2.b2.a)  is excluded.

|A|=14, A(without commas)={oiabcdefghpvxy}. Constraints:
(always) f*g=d
 (C1) a+b=d
  (C1.a') f+g=i
   (C1.a2') a*b=p a*c=o; then p>o  there is no duality
     (and) p*c=o
    (C1.a2.b') d+e=v d+c=g;then v<g
      (and) v+c=g; and a+e=a+b+e=d+e=v  f+c=f+d+c=f+g=i
      (that is) a+e=v  f+c=i
     (C1.a2.b2) p+c=y; then y<g
      (C1.a2.b2.b) f*h=x; then x>b
Result for A=(H0/C1.a2.b2.b):  |Sub(A)| = 5143, that is,
sigma(A) = |Sub(A)|*2^(8-|A|) =  80.3593750000000000 .
      Few subuniverses, (C1.a2.b2.b)  is excluded.
      Thus,  (C1.a2.b2) is excluded.   
      Thus,  (C1.a2.b') is excluded. 


 (C1) a+b=d (in addition to f*g=d)%
  (C1.a')   f+g=i
   (C1.a2') a*b=:p>o and a*c=o,  then p*c=o  
       (At present, there is no duality.)
    (C1.a2.c') d+e=:v<g, d+c=:V<g and e+c=g, then 
      a+e=a+b+e=d+e=v, and (#1) gives
         v+c=g, e+V=g, and v+V=g              
|A|=13, A(without commas)={oiabcdefghpvV}. Constraints:
(always) f*g=d
 a+b=d
  f+g=i
   a*b=p  a*c=o   p*c=o
    d+e=v d+c=V e+c=g   a+e=v v+c=g e+V=g v+V=g
Result for A=(H0/C1.a2.c'):  |Sub(A)| = 2354, that is,
sigma(A) = |Sub(A)|*2^(8-|A|) =  73.5625000000000000 .
      Few subuniverses, (C1.a2.c')  is excluded.
      Thus, (C1.a2') is excluded.

 (C1) a+b=d (in addition to f*g=d)%
  (C1.a')   f+g=i
   (C1.a3') a*b=:p>o, a*c=:P>o and b*c=o  
      (At present, there is no duality.) Then
            p*c=o, P*b=o and p*P=o
    (C1.a3.a') d+e=g, then f+e=f+d+e=f+g=i, a+e=a+b+e=d+e=g 
     (C1.a3.a1') d*e=b, then a*e=a*d*e=a*b=p, f*e=f*g*e=d*e=b 
|A|=12, A(without commas)={oiabcdefghpP}. Constraints:
(always) f*g=d
 a+b=d
  f+g=i
   a*b=p a*c=P b*c=o    p*c=o P*b=o p*P=o
    d+e=g   f+e=i  a+e=g
     d*e=b   a*e=p f*e=b
Result for A=(H0/C1.a3.a1'):  |Sub(A)| = 1218, that is,
sigma(A) = |Sub(A)|*2^(8-|A|) =  76.1250000000000000 .
      Few subuniverses, (C1.a3.a1')  is excluded.

 (C1) a+b=d (in addition to f*g=d)%
  (C1.a')   f+g=i
   (C1.a3') a*b=:p>o, a*c=:P>o and b*c=o  (At present, there
       is no duality.) Then  p*c=o, P*b=o and p*P=o
    (C1.a3.a') d+e=g, then f+e=f+d+e=f+g=i, a+e=a+b+e=d+e=g 
     (C1.a3.a2') d*e=:x>b and d*h=b, then f*e=f*g*e=d*e=x, 
           a*h=a*d*h=a*b=p, and x*h=b by (#1) % 
|A|=13, A(without commas)={oiabcdefghpPx}. Constraints:
(always) f*g=d
 a+b=d
  f+g=i
   a*b=p a*c=P b*c=o    p*c=o P*b=o p*P=o
    d+e=g   f+e=i  a+e=g
     d*e=x d*h=b   f*e=x a*h=p x*h=b
Result for A=(H0/C1.a3.a2'):  |Sub(A)| = 1882, that is,
sigma(A) = |Sub(A)|*2^(8-|A|) =  58.8125000000000000 .
      Few subuniverses, (C1.a3.a2') is excluded.

 (C1) a+b=d (in addition to f*g=d)%
  (C1.a')   f+g=i
   (C1.a3') a*b=:p>o, a*c=:P>o and b*c=o  (At present, 
 there is no duality.) Then p*c=o, P*b=o and p*P=o
    (C1.a3.a') d+e=g, then f+e=f+d+e=f+g=i, a+e=a+b+e=d+e=g 
     (C1.a3.a3') d*e=:x>b, d*h=:X>b and e*h=b, 
       then f*e=f*g*e=d*e=x, and x*h=b e*X=b and x*X=b by #1 
|A|=14, A(without commas)={oiabcdefghpPxX}. Constraints:
(always) f*g=d
 a+b=d
  f+g=i
   a*b=p a*c=P b*c=o    p*c=o P*b=o p*P=o
    d+e=g   f+e=i  a+e=g
     d*e=x d*h=X e*h=b   f*e=x x*h=b e*X=b x*X=b
Result for A=(H0/C1.a3.a3'):  |Sub(A)| = 3114, that is,
sigma(A) = |Sub(A)|*2^(8-|A|) =  48.6562500000000000 .
      Few subuniverses, (C1.a3.a3') is excluded.
      Thus, (C1.a3.a') is excluded.

 (C1) a+b=d (in addition to f*g=d)%
  (C1.a')   f+g=i
   (C1.a3') a*b=:p>o, a*c=:P>o and b*c=o. (At present, there
     is no duality.) Then p*c=o, P*b=o and p*P=o.
    (C1.a3.b') d+e=:v<g and d+c=g, then v+c=g by (#1), 
      a+e=a+b+e=d+e=v, f+c=f+d+c=f+g=i 
|A|=13, A(without commas)={oiabcdefghpPv}. Constraints:
(always) f*g=d
 a+b=d
  f+g=i
   a*b=p a*c=P b*c=o    p*c=o P*b=o p*P=o
    d+e=v d+c=g   v+c=g a+e=v f+c=i
Result for A=(H0/C1.a3.b'):  |Sub(A)| = 2428, that is,
sigma(A) = |Sub(A)|*2^(8-|A|) =  75.8750000000000000 .
      Few subuniverses, (C1.a3.b') is excluded.

 (C1) a+b=d (in addition to f*g=d)%
  (C1.a')   f+g=i
   (C1.a3') a*b=:p>o, a*c=:P>o and b*c=o (At present, there
     is no duality.) Then p*c=o, P*b=o and p*P=o.
    (C1.a3.c') d+e=:v<g, d+c=:V<g and e+c=g, then 
      a+e=a+b+e=d+e=v, and (#1) gives v+c=g, e+V=g, and v+V=g 
|A|=14, A(without commas)={oiabcdefghpPvV}. Constraints:
(always) f*g=d
 a+b=d
  f+g=i
   a*b=p a*c=P b*c=o    p*c=o P*b=o p*P=o
    d+e=v d+c=V e+c=g   a+e=v v+c=g e+V=g v+V=g
Result for A=(H0/C1.a3.c'):  |Sub(A)| = 3856, that is,
sigma(A) = |Sub(A)|*2^(8-|A|) =  60.2500000000000000 .
 Few subuniverses, (C1.a3.c') is excluded.
 Thus, (C1.a3') is excluded.
 Thus, (C1.a') is excluded. Its dual is also excluded; 
 (C1.b') and (C1.c') will use this fact. 

 (C1) a+b=d (in addition to f*g=d)%
  (C1.b') f+g=:T<i, a*b=:t>o, and f+h=i; then T+h=i, t*c=o
    and there is no duality
   (C1.b1')  d+e=g, then f+e=f+d+e=f+g=T and a+e=a+b+e=d+e=g
    (C1.b1.a') d*e=b, then a*e=a*d*e=a*b=t, f*e=f*g*e=d*e=b
|A|=12, A(without commas)={oiabcdefghTt}. Constraints:
(always) f*g=d
 a+b=d
  f+g=T a*b=t  f+h=i  T+h=i  t*c=o
   d+e=g  f+e=T  a+e=g
    d*e=b  a*e=t f*e=b
Result for A=(H0/C1.b1.a'):  |Sub(A)| = 1310, that is,
sigma(A) = |Sub(A)|*2^(8-|A|) =  81.8750000000000000 .
      Few subuniverses, (C1.b1.a') is excluded.

 (C1) a+b=d (in addition to f*g=d)%
  (C1.b') f+g=:T<i, a*b=:t>o, and f+h=i; then T+h=i, t*c=o 
     and there is no duality
   (C1.b1')  d+e=g, then f+e=f+d+e=f+g=T and a+e=a+b+e=d+e=g
    (C1.b1.b') d*e=:x>b and d*h=b, then 
       f*e=f*g*e=d*e=x, and x*h=b 
|A|=13, A(without commas)={oiabcdefghTtx}. Constraints:
(always) f*g=d
 a+b=d
  f+g=T a*b=t  f+h=i  T+h=i t*c=o
   d+e=g  f+e=T  a+e=g
    d*e=x  d*h=b   f*e=x x*h=b
Result for A=(H0/C1.b1.b'):  |Sub(A)| = 2226, that is,
sigma(A) = |Sub(A)|*2^(8-|A|) =  69.5625000000000000 .
    Few subuniverses, (C1.b1.b') is excluded.

 (C1) a+b=d (in addition to f*g=d)%
  (C1.b') f+g=:T<i, a*b=:t>o, and f+h=i; 
    then T+h=i, t*c=o and there is no duality
   (C1.b1')  d+e=g, then f+e=f+d+e=f+g=T and a+e=a+b+e=d+e=g
    (C1.b1.c') d*e=:x>b, d*h=:X>b, and e*h=b, 
      then f*e=f*g*e=d*e=x, x*h=b, X*e=b, and x*X=b 
|A|=14, A(without commas)={oiabcdefghTtxX}. Constraints:
(always) f*g=d
 a+b=d
  f+g=T a*b=t  f+h=i  T+h=i t*c=o
   d+e=g  f+e=T  a+e=g
    d*e=x d*h=X e*h=b   f*e=x x*h=b  X*e=b  x*X=b
Result for A=(H0/C1.b1.c'):  |Sub(A)| = 3602, that is,
sigma(A) = |Sub(A)|*2^(8-|A|) =  56.2812500000000000 .
      Few subuniverses, (C1.b1.c') is excluded.
      Thus, (C1.b1') is excluded.       

 (C1) a+b=d (in addition to f*g=d)%
  (C1.b') f+g=:T<i, a*b=:t>o, and f+h=i; 
     then T+h=i, t*c=o and there is no duality
   (C1.b2') d+e=:v<g and d+c=g, then v+c=g by (#1), 
         a+e=a+b+e=d+e=v, f+c=f+d+c=f+g=T
|A|=13, A(without commas)={oiabcdefghTtv}. Constraints:
(always) f*g=d
 a+b=d
  f+g=T a*b=t  f+h=i  T+h=i t*c=o
   d+e=v d+c=g   v+c=g  a+e=v f+c=T
Result for A=(H0/C1.b2'):  |Sub(A)| = 2527, that is,
sigma(A) = |Sub(A)|*2^(8-|A|) =  78.9687500000000000 .
      Few subuniverses, (C1.b2') is excluded.

 (C1) a+b=d (in addition to f*g=d)%
  (C1.b') f+g=:T<i, a*b=:t>o, and f+h=i; 
      then T+h=i, t*c=o and there is no duality
   (C1.b3') d+e=:v<g, d+c=:V<g and e+c=g, 
      then v+c=g, e+V=g, v+V=g,  a+e=a+b+e=d+e=v
|A|=14, A(without commas)={oiabcdefghTtvV}. Constraints:
(always) f*g=d
 a+b=d
  f+g=T a*b=t  f+h=i  T+h=i t*c=o
   d+e=v d+c=V e+c=g   v+c=g e+V=g v+V=g a+e=v
Result for A=(H0/C1.b3'):  |Sub(A)| = 4038, that is,
sigma(A) = |Sub(A)|*2^(8-|A|) =  63.0937500000000000 .
      Few subuniverses, (C1.b3') is excluded.
      Thus, (C1.b') is excluded.
 Now let us summarize where we are. Together with the 
 automatic f*g=d, (C1) yields a selfdual situation. In this
 selfdual situation, (C1.a') and its dual have been excluded. 
 Since (C1.b) is also excluded, so is its dual. By (#1), two
 of f,g, and h join to i; but neither f+g by (C1.a), nor f+h
 by (C1.b).  By duality, neither a*b, nor a*c is o. This fact 
 explain the form of (C1.c') below.
 

 (C1) a+b=d (in addition to f*g=d)%
  (C1.c') f+g=:T<i, f+h=:S<i,g+h=i,a*b=:t>o,a*c=:s>o, b*c=o, 
    then  T+h=i, g+S=i, T+S=i, t*c=o, b*s=o, t*s=o
|A|=14, A(without commas)={oiabcdefghTtSs}. Constraints:
(always) f*g=d
 a+b=d
  f+g=T f+h=S g+h=i a*b=t a*c=s b*c=o  T+h=i g+S=i T+S=i
  t*c=o b*s=o t*s=o
Result for A=(H0/C1.c'):  |Sub(A)| = 4916, that is,
sigma(A) = |Sub(A)|*2^(8-|A|) =  76.8125000000000000 .
      Few subuniverses, (C1.c') is excluded.
      Thus, (C1) is excluded.
      

 (C2) a+b=:D<d (in addition to f*g=d; selfdual situation)%
  (C2.a')   f+g=i
   (C2.a1') a*b=o  (again, the situation is selfdual)
    (C2.a1.a')   d+e=g, then f+e=f+d+e=f+g=i 
     (C2.a1.a1') D*e=b, then a*e=a*D*e=a*b=o (selfdual)
      (C2.a1.a1.a) D+e=g, then a+e=a+b+e=D+e=g
       (C2.a1.a1.a1) d*e=b, then f*e=f*g*e=d*e=b (selfdual) 
        (C2.a1.a1.a1.a)  D*h=b, then a*h=a*D*h=a*b=o
         (C2.a1.a1.a1.a1) d+c=g, f+c=f+d+c=f+g=i (selfdual)
|A|=11, A(without commas)={oiabcdefghD}. Constraints:
(always) f*g=d
 a+b=D
  f+g=i
   a*b=o
    d+e=g   f+e=i
     D*e=b   a*e=o
      D+e=g   a+e=g
       d*e=b   f*e=b
        D*h=b   a*h=o
         d+c=g   f+c=i
Result for A=(H0/C2.a1.a1.a1.a1):  |Sub(A)| = 605, that is,
sigma(A) = |Sub(A)|*2^(8-|A|) =  75.6250000000000000 .
      Few subuniverses, (C2.a1.a1.a1.a1) is excluded.

 (C2) a+b=:D<d (in addition to f*g=d; selfdual situation
  (C2.a')   f+g=i
   (C2.a1') a*b=o  (again, the situation is selfdual)
    (C2.a1.a')  d+e=g, then f+e=f+d+e=f+g=i 
     (C2.a1.a1') D*e=b, then a*e=a*D*e=a*b=o (selfdual)
      (C2.a1.a1.a) D+e=g, then a+e=a+b+e=D+e=g
       (C2.a1.a1.a1) d*e=b, then f*e=f*g*e=d*e=b (selfdual)
        (C2.a1.a1.a1.a)  D*h=b, then a*h=a*D*h=a*b=o
         (C2.a1.a1.a1.a2) d+c=:u<g, then u+e=g 
|A|=12, A(without commas)={oiabcdefghDu}. Constraints:
(always) f*g=d
 a+b=D
  f+g=i
   a*b=o
    d+e=g   f+e=i
     D*e=b   a*e=o
      D+e=g   a+e=g
       d*e=b   f*e=b
        D*h=b   a*h=o
         d+c=u  u+e=g
Result for A=(H0/C2.a1.a1.a1.a2):  |Sub(A)| = 1203, that is,
sigma(A) = |Sub(A)|*2^(8-|A|) =  75.1875000000000000 .
      Few subuniverses, (C2.a1.a1.a1.a2) is excluded.
      Thus, (C2.a1.a1.a1.a) is excluded.

 (C2) a+b=:D<d (in addition to f*g=d; selfdual situation)
  (C2.a')   f+g=i
   (C2.a1') a*b=o  (again, the situation is selfdual)
    (C2.a1.a')   d+e=g, then f+e=f+d+e=f+g=i 
     (C2.a1.a1') D*e=b, then a*e=a*D*e=a*b=o (selfdual)
      (C2.a1.a1.a) D+e=g, then a+e=a+b+e=D+e=g
       (C2.a1.a1.a1) d*e=b, then f*e=f*g*e=d*e=b (selfdual)
           By duality, together with (C2.a1.a1.a1.a), its
           dual,  d+c=g is also excluded! Thus: 
        (C2.a1.a1.a1.b) D*h=:B>b & d+c=:G<g,then e*B=b,e+G=g
|A|=13, A(without commas)={oiabcdefghDBG}. Constraints:
(always) f*g=d
 a+b=D
  f+g=i
   a*b=o
    d+e=g   f+e=i
     D*e=b   a*e=o
      D+e=g   a+e=g
       d*e=b   f*e=b
        D*h=B d+c=G  e*B=b e+G=g
Result for A=(H0/C2.a1.a1.a1.b):  |Sub(A)| = 2426, that is,
sigma(A) = |Sub(A)|*2^(8-|A|) =  75.8125000000000000 .
      Few subuniverses, (C2.a1.a1.a1.b) is excluded.
      Thus, (C2.a1.a1.a1) is excluded.

 (C2) a+b=:D<d (in addition to f*g=d; selfdual situation)
  (C2.a')   f+g=i
   (C2.a1') a*b=o  (again, the situation is selfdual)
    (C2.a1.a')   d+e=g, then f+e=f+d+e=f+g=i 
     (C2.a1.a1') D*e=b, then a*e=a*D*e=a*b=o (selfdual)
      (C2.a1.a1.a) D+e=g, then a+e=a+b+e=D+e=g
       (C2.a1.a1.a2) d*e=:B, then D*B=b, a*B=a*D*B=a*b=o 
           and f*e=f*g*e=d*e=B % 87.25
        (C2.a1.a1.a2.a) g+h=i
|A|=13, A(without commas)={oiabcdefghDBG}. Constraints:
(always) f*g=d
 a+b=D
  f+g=i
   a*b=o
    d+e=g   f+e=i
     D*e=b   a*e=o
      D+e=g   a+e=g
       d*e=B   D*B=b a*B=o f*e=B
        g+h=i
Result for A=(H0/C2.a1.a1.a2.a):  |Sub(A)| = 2528, that is,
sigma(A) = |Sub(A)|*2^(8-|A|) =  79.0000000000000000 .
      Few subuniverses, (C2.a1.a1.a2.a) is excluded.

 (C2) a+b=:D<d (in addition to f*g=d; selfdual)%
  (C2.a')   f+g=i
   (C2.a1') a*b=o  (again, the situation is selfdual)
    (C2.a1.a')   d+e=g, then f+e=f+d+e=f+g=i 
     (C2.a1.a1') D*e=b, then a*e=a*D*e=a*b=o selfdual)
      (C2.a1.a1.a) D+e=g, then a+e=a+b+e=D+e=g
       (C2.a1.a1.a2) d*e=:B, 
        then D*B=b, a*B=a*D*B=a*b=o and f*e=f*g*e=d*e=B
        (C2.a1.a1.a2.b) g+h=:p<i, then f+p=i
|A|=14, A(without commas)={oiabcdefghDBGp}. Constraints:
(always) f*g=d
 a+b=D
  f+g=i
   a*b=o
    d+e=g   f+e=i
     D*e=b   a*e=o
      D+e=g   a+e=g
       d*e=B   D*B=b a*B=o f*e=B
        g+h=p   f+p=i
Result for A=(H0/C2.a1.a1.a2.b):  |Sub(A)| = 4604, that is,
sigma(A) = |Sub(A)|*2^(8-|A|) =  71.9375000000000000 .
      Few subuniverses, (C2.a1.a1.a2.b) is excluded.
      Thus, (C2.a1.a1.a2) is excluded.
      Thus, (C2.a1.a1.a) is excluded.

 (C2) a+b=:D<d (selfdual situation)
  (C2.a')   f+g=i
   (C2.a1') a*b=o  (again, the situation is selfdual)
    (C2.a1.a')   d+e=g, then f+e=f+d+e=f+g=i 
     (C2.a1.a1') D*e=b, then a*e=a*D*e=a*b=o (selfdual)
      (C2.a1.a1.b) D+e=:x<g, 
         then a+e=a+b+e=D+e=x, d+x=g, f+x=f+d+x=f+g=i
       (C2.a1.a1.b1) d*e=b, then f*e=f*g*e=d*e=b
        (C2.a1.a1.b1.a) g+h=i 

|A|=12, A(without commas)={oiabcdefghDx}. Constraints:
(always) f*g=d
 a+b=D
  f+g=i
   a*b=o
    d+e=g   f+e=i
     D*e=b   a*e=o
      D+e=x   a+e=x  d+x=g f+x=i
       d*e=b   f*e=b
        g+h=i
Result for A=(H0/C2.a1.a1.b1.a):  |Sub(A)| = 1254, that is,
sigma(A) = |Sub(A)|*2^(8-|A|) =  78.3750000000000000 .
      Few subuniverses, (C2.a1.a1.b1.a) is excluded.

 (C2) a+b=:D<d (in addition to f*g=d; selfdual)%
  (C2.a')   f+g=i
   (C2.a1') a*b=o  (again, the situation is selfdual)
    (C2.a1.a')   d+e=g, then f+e=f+d+e=f+g=i 
     (C2.a1.a1') D*e=b, then a*e=a*D*e=a*b=o (selfdual)
      (C2.a1.a1.b) D+e=:x<g, 
        then a+e=a+b+e=D+e=x, d+x=g, f+x=f+d+x=f+g=i 
       (C2.a1.a1.b1) d*e=b, then f*e=f*g*e=d*e=b 
        (C2.a1.a1.b1.b)  g+h=:p<i, then f+p=i
|A|=13, A(without commas)={oiabcdefghDxp}. Constraints:
(always) f*g=d
 a+b=D
  f+g=i
   a*b=o
    d+e=g   f+e=i
     D*e=b   a*e=o
      D+e=x   a+e=x  d+x=g f+x=i
       d*e=b   f*e=b
        g+h=p   f+p=i
Result for A=(H0/C2.a1.a1.b1.b):  |Sub(A)| = 2338, that is,
sigma(A) = |Sub(A)|*2^(8-|A|) =  73.0625000000000000 .
      Few subuniverses, (C2.a1.a1.b1.b) is excluded.
      Thus, (C2.a1.a1.b1) is excluded.

 (C2) a+b=:D<d (in addition to f*g=d; selfdual situation)
  (C2.a')   f+g=i
   (C2.a1') a*b=o  (again, the situation is selfdual)
    (C2.a1.a')   d+e=g, then f+e=f+d+e=f+g=i 
     (C2.a1.a1') D*e=b, then a*e=a*D*e=a*b=o selfdual)
      (C2.a1.a1.b) D+e=:x<g, 
        then a+e=a+b+e=D+e=x, d+x=g, f+x=f+d+x=f+g=i
       (C2.a1.a1.b2) d*e=:y>b, 
        then f*e=f*g*e=d*e=y and D*y=D*e*d=b*d=b
|A|=13, A(without commas)={oiabcdefghDxy}. Constraints:
(always) f*g=d
 a+b=D
  f+g=i
   a*b=o
    d+e=g   f+e=i
     D*e=b   a*e=o
      D+e=x   a+e=x  d+x=g f+x=i
       d*e=y   f*e=y  D*y=b
Result for A=(H0/C2.a1.a1.b2):  |Sub(A)| = 2424, that is,
sigma(A) = |Sub(A)|*2^(8-|A|) =  75.7500000000000000 .
      Few subuniverses, (C2.a1.a1.b2)  is excluded.
      Thus, (C2.a1.a1.b) is excluded.
      Thus, (C2.a1.a1') is excluded.

 (C2) a+b=:D<d (in addition to f*g=d; selfdual situation)
  (C2.a')   f+g=i
   (C2.a1') a*b=o  (again, the situation is selfdual)
    (C2.a1.a')   d+e=g, then f+e=f+d+e=f+g=i 
     (C2.a1.a2') D*e=:x>b and D*h=b, 
        then  x*h=b by (#1) and a*h=a*D*h=a*b=o
      (C2.a1.a2.a) g+h=i
       (C2.a1.a2.a1) f+h=i  
        (C2.a1.a2.a1.a) a*c=o 
|A|=12, A(without commas)={oiabcdefghDx}. Constraints:
(always) f*g=d
 a+b=D
  f+g=i
   a*b=o
    d+e=g   f+e=i
     D*e=x D*h=b   x*h=b  a*h=o
      g+h=i
       f+h=i
        a*c=o
Result for A=C2.a1.a2.a1.a:  |Sub(A)| = 1258, that is,
sigma(A) = |Sub(A)|*2^(8-|A|) =  78.6250000000000000 .
      Few subuniverses, (C2.a1.a2.a1.a)  is excluded.

 (C2) a+b=:D<d (in addition to f*g=d; selfdual situation)
  (C2.a')   f+g=i
   (C2.a1') a*b=o  (again, the situation is selfdual)
    (C2.a1.a')   d+e=g, then f+e=f+d+e=f+g=i 
     (C2.a1.a2') D*e=:x>b and D*h=b,
 then  x*h=b by (#1) and a*h=a*D*h=a*b=o %
      (C2.a1.a2.a) g+h=i
       (C2.a1.a2.a1) f+h=i  
        (C2.a1.a2.a1.b) a*c=:y>o, then y*b=o 
|A|=13, A(without commas)={oiabcdefghDxy}. Constraints:
(always) f*g=d
 a+b=D
  f+g=i
   a*b=o
    d+e=g   f+e=i
     D*e=x D*h=b   x*h=b  a*h=o
      g+h=i
       f+h=i
        a*c=o
         a*c=y    y*b=o
Result for A=(H0/C2.a1.a2.a1.b):  |Sub(A)| = 2107, that is,
sigma(A) = |Sub(A)|*2^(8-|A|) =  65.8437500000000000 .
      Few subuniverses, (C2.a1.a2.a1.b)  is excluded.
      Thus, (C2.a1.a2.a1) is excluded.

 (C2) a+b=:D<d (in addition to f*g=d; selfudal situation)
  (C2.a')   f+g=i
   (C2.a1') a*b=o  (again, the situation is selfdual)
    (C2.a1.a')   d+e=g, then f+e=f+d+e=f+g=i 
     (C2.a1.a2') D*e=:x>b and D*h=b, 
         then  x*h=b by (#1) and a*h=a*D*h=a*b=o %
      (C2.a1.a2.a) g+h=i
       (C2.a1.a2.a2) f+h:=z<i, then z+g=i 
|A|=13, A(without commas)={oiabcdefghDxz}. Constraints:
(always) f*g=d
 a+b=D
  f+g=i
   a*b=o
    d+e=g   f+e=i
     D*e=x D*h=b   x*h=b  a*h=o
      g+h=i
       f+h=z    z+g=i
Result for A=(H0/C2.a1.a2.a2):  |Sub(A)| = 2422, that is,
sigma(A) = |Sub(A)|*2^(8-|A|) =  75.6875000000000000 .
      Few subuniverses, (C2.a1.a2.a2)  is excluded.
      Thus, (C2.a1.a2.a) is excluded.

 (C2) a+b=:D<d (in addition to f*g=d; selfdual situation)%
  (C2.a')   f+g=i
   (C2.a1') a*b=o  (again, the situation is selfdual)
    (C2.a1.a')   d+e=g, then f+e=f+d+e=f+g=i 
     (C2.a1.a2') D*e=:x>b and D*h=b, 
        then  x*h=b by (#1) and a*h=a*D*h=a*b=o %
      (C2.a1.a2.b) g+h:=y<i, then f+y=i 
|A|=13, A(without commas)={oiabcdefghDxy}. Constraints:
(always) f*g=d
 a+b=D
  f+g=i
   a*b=o
    d+e=g   f+e=i
     D*e=x D*h=b   x*h=b  a*h=o
      g+h=y   f+y=i
Result for A=(H0/C2.a1.a2.b):  |Sub(A)| = 2570, that is,
sigma(A) = |Sub(A)|*2^(8-|A|) =  80.3125000000000000 .
      Few subuniverses, (C2.a1.a2.b)  is excluded.
      Thus, (C2.a1.a2')  is excluded.

 (C2) a+b=:D<d (in addition to f*g=d; selfdual situation)
  (C2.a')   f+g=i
   (C2.a1') a*b=o  (again, the situation is selfdual)
    (C2.a1.a')   d+e=g, then f+e=f+d+e=f+g=i 
     (C2.a1.a3') D*e=:x>b, D*h=:y>b and e*h=b, 
        then x*h=b, e*y=b, and x*y=b 
|A|=13, A(without commas)={oiabcdefghDxy}. Constraints:
(always) f*g=d
 a+b=D
  f+g=i
   a*b=o
    d+e=g   f+e=i
     D*e=x D*h=y e*h=b   x*h=b e*y=b x*y=b
Result for A=(H0/C2.a1.a3'):  |Sub(A)| = 2412, that is,
sigma(A) = |Sub(A)|*2^(8-|A|) =  75.3750000000000000 .
 Few subuniverses, (C2.a1.a3')   is excluded.
 Thus, (C2.a1.a') is excluded. Since (C2.a1') describes
 a selfdual situation, the dual of (C2.a1.a')  is also
 excluded, that is, D*e is not b; 
 this fact will be used in (C2.a1.b') and (C2.a1.c').

 (C2) a+b=:D<d (in addition to f*g=d; selfdual situation)
  (C2.a')   f+g=i
   (C2.a1') a*b=o  (again, the situation is selfdual)
    (C2.a1.b') d+e=:J<g and d+c=g and, as mentioned above, 
      D*e=:j>b, then J+c=g, f+c=f+d+c=f+g=i, j*h=b. 
|A|=13, A(without commas)={oiabcdefghDJj}. Constraints:
(always) f*g=d
 a+b=D
  f+g=i
   a*b=o
    d+e=J d+c=g D*e=j   J+c=g f+c=i j*h=b
Result for A=(H0/C2.a1.b'):  |Sub(A)| = 2586, that is,
sigma(A) = |Sub(A)|*2^(8-|A|) =  80.8125000000000000 .
      Few subuniverses, (C2.a1.b')   is excluded.

 (C2) a+b=:D<d (in addition to f*g=d; selfdual situation)
  (C2.a')   f+g=i
   (C2.a1') a*b=o  (again, the situation is selfdual)
    (C2.a1.c') d+e=:J<g, d+c=:K<g, e+c=g, and D*e=:j>b, then
      J+c=g, K+e=g, J+K=g,  j*h=b 
|A|=14, A(without commas)={oiabcdefghDJjK}. Constraints:
(always) f*g=d
 a+b=D
  f+g=i
   a*b=o
    d+e=J d+c=K e+c=g D*e=j   J+c=g K+e=g J+K=g j*h=b
Result for A=(H0/C2.a1.c'):  |Sub(A)| = 4190, that is,
sigma(A) = |Sub(A)|*2^(8-|A|) =  65.4687500000000000 .
      Few subuniverses, (C2.a1.c')  is excluded.
      Thus, (C2.a1') is excluded (and NO selfduality now)

 (C2) a+b=:D<d (in addition to f*g=d; selfdual situation)%
  (C2.a')   f+g=i
   (C2.a2') a*b=:u>o, then u*c=o 
     (C2.a2.a')   d+e=g, then f+e=f+d+e=f+g=i 
      (C2.a2.a1') D*e=b, then a*e=a*D*e=a*b=u 
       (C2.a2.a1.a)  D+e=g, then a+e=a+b+e=D+e=g
        (C2.a2.a1.a1)  d*e=b, then f*e=f*g*e=d*e=b
         (C2.a2.a1.a1.a) D*h=b, then a*h=a*D*h=a*b=u  
|A|=12, A(without commas)={oiabcdefghDu}. Constraints:
(always) f*g=d
 a+b=D
  f+g=i
   a*b=u   u*c=o
    d+e=g   f+e=i
     D*e=b   a*e=u
      D+e=g    a+e=g
       d*e=b    f*e=b
        D*h=b   a*h=u
Result for A=(H0/C2.a2.a1.a1.a):  |Sub(A)| = 1216, that is,
sigma(A) = |Sub(A)|*2^(8-|A|) =  76.0000000000000000 .
   Few subuniverses, (C2.a2.a1.a1.a)  is excluded.

 (C2) a+b=:D<d (in addition to f*g=d; selfdual situation)
  (C2.a')   f+g=i
   (C2.a2') a*b=:u>o, then u*c=o 
     (C2.a2.a')   d+e=g, then f+e=f+d+e=f+g=i 
      (C2.a2.a1') D*e=b, then a*e=a*D*e=a*b=u 
       (C2.a2.a1.a)  D+e=g, then a+e=a+b+e=D+e=g
        (C2.a2.a1.a1)  d*e=b, then f*e=f*g*e=d*e=b
         (C2.a2.a1.a1.b) D*h=:B>b,  then e*B=b 

|A|=13, A(without commas)={oiabcdefghDuB}. Constraints:
(always) f*g=d
 a+b=D
  f+g=i
   a*b=u   u*c=o
    d+e=g   f+e=i
     D*e=b   a*e=u
      D+e=g    a+e=g
       d*e=b    f*e=b
        D*h=B    e*B=b
Result for A=(H0/C2.a2.a1.a1.b):  |Sub(A)| = 2470, that is,
sigma(A) = |Sub(A)|*2^(8-|A|) =  77.1875000000000000 .
      Few subuniverses, (C2.a2.a1.a1.b) is excluded.
      Thus, (C2.a2.a1.a1) is excluded.

 (C2) a+b=:D<d (in addition to f*g=d; selfdual situation)
  (C2.a')   f+g=i
   (C2.a2') a*b=:u>o, then u*c=o 
     (C2.a2.a')   d+e=g, then f+e=f+d+e=f+g=i 
      (C2.a2.a1') D*e=b, then a*e=a*D*e=a*b=u 
       (C2.a2.a1.a)  D+e=g, then a+e=a+b+e=D+e=g
        (C2.a2.a1.a2) d*e=:x>b, then D*x=b, f*e=f*g*e=d*e=x 
|A|=13, A(without commas)={oiabcdefghDux}. Constraints:
(always) f*g=d
 a+b=D
  f+g=i
   a*b=u   u*c=o
    d+e=g   f+e=i
     D*e=b   a*e=u
      D+e=g    a+e=g
       d*e=x   D*x=b f*e=x
Result for A=(H0/C2.a2.a1.a2):  |Sub(A)| = 2494, that is,
sigma(A) = |Sub(A)|*2^(8-|A|) =  77.9375000000000000 .
      Few subuniverses, (C2.a2.a1.a2) is excluded.
      Thus, (C2.a2.a1.a) is excluded.

 (C2) a+b=:D<d (in addition to f*g=d; seldfual situation)
  (C2.a')   f+g=i
   (C2.a2') a*b=:u>o, then u*c=o 
     (C2.a2.a')   d+e=g, then f+e=f+d+e=f+g=i 
      (C2.a2.a1') D*e=b, then a*e=a*D*e=a*b=u 
       (C2.a2.a1.b) D+e=:v<g, then d+v=g and a+e=a+b+e=D+e=v 
|A|=14, A(without commas)={oiabcdefghDuvx}. Constraints:
(always) f*g=d
 a+b=D
  f+g=i
   a*b=u   u*c=o
    d+e=g   f+e=i
     D*e=b   a*e=u
      D+e=v  d+v=g a+e=v
Result for A=(H0/C2.a2.a1.b):  |Sub(A)| = 5272, that is,
sigma(A) = |Sub(A)|*2^(8-|A|) =  82.3750000000000000 .
      Few subuniverses, (C2.a2.a1.b)  is excluded.
      Thus, (C2.a2.a1') is excluded.

 (C2) a+b=:D<d (in addition to f*g=d; selfdual situation)%
  (C2.a')   f+g=i
   (C2.a2') a*b=:u>o, then u*c=o 
     (C2.a2.a')   d+e=g, then f+e=f+d+e=f+g=i 
      (C2.a2.a2') D*e=:B>b and D*h=b; 
        then B*h=b and a*h=a*D*h=a*b=u  
|A|=13, A(without commas)={oiabcdefghDuB}. Constraints:
(always) f*g=d
 a+b=D
  f+g=i
   a*b=u   u*c=o
    d+e=g   f+e=i
     D*e=B D*h=b  B*h=b a*h=u
Result for A=(H0/C2.a2.a2'):  |Sub(A)| = 2542, that is,
sigma(A) = |Sub(A)|*2^(8-|A|) =  79.4375000000000000 .
      Few subuniverses, (C2.a2.a2')  is excluded.

 (C2) a+b=:D<d (in addition to f*g=d; selfdual situation)%
  (C2.a')   f+g=i
   (C2.a2') a*b=:u>o, then u*c=o 
     (C2.a2.a')   d+e=g, then f+e=f+d+e=f+g=i 
      (C2.a2.a3') D*e=:B>b,  D*h=H>b, and e*h=b; 
          then B*h=b, H*e=b, B*H=b 

|A|=14, A(without commas)={oiabcdefghDuBH}. Constraints:
(always) f*g=d
 a+b=D
  f+g=i
   a*b=u   u*c=o
    d+e=g   f+e=i
     D*e=B D*h=H e*h=b   B*h=b H*e=b B*H=b
Result for A=(H0/C2.a2.a3'):  |Sub(A)| = 4170, that is,
sigma(A) = |Sub(A)|*2^(8-|A|) =  65.1562500000000000 .
      Few subuniverses, (C2.a2.a3')  is excluded.
      Thus, (C2.a2.a') is excluded.

 (C2) a+b=:D<d (in addition to f*g=d; selfdual situation)%
  (C2.a')   f+g=i
   (C2.a2') a*b=:u>o, then u*c=o 
     (C2.a2.b') d+e=:G<g and d+c=g, 
        then G+c=g and f+c=f+d+c=f+g=i
      (C2.a2.b1) d*e=b, then f*e=f*g*e=d*e=b and D*e=b 
|A|=13, A(without commas)={oiabcdefghDuG}. Constraints:
(always) f*g=d
 a+b=D
  f+g=i
   a*b=u   u*c=o
    d+e=G d+c=g   G+c=g f+c=i
     d*e=b   f*e=b D*e=b
Result for A=(H0/C2.a2.b1):  |Sub(A)| = 2566, that is,
sigma(A) = |Sub(A)|*2^(8-|A|) =  80.1875000000000000 .
      Few subuniverses, (C2.a2.b1) is excluded.

 (C2) a+b=:D<d (in addition to f*g=d; selfdual situation)%
  (C2.a')   f+g=i
   (C2.a2') a*b=:u>o, then u*c=o 
     (C2.a2.b') d+e=:G<g and d+c=g, 
         then G+c=g and f+c=f+d+c=f+g=i
      (C2.a2.b2) d*e=:B>b, then f*e=f*g*e=d*e=B  
       (C2.a2.b2.a) f+h=i 
        (C2.a2.b2.a1) g+h=i
|A|=14, A(without commas)={oiabcdefghDuGB}. Constraints:
(always) f*g=d
 a+b=D
  f+g=i
   a*b=u   u*c=o
    d+e=G d+c=g   G+c=g f+c=i
     d*e=B   f*e=B
      f+h=i
       g+h=i
Result for A=(H0/C2.a2.b2.a1):  |Sub(A)| = 4833, that is,
sigma(A) = |Sub(A)|*2^(8-|A|) =  75.5156250000000000 .
      Few subuniverses, (C2.a2.b2.a1)  is excluded.

 (C2) a+b=:D<d (in addition to f*g=d; selfdual situation)%
  (C2.a')   f+g=i
   (C2.a2') a*b=:u>o, then u*c=o 
     (C2.a2.b') d+e=:G<g and d+c=g, 
        then G+c=g and f+c=f+d+c=f+g=i
      (C2.a2.b2) d*e=:B>b, then f*e=f*g*e=d*e=B  
       (C2.a2.b2.a) f+h=i 
        (C2.a2.b2.a2) g+h=:p<i, then f+p=i
|A|=15, A(without commas)={oiabcdefghDuGBp}. Constraints:
(always) f*g=d
 a+b=D
  f+g=i
   a*b=u   u*c=o
    d+e=G d+c=g   G+c=g f+c=i
     d*e=B   f*e=B
      f+h=i
       g+h=p   f+p=i
Result for A=(H0/C2.a2.b2.a2):  |Sub(A)| = 9155, that is,
sigma(A) = |Sub(A)|*2^(8-|A|) =  71.5234375000000000 .
      Few subuniverses, (C2.a2.b2.a2)  is excluded.
      Thus, (C2.a2.b2.a) is excluded.

 (C2) a+b=:D<d (in addition to f*g=d; selfdual situation)%
  (C2.a')   f+g=i
   (C2.a2') a*b=:u>o, then u*c=o 
     (C2.a2.b') d+e=:G<g and d+c=g, 
         then G+c=g and f+c=f+d+c=f+g=i
      (C2.a2.b2) d*e=:B>b, then f*e=f*g*e=d*e=B  
       (C2.a2.b2.b) f+h=:s<i then s+g=i 
|A|=15, A(without commas)={oiabcdefghDuGBs}. Constraints:
(always) f*g=d
 a+b=D
  f+g=i
   a*b=u   u*c=o
    d+e=G d+c=g   G+c=g f+c=i
     d*e=B   f*e=B
      f+h=s  s+g=i
Result for A=(H0/C2.a2.b2.b):  |Sub(A)| = 9384, that is,
sigma(A) = |Sub(A)|*2^(8-|A|) =  73.3125000000000000 .
      Few subuniverses, (C2.a2.b2.b) is excluded.
      Thus, (C2.a2.b2) is excluded.
      Thus, (C2.a2.b') is excluded.

 (C2) a+b=:D<d (in addition to f*g=d; selfdual situation)%
  (C2.a')   f+g=i
   (C2.a2') a*b=:u>o, then u*c=o 
     (C2.a2.c') d+e=:G<g, d+c=:s<g and e+c=g, 
        then G+c=g, e+s=g and G+s=g
|A|=15, A(without commas)={oiabcdefghDuGBs}. Constraints:
(always) f*g=d
 a+b=D
  f+g=i
   a*b=u   u*c=o
    d+e=G d+c=s e+c=g  G+c=g e+s=g G+s=g
Result for A=(H0/C2.a2.c'):  |Sub(A)| = 10092, that is,
sigma(A) = |Sub(A)|*2^(8-|A|) =  78.8437500000000000 .
      Few subuniverses, (C2.a2.c') is excluded.
      Thus, (C2.a2') is excluded.

 (C2) a+b=:D<d (in addition to f*g=d; selfdual situation)%
  (C2.a')   f+g=i
   (C2.a3') a*b=:u>o, a*c=:v>o, and b*c=o, 
        then u*c=o, b*v=o and u*v=o 
     (C2.a3.a')   d+e=g, then f+e=f+d+e=f+g=i
      (C2.a3.a1') D*e=b, then a*e=a*D*e=a*b=u
|A|=13, A(without commas)={oiabcdefghDuv}. Constraints:
(always) f*g=d
 a+b=D
  f+g=i
   a*b=u a*c=v b*c=o   u*c=o b*v=o u*v=o
    d+e=g  f+e=i
     D*e=b   a*e=u
Result for A=(H0/C2.a3.a1'):  |Sub(A)| = 2476, that is,
sigma(A) = |Sub(A)|*2^(8-|A|) =  77.3750000000000000 .
      Few subuniverses, (C2.a3.a1') is excluded.

 (C2) a+b=:D<d (in addition to f*g=d; selfdual situation)%
  (C2.a')   f+g=i
   (C2.a3') a*b=:u>o, a*c=:v>o, and b*c=o, 
       then u*c=o, b*v=o and u*v=o 
     (C2.a3.a')   d+e=g, then f+e=f+d+e=f+g=i
      (C2.a3.a2') D*e=:w>b and D*h=b, 
         then w*h=b and a*h=a*D*h=a*b=u
|A|=14, A(without commas)={oiabcdefghDuvw}. Constraints:
(always) f*g=d
 a+b=D
  f+g=i
   a*b=u a*c=v b*c=o   u*c=o b*v=o u*v=o
    d+e=g  f+e=i
     D*e=w D*h=b   w*h=b a*h=u
Result for A=(H0/C2.a3.a2'):  |Sub(A)| = 3897, that is,
sigma(A) = |Sub(A)|*2^(8-|A|) =  60.8906250000000000 .
      Few subuniverses, (C2.a3.a2') is excluded.

 (C2) a+b=:D<d (in addition to f*g=d; selfdual situation)%
  (C2.a')   f+g=i
   (C2.a3') a*b=:u>o, a*c=:v>o, and b*c=o, 
           then u*c=o, b*v=o and u*v=o 
     (C2.a3.a')   d+e=g, then f+e=f+d+e=f+g=i
      (C2.a3.a3') D*e=:w>b, D*h=:W>b, and e*h=b, 
         then w*h=b, e*W=b, and w*W=b
|A|=15, A(without commas)={oiabcdefghDuvwW}. Constraints:
(always) f*g=d
 a+b=D
  f+g=i
   a*b=u a*c=v b*c=o   u*c=o b*v=o u*v=o
    d+e=g  f+e=i
     D*e=w D*h=W e*h=b   w*h=b  e*W=b   w*W=b
Result for A=(H0/C2.a3.a3'):  |Sub(A)| = 6351, that is,
sigma(A) = |Sub(A)|*2^(8-|A|) =  49.6171875000000000 .
      Few subuniverses, (C2.a3.a3') is excluded.
      Thus, (C2.a3.a') is excluded.

 (C2) a+b=:D<d (in addition to f*g=d; selfdual situation)%
  (C2.a')   f+g=i
   (C2.a3') a*b=:u>o, a*c=:v>o, and b*c=o, 
        then u*c=o, b*v=o and u*v=o 
     (C2.a3.b')   d+e=:G<g and d+c=g, 
         then G+c=g and f+c=f+d+c=f+g=i
|A|=14, A(without commas)={oiabcdefghDuvG}. Constraints:
(always) f*g=d
 a+b=D
  f+g=i
   a*b=u a*c=v b*c=o   u*c=o b*v=o u*v=o
    d+e=G d+c=g   G+c=g f+c=i
Result for A=(H0/C2.a3.b'):  |Sub(A)| = 5112, that is,
sigma(A) = |Sub(A)|*2^(8-|A|) =  79.8750000000000000 .
      Few subuniverses, (C2.a3.b') is excluded.

 (C2) a+b=:D<d (in addition to f*g=d; selfdual situation)%
  (C2.a')   f+g=i
   (C2.a3') a*b=:u>o, a*c=:v>o, and b*c=o, 
        then u*c=o, b*v=o and u*v=o 
     (C2.a3.c') d+e=:G<g, d+c=:s<g, and e+c=g, 
          then G+c=g, e+s=g, G+s=g 
|A|=15, A(without commas)={oiabcdefghDuvGs}. Constraints:
(always) f*g=d
 a+b=D
  f+g=i
   a*b=u a*c=v b*c=o   u*c=o b*v=o u*v=o
    d+e=G d+c=s e+c=g   G+c=g e+s=g G+s=g
Result for A=(H0/C2.a3.c'):  |Sub(A)| = 8018, that is,
sigma(A) = |Sub(A)|*2^(8-|A|) =  62.6406250000000000 .
      Few subuniverses, (C2.a3.c') is excluded.
      Thus, (C2.a3') is excluded.
      Thus, (C2.a') is excluded.
  Note that (C2) describes a selfdual situation. Hence, the
  dual of (C2.a') is also excluded, i.e., a*b>o; this will be
  used in (C2.b') and (C2.c') !!!

 (C2) a+b=:D<d (in addition to f*g=d; selfdual situation)%
  (C2.b') f+g=:I<i, f+h=i, a*b=:Z>o,  then I+h=i and Z*c=o
    (C2.b1')  d+e=g, then f+e=f+d+e=f+g=I
     (C2.b1.a') D*e=b, then a*e=a*D*e=a*b=Z  
      (C2.b1.a1) a*c=o

|A|=13, A(without commas)={oiabcdefghDIZ}. Constraints:
(always) f*g=d
 a+b=D
  f+g=I f+h=i a*b=Z   I+h=i Z*c=o
   d+e=g   f+e=I
    D*e=b   a*e=Z
     a*c=o
Result for A=(H0/C2.b1.a1):  |Sub(A)| = 2578, that is,
sigma(A) = |Sub(A)|*2^(8-|A|) =  80.5625000000000000 .
      Few subuniverses, (C2.b1.a1) is excluded.

 (C2) a+b=:D<d (in addition to f*g=d; selfdual situation)%
  (C2.b') f+g=:I<i, f+h=i, a*b=:Z>o,  then I+h=i and Z*c=o
      (C2.b1')  d+e=g, then f+e=f+d+e=f+g=I
     (C2.b1.a') D*e=b, then a*e=a*D*e=a*b=Z  
      (C2.b1.a2) a*c=:x>o, then b*c=o, b*x=o, and Z*x=o
|A|=14, A(without commas)={oiabcdefghDIZx}. Constraints:
(always) f*g=d
 a+b=D
  f+g=I f+h=i a*b=Z   I+h=i Z*c=o
   d+e=g   f+e=I
    D*e=b   a*e=Z
     a*c=x   b*c=o b*x=o  Z*x=o
Result for A=(H0/C2.b1.a2):  |Sub(A)| = 4117, that is,
sigma(A) = |Sub(A)|*2^(8-|A|) =  64.3281250000000000 .
      Few subuniverses, (C2.b1.a2) is excluded.
      Thus, (C2.b1.a') is excluded.

 (C2) a+b=:D<d (in addition to f*g=d; selfdual situation)%
  (C2.b') f+g=:I<i, f+h=i, a*b=:Z>o, then I+h=i and Z*c=o
    (C2.b1')  d+e=g, then f+e=f+d+e=f+g=I
     (C2.b1.b')D*e=:B>b and D*h=b,then B*h=b, a*h=a*D*h=a*b=Z
|A|=14, A(without commas)={oiabcdefghDIZB}. Constraints:
(always) f*g=d
 a+b=D
  f+g=I f+h=i a*b=Z   I+h=i Z*c=o
   d+e=g   f+e=I
    D*e=B D*h=b   B*h=b a*h=Z
Result for A=(H0/C2.b1.b'):  |Sub(A)| = 4518, that is,
sigma(A) = |Sub(A)|*2^(8-|A|) =  70.5937500000000000 .
      Few subuniverses, (C2.b1.b') is excluded.

 (C2) a+b=:D<d (in addition to f*g=d; selfdual situation)%
  (C2.b') f+g=:I<i, f+h=i, a*b=:Z>o,  then I+h=i and Z*c=o
    (C2.b1')  d+e=g, then f+e=f+d+e=f+g=I
     (C2.b1.c') D*e=:B>b, D*h=:u>b, and e*h=b, 
         then B*h=b, e*u=b, and B*u=b  
|A|=15, A(without commas)={oiabcdefghDIZBu}. Constraints:
(always) f*g=d
 a+b=D
  f+g=I f+h=i a*b=Z   I+h=i Z*c=o
   d+e=g   f+e=I
    D*e=B D*h=u e*h=b   B*h=b e*u=b B*u=b
Result for A=(H0/C2.b1.c'):  |Sub(A)| = 7340, that is,
sigma(A) = |Sub(A)|*2^(8-|A|) =  57.3437500000000000 .
      Few subuniverses, (C2.b1.c') is excluded.
      Thus, (C2.b1') is excluded.
 Since (C2.b') describes a selfdual situations, the dual of
 (C2.b1') is also excluded, i.e., D*e>b; this will be used 
 in (C2.b2') and (C2.b3') !   

 (C2) a+b=:D<d (in addition to f*g=d; selfdual situation)%
  (C2.b') f+g=:I<i, f+h=i, a*b=:Z>o,  then I+h=i and Z*c=o
    (C2.b2') d+e=:G<g, d+c=g, D*e=:B>b, 
      then  G+c=g, B*h=b, f+c=f+d+c=f+g=I
|A|=15, A(without commas)={oiabcdefghDIZGB}. Constraints:
(always) f*g=d
 a+b=D
  f+g=I f+h=i a*b=Z   I+h=i Z*c=o
   d+e=G d+c=g D*e=B   G+c=g B*h=b f+c=I
Result for A=(H0/C2.b2'):  |Sub(A)| = 8162, that is,
sigma(A) = |Sub(A)|*2^(8-|A|) =  63.7656250000000000 .
      Few subuniverses, (C2.b2') is excluded.

 (C2) a+b=:D<d (in addition to f*g=d; selfdual situation)%
  (C2.b') f+g=:I<i, f+h=i, a*b=:Z>o,  then I+h=i and Z*c=o
      (C2.b3') d+e=:G<g, d+c=:p<g, e+c=g, and  D*e=:B>b, then 
               G+c=g, p+e=g, G+p=g, and B*h=b 
|A|=16, A(without commas)={oiabcdefghDIZGBp}. Constraints:
(always) f*g=d
 a+b=D
  f+g=I f+h=i a*b=Z   I+h=i Z*c=o
   d+e=G d+c=p e+c=g D*e=B    G+c=g p+e=g G+p=g B*h=b
Result for A=(H0/C2.b3'):  |Sub(A)| = 13168, that is,
sigma(A) = |Sub(A)|*2^(8-|A|) =  51.4375000000000000 .
      Few subuniverses, (C2.b3') is excluded.
      Thus, (C2.b') is excluded.      
 Since (C2) is a selfdual situation in which (C2.a') and
 (C2.b') are excluded, their duals are also excluded, i.e.,
  a*b>o and a*c>o; this will be used in (C2.c')!

 (C2) a+b=:D<d (in addition to f*g=d; selfdual situation)%
  (C2.c')f+g=:I<i, f+h=:Q<i, g+h=i, a*b=:Z>o, a*c:=q>o,b*c=o, 
     then I+h=i, g+Q=i, I+Q=i,  Z*c=o, b*q=o, and Z*q=o.
|A|=15, A(without commas)={oiabcdefghDIZQq}. Constraints:
(always) f*g=d
 a+b=D
  f+g=I f+h=Q g+h=i  a*b=Z a*c=q b*c=o   I+h=i g+Q=i I+Q=i
  Z*c=o b*q=o Z*q=o
Result for A=(H0/C2.c'):  |Sub(A)| = 9800, that is,
sigma(A) = |Sub(A)|*2^(8-|A|) =  76.5625000000000000 .
      Few subuniverses, (C2.c') is excluded.
      Thus, (C2) is excluded.
 This proves that H_0 cannot be a subposet of a finite 
 lattice with many subuniverses. Q.e.d.

The computation took 265/1000 seconds.
\end{verbatim}
\normalfont

\clearpage
\def\appendixhead{{Cz\'edli: Eighty-three sublattices / Appendix: $K_5$}}
\markboth\appendixhead\appendixhead
\centerline{\bf{{\LARGE Appendix: $K_5$}}}
\normalfont
\begin{verbatim}
Version of December 31, 2018; reformulated May 7, 2019
We deal with the auxiliary lattice K_5 from the paper
|A|=12, A(without commas)={oiabcdefghjk}. Constraints:
(always) e+f=g e*f=c g+h=k c*d=b a*b=o j*k=i;as in the paper
  (whence) e*d=b a*d=o; since e*d=e*f*d=c*d=b  a*d=a*c*d=a*b=o
  (and) e+h=k j+h=i; since e+h=e+f+h=g+h=k and j+h=j+g+h=j+k=i
Result for A=K_5:  |Sub(A)| = 1558, that is,
sigma(A) = |Sub(A)|*2^(8-|A|) =  97.3750000000000000 .
We are not ready yet.     

|A|=12, A(without commas)={oiabcdefghjk}. Constraints:
(always) e+f=g e*f=c g+h=k c*d=b a*b=o j*k=i;as in the paper
  (whence) e*d=b a*d=o; since e*d=e*f*d=c*d=b  a*d=a*c*d=a*b=o
  (and) e+h=k j+h=i; since e+h=e+f+h=g+h=k and j+h=j+g+h=j+k=i
 (C1) c+d=f g*h=f ; Case (C1); see the paper
    (so) e+d=g e*h=c ; since e+d=e+c+d=e+f=g and e*h=e*g*h=e*f=c
Result for A=K_5:  |Sub(A)| = 1267, that is,
sigma(A) = |Sub(A)|*2^(8-|A|) =  79.1875000000000000 .
This case is excluded.

|A|=13, A(without commas)={oiabcdexyghjk}. Constraints:
(instead of always) g+h=k c*d=b a*b=o j*k=i ;  f is removed
  (whence) e*d=b a*d=o ; since e*d=e*f*d=c*d=b  a*d=a*c*d=a*b=o
  (and) e+h=k j+h=i ; since e+h=e+f+h=g+h=k and j+h=j+g+h=j+k=i
 (C2) c+d=x g*h=y ; c+d=:x and g*h=:y   see the paper
    (so) e+y=g e*x=c ; since e+f=g and e*f=c
Result for A=K_5:  |Sub(A)| = 2578, that is,
sigma(A) = |Sub(A)|*2^(8-|A|) =  80.5625000000000000 .
This case is excluded. All cases have been excluded. Q.e.d  

The computation took 16/1000 seconds.
\end{verbatim}
\normalfont

\clearpage
\def\appendixhead{{Cz\'edli: Eighty-three sublattices / Appendix: Sample input}}
\markboth\appendixhead\appendixhead
\centerline{\bf{{\LARGE Appendix: Sample input}}}
\normalfont
\begin{verbatim}
\P Version of December 30, 2018; reformatted and the last
\P two partial algebras added on May 7, 2019.
% This is a sample data file. Do not change its STRUCTURE!
% In particular, do not change the order of the command, 
% and do not leave a command. 
% For instruction, see the end of this file, and see the comment 
% lines in the data part.
\verbose=false 
% You may also choose: \verbose=true; but the difference is minimal.
\subtrahend-in-exponent=8 
% Small natural number, for another form of |Sub(A)|
\operationsymbols=+* % Only for information; 
% Although this this command is obligatory, the program does NOT
% use any property of the operations; only the syntax is checked.
% Empty lines are allowed but neglected. Everything is 
% case-sensitive. "%" introduces comments that are dropped, but
% <backslach>P introduces comments that are copied into the
% output file. Both forms of comments are optional, of course.
% There is a third form of comment but only in the list of
% constraints; it is carried over the output file and it 
% starts with "\w"; see later for examples.
% And, only in the "constraints part", there is a fourth way:
% namely, we can use a "parenthesized command"; see later.

\beginjob
%
\P    First,  we are dealing with the pentagon lattice.
% Any line containing the <backslash>P command is printed into the 
% output file and to the screen. <backslash>P is suggested only 
% after places where the program automatically prints a line-break.

\name
N_5  %The name of the algebra
% Here a backslash>P  line is not suggested because of the 
% possibility of an ugly line-break, even if such a line would not 
% disturb the computation.
\size
 5  %The size of the algebra
\elements
01abc   % the elements of the algebra; a single character each, 
        % case-sensitive, without spaces
\constraints 
(Case-name1)  a*b=0,  a+b=1 \w Explanation/1 (to be output). 
% In this "constraints" part, use \w, because \P is not allowed here!
% the constraints are separated by spaces or commas, multiple 
% separators are allowed. The "\w "command is allowed only here, 
% in the scope of the \constraints command The parenthesized comment 
% is optional; if it is present then it is printed. Note the syntax:
% either there are no parentheses, or there are exactly one opening 
% one and exactly one closing one in the same line later
 (Subcase-name) c*b=0  \w Explanation/2 to be printed % 
  (Sub-subcase-name)   c+b=1 \w Explanation/3 to be printed % 
% The spaces preceding the opening parenthesis are also printed.
% This feature and the parenthesized part are very appropriate to 
% organize our work.
\endofjob

\P Comments on the result above (free text)
\P    (free text, continued)
 
\beginjob
\name
M_3 %The name of the algebra
\size
 5  %size of the algebra
\elements
01abc
\constraints
(C1) a*b=0 \w %The empty print gives only a colon and a line break.
 (C1a)  a+b=1  \w  but this algebra needs no cases to distinguish! 
   (C1a.1)  b*c=0 \w "C" in parentheses stands for "Case" 
    (C1a.1b) b+c=1 c*a=0,  c+a=1 
\endofjob

\beginjob
\P We are dealing with the 6-element modular lattice of length 2
\name
M_4 %The name of the algebra
\size
 6  %size of the algebra
\elements
01abcd
\constraints
a*b=0 a+b=1 a*c=0 a+c=1 a*d=0 a+d=1 \w No cases are needed.
b*c=0 b+c=1 b*d=0 b+d=1 c*d=0 c+d=1  
\endofjob

\P
\P The next two partial algebras together are to indicate 
\P the difficulty in the paper

\beginjob
\name
without x
\size 
7
\elements
cdefgoi
\constraints 
c+e=g, f*g=d    
\endofjob

\beginjob
\name
with x
\size 
8
\elements
cdefgoix
\constraints 
c+e=x, f*g=d   
\endofjob

\enddata

The rest of the file will not be processed by the program!

                  INSTRUCTION     
%*******************************************************************

In such a file, arbitrarily many partial algebras can be given. 
(Here, we give only five.) Percentage sign (%) means that the rest  
of the line is only a comment, which does not influence the 
computation. Note that if we have a unary operation f and f(x)=y, 
then we have to type the corresponding constraint as  xfx=y .

Notes on the syntax: no line can exceed 250 characters but it is 
suggested that 80 characters  should not be exceeded. The 
constraints are separated by commas of spaces, in a single line 
or in more lines. The elements of the algebra should be denoted 
by letters or decimals; some characters are forbidden, and the 
characters are case-sensitive.

Also, study the comments on the syntax in the comment lines that 
occur among the data above.

Improper data will lead to irregular termination of the program.

                  ALGORITHM
%*******************************************************************


The program lists ALL subsets of the base set and counts those that
are closed with respect to all constraints. No property of the 
operations are taken into account. In fact, even multiple valued 
partial operations are allowed; for example,  xfy=u  and xfy=v 
can both occur among the constraints. Since the running time is 
exponential, the program is not effective for large partial 
algebras. For |A|<26, the program is likely to compute the result
in a reasonable time in any reasonable computer. For |A|>30, this is 
not always so. In addition to |A|, the running time depends also 
on the constraints.

                  AS IT IS

The program is for Windows 10; it is likely to run under Windows 7 
or higher. The program is "as it is", nothing is guaranteed.
\end{verbatim}
\normalfont

\clearpage
\def\appendixhead{{Cz\'edli: Eighty-three sublattices / Appendix: Small Kelly--Rival lattices}}
\markboth\appendixhead\appendixhead
\centerline{\bf{{\LARGE Appendix: Small Kelly--Rival lattices}}}
\normalfont
\begin{verbatim}
Version of December 26, 2018; reformatted May 7, 2019
Lattices from the Kelly-Rival: 
Planar lattices paper and some other lattices

A_0 is the 8-element boolean lattice with edges
 oa ab oc Ai Bi Ci aB aC bA bC cA cB
|A|=8, A(without commas)={oiabcABC}. Constraints:
a+b=C a+c=B a+A=i  b+c=A b+B=i  c+C=i
a*b=o a*c=o a*A=o  b*c=o b*B=o  c*C=o
A*B=c A+B=i  A*C=b A+C=i  B*C=a B+C=i
Result for A=A_0:  |Sub(A)| = 74, that is,
sigma(A) = |Sub(A)|*2^(8-|A|) =  74.0000000000000000 .

B is the lattice with edges 
 oa ab oc od ae be bf bg cf dg ei fi gi
|A|=9, A(without commas)={oiabcdefg}. Constraints:
a*b=o a+b=e a*c=o a+c=i a*d=o a+d=i a*f=o a+f=i a*g=o a+g=i
b*c=o b+c=f b*d=o b+d=g  c*d=o c+d=i c*e=o c+e=i c*g=o c+g=i
d*e=o d+e=i d*f=o d+f=i   e*f=b e+f=i e*g=b e+g=i  f*g=b f+g=i
Result for A=B:  |Sub(A)| = 108, that is,
sigma(A) = |Sub(A)|*2^(8-|A|) =  54.0000000000000000 .

C is the lattice with edges 
 ai bi ci da db eb ec fb gd ge og of
|A|=9, A(without commas)={oiabcdefg}. Constraints:
a*b=d a+b=i a*c=g a+c=i a*e=g a+e=i a*f=o a+f=i  b*c=e b+c=i
c*d=g c+d=i c*f=o c+f=i  d*e=g d+e=b d*f=o d+f=b
e*f=o e+f=b   f*g=o f+g=b
Result for A=C:  |Sub(A)| = 137, that is,
sigma(A) = |Sub(A)|*2^(8-|A|) =  68.5000000000000000 .

D is the lattice with edges 
 oa ob ac ae ad be cf dg ef eg fi gi
|A|=9, A(without commas)={oiabcdefg}. Constraints:
a+b=e a*b=o   b+c=f b*c=o b+d=g b*d=o
c*d=a c+d=i c*e=a c+e=f c*g=a c+g=i
d*e=a d+e=g d*f=a d+f=i   f*g=e f+g=i
Result for A=D:  |Sub(A)| = 152, that is,
sigma(A) = |Sub(A)|*2^(8-|A|) =  76.0000000000000000 .

E_0 is the lattice with edges 
 ai bi ci db ea ed fd fc gb oe of og
|A|=9, A(without commas)={oiabcdefg}. Constraints:
a+b=i a*b=e a+c=i a*c=o a+d=i a*d=e a+f=i a*f=o a+g=i a*g=o
b+c=i b*c=f  c+d=i c*d=f c+e=i c*e=o c*g=o c+g=i  d+g=b d*g=o
e+f=d e*f=o  e+g=b e*g=o  f+g=b f*g=o
Result for A=E_0:  |Sub(A)| = 121, that is,
sigma(A) = |Sub(A)|*2^(8-|A|) =  60.5000000000000000 .

E_1 is the lattice with edges
 ai bi ca da ei fb fc gc gd hd he ja of og oh oj 
|A|=11, A(without commas)={oiabcdefghj}. Constraints:
a*b=f a+b=i a*e=h a+e=i
b*c=f b+c=i b*d=o b+d=i b*e=o b+e=i b*g=o b+g=i
      b*h=o b+h=i b*j=o b+j=i
c*d=g c+d=a c*e=o c+e=i c*h=o c+h=a c*j=o c+j=a
d*e=h d+e=i d*f=o d+f=a d*j=o d+j=a
e*f=o e+f=i e*g=o e+g=i e*j=o e+j=i
f*g=o f+g=c f*h=o f+h=a f*j=o f+j=a
g*h=o g+h=d g*j=o g+j=a  h*j=o h+j=a
Result for A=E_1:  |Sub(A)| = 249, that is,
sigma(A) = |Sub(A)|*2^(8-|A|) =  31.1250000000000000 .

F_0 is the lattice with edges
 ai bi ca da eb ec fe fd gc of og
|A|=9, A(without commas)={oiabcdefg}. Constraints:
a+b=i a*b=e  b+c=i b*c=e b+d=i b*d=f b+g=i b*g=o
c+d=a c*d=f   d+e=a d*e=f  d*g=o d+g=a  e+g=c e*g=o  f+g=c f*g=o
Result for A=F_0:  |Sub(A)| = 166, that is,
sigma(A) = |Sub(A)|*2^(8-|A|) =  83.0000000000000000 .

F_1 is the lattice with edges 
 oa od ab ac ah be bf cf cg dg ei fj gj hj ji
|A|=11, A(without commas)={oiabcdefghj}. Constraints:
a*d=o a+d=g   b*c=a b+c=f b*d=o b+d=j b*g=a b+g=j b*h=a b+h=j
c*d=o c+d=g c*e=a c+e=i c*h=a c+h=j
d*e=o d+e=i d*f=o d+f=j d*h=o d+h=j
e*f=b e+f=i e*g=a e+g=i e*h=a e+h=i e*j=b e+j=i
f*g=c f+g=j f*h=a f+h=j    g*h=a g+h=j
Result for A=F_1:  |Sub(A)| = 329, that is,
sigma(A) = |Sub(A)|*2^(8-|A|) =  41.1250000000000000 .

G_0 is the lattice with edges 
 oa ob ac ad ag bd ce de df eh ej fj gj hi ji
|A|=11, A(without commas)={oiabcdefghj}. Constraints:
a*b=o a+b=d   b*c=o b+c=e b*g=o b+g=j  c*d=a c+d=e c*f=a c+f=j c*g=a c+g=j
d*g=a d+g=j  e*f=d e+f=j e*g=a e+g=j  f*g=a f+g=j f*h=d f+h=i
g*h=a g+h=i   h*j=e h+j=i
Result for A=G_0:  |Sub(A)| = 434, that is,
sigma(A) = |Sub(A)|*2^(8-|A|) =  54.2500000000000000 .

H_0 is the lattice with edges
 oa ob oc ad bd be bh cg df dg eg fi gi hi
|A|=10, A(without commas)={oiabcdefgh}. Constraints:
a*b=o a+b=d a*c=o a+c=g  a*e=o a+e=g a*h=o a+h=i  b*c=o b+c=g
c*d=o c+d=g  c*e=o c+e=g c*f=o c+f=i c*h=o c+h=i
d*e=b d+e=g d*h=b d+h=i   e*f=b e+f=i e*h=b e+h=i
f*g=d f+g=i f*h=b f+h=i   g*h=b g+h=i
Result for A=H_0:  |Sub(A)| = 199, that is,
sigma(A) = |Sub(A)|*2^(8-|A|) =  49.7500000000000000 .

The computation took 62/1000 seconds.
\end{verbatim}
\normalfont

\clearpage
\def\appendixhead{{G\'abor Cz\'edli: Eighty-three sublattices and planarity}}
\markboth\appendixhead\appendixhead
\normalfont
\normalsize
\end{document}